\documentclass[english]{article}
\usepackage[T1]{fontenc}
\usepackage[utf8]{inputenc}
\usepackage{babel}
\usepackage{bbold}
\usepackage{multicol}
\usepackage{mathtools}
\usepackage{amsthm}
\usepackage{stmaryrd}
\usepackage{amsmath}
\usepackage{enumitem}
\usepackage[mathscr]{euscript}
\usepackage{amssymb}
\usepackage{multicol}
\usepackage[a4paper,left=3.5cm,right=3.5cm,top=3cm,bottom=3cm]{geometry}
\usepackage{tcolorbox}
\usepackage{dsfont}
\usepackage{fancyheadings}
\newtheorem{theorem}{Theorem}[section]
\newtheorem{crl}[theorem]{Corollary}
\newtheorem{prop}[theorem]{Proposition}
\newtheorem{lm}[theorem]{Lemma}
\newtheorem{defi}[theorem]{Definition}
\newtheorem{rmq}[theorem]{Remark}
\newtheorem*{nota}{Notation}
\def\restriction#1#2{\mathchoice
	{\setbox1\hbox{${\displaystyle #1}_{\scriptstyle #2}$}
		\restrictionaux{#1}{#2}}
	{\setbox1\hbox{${\textstyle #1}_{\scriptstyle #2}$}
		\restrictionaux{#1}{#2}}
	{\setbox1\hbox{${\scriptstyle #1}_{\scriptscriptstyle #2}$}
		\restrictionaux{#1}{#2}}
	{\setbox1\hbox{${\scriptscriptstyle #1}_{\scriptscriptstyle #2}$}
		\restrictionaux{#1}{#2}}}
\def\restrictionaux#1#2{{#1\,\smash{\vrule height .8\ht1 depth .85\dp1}}_{\,#2}}

\newcommand{\rr}{\mathbb{R}}
\newcommand{\cc}{\mathbb{C}}

\newcommand{\spn}{\text{Span}}
\newcommand{\nn}{\mathbb{N}}
\newcommand{\zz}{\mathbb{Z}}
\newcommand{\dt}{\mathrm{d}t}
\newcommand{\dx}{\mathrm{d}x}
\newcommand{\ds}{\mathrm{d}s}
\newcommand{\ph}{\varphi}
\newcommand{\ep}{\varepsilon}
\newcommand{\B}{\mathcal{B}}
\newcommand{\ev}{\textsc{e}}
\newcommand{\dd}{\,\mathrm{d}}

\DeclareMathOperator{\Br}{Br}
\DeclareMathOperator{\ad}{ad}
\numberwithin{equation}{section}
\title{Small-Time Local Controllability of the multi-input bilinear Schrödinger equation thanks to a \\quadratic term}
\author{Théo Gherdaoui\thanks{Univ Rennes, CNRS, IRMAR - UMR 6625, F-35000 Rennes, France}}
\usepackage[pdfborder={0 0 0}]{hyperref}
	
\begin{document}
	\maketitle
	\begin{abstract}
		The goal of this article is to contribute to a better understanding of the relations between 
		the exact controllability of nonlinear PDEs and the control theory for ODEs based on Lie brackets,
		through a study of the Schrödinger PDE with bilinear control.
		
		We focus on the small-time local controllability (STLC) around an equilibrium, when the linearized system is not controllable.
		We study the second-order term in the Taylor expansion of the state, with respect to the control.
		
		For scalar-input ODEs, quadratic terms never recover controllability: they induce signed drifts in the dynamics (see \cite{beauchard2017quadratic}).
		Thus proving STLC requires to go at least to the third order. 
		Similar results were proved for the bilinear Schrödinger PDE with scalar-input controls in  \cite{bournissou2022smalltime}.
		
		In this article, we study the case of {multi-input systems}.
		We clarify among the quadratic Lie brackets, those that allow to recover STLC: they are bilinear with respect to two different controls.
		For ODEs, our result is a consequence of Sussmann's sufficient condition $S(\theta)$ (when focused on quadratic terms),
		but we propose a new proof, designed to prepare an easier transfer to PDEs.
		This proof relies on a representation formula of the state inspired by the Magnus formula.
		By adapting it, we prove a new STLC result for the bilinear Schrödinger PDE.\par
		\vspace{0.25 cm}
		\noindent\textbf{Keywords:} Bilinear Schrödinger equation, infinite-dimensional systems, small-time local exact controllability, power series expansion.
	\end{abstract}
	\tableofcontents
	\section{Introduction}
	\subsection{Model and problem}
	In this article, we study the following Schrödinger equation
	\begin{equation} \label{schr}
		\left\lbrace \begin{array}{lr}
			i \partial_t \psi(t,x)=- \partial_x^2\psi(t,x)-\left(u(t) \mu_1(x) + v(t) \mu_2(x)\right)\psi(t,x), &t\in (0,T),x\in(0,1), \\
			\psi(t,0)=\psi(t,1)=0,&t \in (0,T), \\
			\psi(0,x)=\psi_0(x),&x\in(0,1),
		\end{array}\right.
	\end{equation}
	where $\mu_1,\mu_2:(0,1) \to \rr$, $T>0$, $u,v:(0,T)\to \rr$, $\psi:(0,T)\times (0,1) \to \cc$ and $\|\psi_0\|_{L^2}=1$.
	When well-defined, the solution is denoted $\psi$. When required, we will write $\psi(\cdot;(u,v),\psi_0)$ to refer to this solution to emphasize its dependence on the different parameters.
	\medskip
	
	This equation describes the evolution of the wave function $\psi$ of a quantum particle
	in a $1$D infinite square potential well $(0,1)$, subjected to two electric fields with amplitudes $u(t)$ and $v(t)$.
	The functions $\mu_1$ and $\mu_2$, called "dipolar moments", model the interaction between the particle's wave function $\psi$ and the two electric fields $u,v$.
	\medskip
	
	This is a multi-input nonlinear control system: 
	\begin{itemize}\vspace{-0.2 cm}
		\item[-] the state is the wave function $\psi: (0,T) \to \mathcal{S}$, where $\mathcal{S}$ denotes the $L^2(0,1)$ sphere,  \vspace{-0.2 cm}
		\item[-] the controls are $u,v:(0,T) \to \rr$, they act bilinearly on the state, through the term $(u(t) \mu_1(x) + v(t) \mu_2(x))\psi(t,x)$,
		this is a reason why we speak of "bilinear Schrödinger equation".
	\end{itemize}
	\medskip
	The ground state is the particular trajectory $\psi_1(t,x):=\varphi_1(x) e^{-i \lambda_1 t}$, with $(u,v)=0$ where $\varphi_1(x):=\sqrt{2} \sin( \pi x)$ and $\lambda_1:=\pi^2$.
	We are interested in the local controllability of the nonlinear system \eqref{schr} around the ground state,
	i.e.\ realizing small motions around the ground state with small controls:
	given a maximal size $\varepsilon>0$ for the controls, 
	does there exist $T,\delta>0$ such that,
	for any target $\psi_f$ close enough to the ground state, i.e. $\|\psi_f - \psi_1(T)\| < \delta$,
	there exist controls $u,v:(0,T) \to \rr$ such that $\|(u,v)\| \leqslant\varepsilon$ and $\psi(T;(u,v),\varphi_1)=\psi_f$.
	The different norms used are of great importance and will be precised later.
	\medskip
	
	More precisely, we are interested in the \textbf{small-time} local controllability around the ground state
	i.e.\ realizing, in arbitrarily small times $T>0$, small motions around the ground state with small controls:
	given $T, \varepsilon>0$,
	does there exist $\delta>0$ such that,
	for any target $\psi_f$ such that $\|\psi_f - \psi_1(T)\| < \delta$,
	there exist controls $u,v:(0,T) \to \rr$ such that $\|(u,v)\|\leqslant \varepsilon$ and $\psi(T;(u,v),\varphi_1)=\psi_f$.
	This corresponds to the local surjectivity around $0$ of the nonlinear end-point map 
	$\Theta_T:(u,v) \mapsto \psi(T;(u,v),\varphi_1)$.
	
	Any positive answer to the STLC problem may be thought of as a nonlinear local open mapping theorem, 
	which underlines the deepness and intricacy of this problem, 
	when the inverse mapping theorem (or linear test, see  \cite[Section $3.1$]{coronbook}) cannot be used.
	
	Small-time controllability has particularly relevant physical implications, both from a fundamental viewpoint and for technological applications. 
	Indeed, quantum systems, once engineered, suffer of very short lifespan before decaying (\textit{e.g.}\ through spontaneous photon emissions) and losing their non-classical properties (such as superposition). Hence, the capability of controlling them in a minimal time is also an important challenge in physics.
	\subsection{The controllable linearized case: constrained functional framework}
	Under appropriate assumptions on the dipolar moments $\mu_1, \mu_2$ the small-time local controllability of \eqref{schr} can be proved by applying the linear test. 
	One then obtains the following statement, that can easily be proved by adapting with two controls the strategy developed in \cite{beauchard2010local,bournissou2021local} with one control.
	This statement requires the introduction of the following notations: we define the operator \begin{equation}\label{laplacien}A:=-\partial^2_x \text{ with  domain } D(A):=H^2(0,1)\cap H^1_0(0,1).\end{equation} Except explicit precision, we will work with complex valued functions. We consider the space, $L^2(0,1)$, endowed with the classical scalar product
	$$\forall f,g\in L^2(0,1), \quad \left\langle f,g\right\rangle:=\int_0^1f(x)\overline{g(x)}\dx.$$

	According to classical functional analysis, the eigenvalues and eigenvectors are
	\begin{equation}\label{vp}\lambda_j:=(j\pi)^2,\qquad\ph_j:=\sqrt{2}\sin(j\pi\cdot),\qquad j\geqslant1.\end{equation}
	
	The family $(\ph_j)_{j\geqslant1}$ is an orthonormal basis of $L^2(0,1)$. We define
	\begin{equation}\label{solutionlibre}\psi_j(t,x):=\ph_j(x)e^{-i\lambda_jt},\qquad (t,x)\in (0,T)\times (0,1), \qquad j\geqslant1.\end{equation}
	
	These functions are the solutions to the free (with controls $(u,v)\equiv 0$) Schrödinger equation \eqref{schr}, with initial data $\ph_j$ at time $t=0$. They are called eigenstates. When $j=1$, $\psi_1$ is called the fundamental state, or ground state. Finally, we define the spaces $H^s_{(0)}(0,1)=D(A^{s/2})$, equipped with the norm
	$$\left\|\ph\right\|_{H^s_{(0)}}:=\left(\sum_{j=1}^{+\infty}|j^s\langle\ph,\ph_j\rangle|^2\right)^{1/2}.$$
	We consider, for $k\in\nn$, $H^k((0,T),\rr)$, the real Sobolev space, equipped with the usual $H^k(0,T)$ norm and $H_0^k(0,T)$ the adherence of $\mathcal{C}^{\infty}_c(0,T)$ with this norm.
	\begin{theorem}\label{lintest}
		Let $p,m \in \nn$ and $\mu_1,\mu_2 \in H^{2(p+m)+3}((0,1),\rr)$ be such that $\mu_{\ell}^{(2k+1)}(0)=\mu_{\ell}^{(2k+1)}(1)=0$ for $k=0,\cdots,p-1$, $\ell\in\{1,2\}$ and
		\begin{equation} \label{hyp_linearisé_controlable}
			\exists C>0, \quad \forall j \in \nn^*, \quad
			|\langle \mu_1 \varphi_1 , \varphi_j \rangle| +  |\langle \mu_2 \varphi_1 , \varphi_j \rangle| \geqslant \frac{C}{j^{2p+3}}.
		\end{equation}
		Then, the bilinear Schrödinger equation \eqref{schr} is $H^m_0((0,T),\rr)$-small-time locally controllable in $H^{2(p+m)+3}_{(0)}(0,1)$:
		for every $T, \varepsilon>0$, there exists $\delta>0$ such that
		for every $\psi_f \in \mathcal{S} \cap H^{2(p+m)+3}_{(0)}(0,1)$ with $\|\psi_f - \psi_1(T)\|_{H^{2(p+m)+3}}<\delta$,
		there exist $u,v \in H^m_0((0,T),\rr)$ such that $\|(u,v)\|_{H^m}\leqslant\varepsilon$ and
		$\psi(T;(u,v),\varphi_1)=\psi_f$.
	\end{theorem} 
	In this statement, there are two parameters:
	\begin{enumerate}\vspace{-0.2 cm}
		\item[-] $p$ is the number of odd derivatives of $\mu_{\ell}$ that vanish at the boundary, $\ell\in\{1,2\}$,\vspace{-0.2 cm}
		\item[-] $m$ corresponds to the regularity expected on the controls: $u,v \in H^m_0((0,T),\rr)$.
	\end{enumerate}
	The property
	\begin{equation}\label{munes}\forall j \in \nn^*, \qquad |\langle \mu_1 \varphi_1 , \varphi_j \rangle| +  |\langle \mu_2 \varphi_1 , \varphi_j \rangle| \neq 0\end{equation}
	is necessary for the controllability of the linearized system around the ground state.
	The stronger assumption (\ref{hyp_linearisé_controlable}) guarantees that the linearized system is $H^m_0((0,T),\rr)$-small-time controllable in
	$H^{2(p+m)+3}_{(0)}(0,1)$.
	Then, a smoothing effect allows to prove that the end-point map is of class $\mathcal{C}^1$ between the following spaces
	$$\Theta_T:(u,v)\in H^m_0((0,T),\rr)^2 \mapsto \psi(T;(u,v),\varphi_1) \in H^{2(m+p)+3}_{(0)}(0,1)\cap \mathcal{S}.$$
	Thus, by applying the inverse mapping theorem, we obtain the local controllability of the nonlinear system \eqref{schr}.
	
	The choice of functional spaces is determined by the controllability of the linearized system:
	once $p, m$ are fixed and \eqref{hyp_linearisé_controlable} is verified, then we must work with $\psi(T) \in H^{2(p+m)+3}_{(0)}(0,1)$.
	
	Note that the assumption \eqref{hyp_linearisé_controlable} prevents $\mu_{\ell}^{(2p+1)}$ from vanishing at the boundary because, if $\mu_{\ell}$ satisfies the hypotheses of Theorem \ref{lintest}, then integrations by parts give: for all $k,j\in\nn^*$, distinct, for all $\ell\in\{1,2\}$,
	\begin{equation}\label{riemann:lebesgue}		\langle\mu_{\ell}\ph_j,\ph_k\rangle=\frac{4(-1)^p(p+1)k}{j^{2p+3}\pi^{2p+2}}\left((-1)^{j+k}\mu_{\ell}^{(2p+1)}(1)-\mu_{\ell}^{(2p+1)}(0)\right)+\underset{j\to+\infty}{o}\left(\frac{1}{j^{2p+3}}\right).
	\end{equation}
	\subsection{Our case study and its interpretation in terms of Lie brackets}\label{dipolar}

	In this article, we study the case where the linearized system around the ground state is not controllable,
	so we consider an  integer $K \geqslant2$ and we assume 
	$$\mathbf{(H)_{lin,K,1}}: \langle\mu_1\ph_1,\ph_K\rangle=\langle\mu_2\ph_1,\ph_K\rangle=0.$$
	In this case, the linearized system is not controllable: it misses one complex direction $\langle \psi(t), \varphi_K \rangle \in \cc$. In other words, the first-order term in the Taylor expansion at $(0,0)$ of the map $(u,v) \mapsto \langle \psi(t;(u,v),\varphi_1),\varphi_K \rangle$ vanishes.
	We assume that this is the only direction lost, i.e. all other directions are controllable on the linearized system, with the usual functional framework, 
	so we consider integers $p, m \in \nn$ and we assume
	$$\mathbf{(H)_{reg}}: \mu_{\ell}\in H^{2(p+m)+3}((0,1),\rr), \text{ with } \restriction{\mu_{\ell}^{(2k+1)}}{\{0,1\}}=0\text{ for }0\leqslant k\leqslant p-1, \ \ell\in\{1,2\}.$$
	$$\mathbf{(H)_{lin,K,2}}: \text{there exists $C>0$ such that }\forall j\in\nn^*\setminus\left\lbrace K\right\rbrace, \left\|\left(\displaystyle\left\langle\mu_{\ell}\ph_1,\ph_j\right\rangle\right)_{1\leqslant \ell\leqslant 2}\right\|\geqslant\displaystyle\frac{C}{j^{2p+3}}.$$
	To control the complex direction $\langle \psi(t),\varphi_K \rangle$ of the nonlinear system \eqref{schr}, 
	we will use a power series expansion of the order $2$, i.e.\
	the second-order term in the Taylor expansion at $(0,0)$  of the map $(u,v) \mapsto \langle \psi(t;(u,v),\ph_1), \varphi_K \rangle$.
	At this step of the proof,
	we consider an integer $n \in \nn^*$ and we assume that the following assumptions hold:\\\\
	$\mathbf{(H)_{quad,K,1}}$: $\forall k\in\llbracket 1,\lfloor\frac{n+1}{2}\rfloor\rrbracket$, $$A_k^1:=(-1)^{k-1}\sum_{j=1}^{+\infty}\left(\lambda_j-\frac{\lambda_1+\lambda_K}{2}\right)(\lambda_K-\lambda_j)^{k-1}(\lambda_j-\lambda_1)^{k-1}c_j =0,$$
	$\mathbf{(H)_{quad,K,2}}$: $\forall k\in\llbracket 1,\lfloor\frac{n+1}{2}\rfloor\rrbracket$, $$A_k^2:=(-1)^{k-1}\sum_{j=1}^{+\infty}\left(\lambda_j-\frac{\lambda_1+\lambda_K}{2}\right)(\lambda_K-\lambda_j)^{k-1}(\lambda_j-\lambda_1)^{k-1}\tilde{c_j} =0,$$
	$\mathbf{(H)_{quad,K,3}}$: $\forall k\in\llbracket 1,n-1\rrbracket,$ $$\gamma_k:= \displaystyle\sum_{j=1}^{+\infty}\left((\lambda_K-\lambda_j)^{\lfloor\frac{k+1}{2}\rfloor}(\lambda_j-\lambda_1)^{\lfloor\frac{k}{2}\rfloor}d_j-(\lambda_K-\lambda_j)^{\lfloor\frac{k}{2}\rfloor}(\lambda_j-\lambda_1)^{\lfloor\frac{k+1}{2}\rfloor}\tilde{d_j}\right)=0,$$
	$\mathbf{(H)_{quad,K,4}}$:
	$$\gamma_n:= \displaystyle\sum_{j=1}^{+\infty}\left((\lambda_K-\lambda_j)^{\lfloor\frac{n+1}{2}\rfloor}(\lambda_j-\lambda_1)^{\lfloor\frac{n}{2}\rfloor}d_j-(\lambda_K-\lambda_j)^{\lfloor\frac{n}{2}\rfloor}(\lambda_j-\lambda_1)^{\lfloor\frac{n+1}{2}\rfloor}\tilde{d_j}\right)\neq0,$$
	with the notations: for all $j\geqslant 1$,
	\begin{enumerate}
		\begin{multicols}{2}
			\item[] $c_j:=\left\langle\mu_1\ph_1,\ph_j\right\rangle\left\langle \mu_1\ph_j,\ph_K\right\rangle,$
			\item[] $\tilde{c_j}:=\left\langle\mu_2\ph_1,\ph_j\right\rangle\left\langle\mu_2\ph_j,\ph_K\right\rangle,$ 
			\item[]$d_j:=\left\langle\mu_2\ph_1,\ph_j\right\rangle\left\langle\mu_1\ph_j,\ph_K\right\rangle,$ \item[]$\tilde{d_j}:=\left\langle\mu_1\ph_1,\ph_j\right\rangle\left\langle\mu_2\ph_j,\ph_K\right\rangle.$
		\end{multicols}
	\end{enumerate} 
	\begin{rmq}\label{asymptotique}
		The asymptotic behavior \eqref{riemann:lebesgue}  and the hypothesis $n\leqslant 2p+2$ ensure that all the series written in $\mathbf{(H)_{quad,K,1}}$, $\mathbf{(H)_{quad,K,2}}$, $\mathbf{(H)_{quad,K,3}}$ and $\mathbf{(H)_{quad,K,4}}$ converge absolutely. Indeed, for all $\eta\in\nn$, 
		$$\sum_{j=1}^{+\infty}|c_j|j^{4\eta},\sum_{j=1}^{+\infty}|\tilde{c_j}|j^{4\eta},\sum_{j=1}^{+\infty}|d_j|j^{4\eta},\sum_{j=1}^{+\infty}|\tilde{d_j}|j^{4\eta}\leqslant C\sum_{j=1}^{+\infty}\frac{1}{j^{4p+6-4\eta}},$$ which converges if $p+1\geqslant \eta$. In particular, if $\lfloor\frac{n}{2}\rfloor\leqslant p$,  all the series defined previously and  $A_{\lfloor\frac{n}{2}\rfloor+1}^1$, $A_{\lfloor\frac{n}{2}\rfloor+1}^2$, $\gamma_{n+1}$ converge also (technical assumption). 
	\end{rmq}
	The existence of such functions $\mu_1, \mu_2$ is proved in Appendix \ref{existe}.
	
	For $A$ and $B$ two operators, we define by induction on $k\in\nn$ the operator $\underline{\ad}^k_A(B)$ as: $\underline{\ad}_A^0(B)=B$ and $\underline{\ad}_A^{k+1}(B)=\underline{\ad}^k_A(B)A-A\underline{\ad}^k_A(B)$. Under appropriate relations between the parameter $n$ and $p$, 
	the assumptions $\mathbf{(H)_{quad,K,1}}$, $\mathbf{(H)_{quad,K,2}}$, $\mathbf{(H)_{quad,K,3}}$ and  $\mathbf{(H)_{quad,K,4}}$ can be interpreted in terms of Lie brackets. More precisely,
	\begin{equation}\label{sommecrochet1}\forall k= 1,\cdots,\left\lfloor\frac{n+1}{2}\right\rfloor, \qquad 2A_k^1=(-1)^k\left\langle[\underline{\ad}_{A}^{k-1}(B_1),\underline{\ad}_{A}^k(B_1)]\ph_1,\ph_K\right\rangle,\end{equation}
	\begin{equation}\label{sommecrochet2}\forall k= 1,\cdots,\left\lfloor\frac{n+1}{2}\right\rfloor, \qquad 2A_k^2=(-1)^k\left\langle[\underline{\ad}_{A}^{k-1}(B_2),\underline{\ad}_{A}^k(B_2)]\ph_1,\ph_K\right\rangle,\end{equation}
	\begin{equation}\label{sommecrochet3}\forall k= 0,\cdots,n, \qquad \gamma_k=(-1)^k\left\langle[\underline{\ad}_{A}^{\lfloor\frac{k+1}{2}\rfloor}(B_1),\underline{\ad}_{A}^{\lfloor\frac{k}{2}\rfloor}(B_2)]\ph_1,\ph_K\right\rangle,\end{equation} where $A$ is defined in \eqref{laplacien} and $B_{\ell}$ is the multiplication operator by $\mu_{\ell}$ in $L^2(0,1)$.
	We refer to Propositions  \ref{crochetdiminf1}, \ref{crochetdiminf2}, \ref{crochet} and \ref{crochet3} in appendix for a precise proof.
	
	For the single-input bilinear Schrödinger equation
	\begin{equation} \label{Schro_single}
		\left\lbrace \begin{array}{lr}
			i \partial_t \psi(t,x)=- \partial_x^2\psi(t,x) -u(t) \mu_{\ell}(x)\psi(t,x),  & (t,x) \in (0,T)\times(0,1), \\
			\psi(t,0)=\psi(t,1)=0, & t \in (0,T), \\
			\psi(0,x)=\psi_0(x), & x \in (0,1), 
		\end{array}\right.
	\end{equation}
	the Lie bracket $2 A_k^{\ell}$ is known to be an obstruction to small-time local controllability, in an appropriate functional setting (see \cite{bournissou2021quadratic}). The Lie brackets $(\gamma_k)_{1\leqslant k\leqslant n}$ are specific to multi-input system \eqref{schr}, our positive controllability
	result relies on $\gamma_n$. We need the following definition: 	for $T>0$ and $u\in L^1((0,T),\rr)$, one defines the iterated primitives $u_n$ vanishing at $t=0$ by induction as
	$$u_0=u, \qquad  \text{and} \qquad \forall n\in\nn, \ \forall t\in [0,T], \quad u_{n+1}(t)=\int_0^tu_n(s)\ds.$$
	
	Then, heuristically, it is as if the following terms were included in the quadratic development of the solution
	$$-i\sum_{k=1}^{\lfloor\frac{n+1}{2}\rfloor}A_{k}^1 \int_0^T u_k^2(t)\dt -i \sum_{k=1}^{\lfloor\frac{n+1}{2}\rfloor} A_{k}^2\int_0^Tv_k^2(t)\dt + \sum_{k=1}^{n} i^k\gamma_k \int_0^T u_{\lfloor\frac{k}{2}\rfloor+1}(t)v_{\lfloor\frac{k+1}{2}\rfloor}(t)\dt .$$
	The cancellation assumptions $\mathbf{(H)_{quad,K,1}}$, $\mathbf{(H)_{quad,K,2}}$, $\mathbf{(H)_{quad,K,3}}$ reduce this sum to the simpler expression
	\begin{equation} \label{DL_noyau}
		i^n\gamma_n \int_0^T u_{\lfloor\frac n2\rfloor+1}(t)v_{\lfloor\frac{n+1}{2}\rfloor}(t)\dt,
	\end{equation}
	and they allow use to take advantage of the term $\displaystyle\int_0^T u_{\lfloor\frac n2\rfloor+1}(t) v_{\lfloor\frac{n+1}{2}\rfloor}(t)\dt$. For this strategy to work, it is thus necessary to consider $K \geqslant 2$. Indeed, the assumption $\mathbf{(H)_{lin,1,2}}$ leads to $\max(A_k^1,A_k^2)>0$ for every $k$. Thus, at least one of the assumptions $\mathbf{(H)_{quad,1,1}}, \mathbf{(H)_{quad,1,2}}$ is not verified.
	
	\medskip
Unlike the single-input case \cite{bournissou2021quadratic}, the quadratic term alone can recover small-time local controllability in the multi-input case.
	\subsection{Main result and proof strategy}

	Our main result is the following one, i.e.\ the same conclusion as in Theorem \ref{lintest}, when replacing 
	the assumption \eqref{hyp_linearisé_controlable} related to the controllability of the linearized system,
	by the assumptions $\mathbf{(H)_{lin,K,2}}$,$\mathbf{(H)_{quad,K,1}}$,$\mathbf{(H)_{quad,K,2}}$,$\mathbf{(H)_{quad,K,3}}$ and $\mathbf{(H)_{quad,K,4}}$ related to a power series expansion of order two.
	\begin{theorem}\label{maintheorem}
		Let $p,m,K\in \nn$, $n\in\nn^*$ with $K \geqslant 2$ and $\lfloor \frac{n}{2} \rfloor \leqslant p$. Assume that
		$\mu_1, \mu_2$ satisfy $\mathbf{(H)_{reg}}$, $\mathbf{(H)_{lin,K,1}}$, $\mathbf{(H)_{lin,K,2}}$, $\mathbf{(H)_{quad,K,1}}$, $\mathbf{(H)_{quad,K,2}}$, $\mathbf{(H)_{quad,K,3}}$, $\mathbf{(H)_{quad,K,4}}$ and $\langle \mu_1 \varphi_1, \varphi_1 \rangle \neq 0$.
		Then the bilinear Schrödinger equation \eqref{schr} is $H^m_0((0,T),\rr)$-small-time locally controllable in $H^{2(p+m)+3}_{(0)}(0,1)$.
	\end{theorem}
	\begin{rmq}
		In particular, if $\mu_1$ and $\mu_2$ are $\mathcal{C}^{\infty}([0,1],\rr)$, then, for all $m\in\nn$, the bilinear Schrödinger equation \eqref{schr} is $H^m_0$-STLC, with targets in $H^{2(p+m)+3}_{(0)}(0,1)$, when the condition $\mathbf{(H)_{reg}}$ is satisfied.
	\end{rmq}
	The assumption $\mathbf{(H)_{lin,K,2}}$ and the linear test prove that 
	$\mathbb{P}_{\mathcal{H}} \psi(T)$ is $H^m_0((0,T),\rr)-$ small-time locally controllable in $H^{2(p+m)+3}_{(0)}(0,1)$
	(see Theorem \ref{control-projection}), where
	\begin{equation}\label{h}\mathcal{H}:=\overline{\spn_{\cc}\left(\ph_j, \ j\in\nn^*\setminus\left\lbrace K\right\rbrace\right)},\end{equation} and $\mathbb{P}_{\mathcal{H}}$ is  the orthogonal projection on $\mathcal{H}$, which is given by
	$$\mathbb{P}_{\mathcal{H}}:\begin{matrix} L^2(0,1)&\to&\mathcal{H}\\\psi&\mapsto&\psi-\langle\psi,\ph_K\rangle\ph_K\end{matrix}.$$
	Then, the proof of Theorem \ref{maintheorem} is divided in two stages: given a time $T>0$, a target $\psi_f$  and $z \in \cc$,
	\begin{enumerate}\vspace{-0.15 cm}
		\item[-] we find $T_1<T$ and controls $(u_z,v_z)$ such that $\langle \psi(T_1;(u_z,v_z),\varphi_1) , \psi_K(T_1) \rangle$ is close to $z$,\vspace{-0.2 cm}
		\item[-] we apply on $(T_1, T)$ controls $(\tilde{u}_z,\tilde{v}_z)$ in projection such that on has $$\mathbb{P}_{\mathcal{H}} \psi(T;(\tilde{u}_z,\tilde{v}_z),\psi(T_1;(u_z,v_z),\ph_1)) = \mathbb{P}_{\mathcal{H}} \psi_f.$$ Then, we prove that, despite the second step, $\langle \psi(T;(\tilde{u}_z,\tilde{v}_z),\psi(T_1;(u_z,v_z),\ph_1)) , \psi_K(T) \rangle$ stays close to $z$, 
		so that a fixed point argument allows to find $z=z(\psi_f)$ such that $\langle \psi(T;(\tilde{u},\tilde{v}),\psi(T_1;(u,v),\ph_1)) , \psi_K(T) \rangle = \langle \psi_f , \psi_K(T) \rangle.$
	\end{enumerate}
	Finally $\psi(T;(\tilde{u},\tilde{v}),\psi(T_1;(u,v),\ph_1))=\psi_f$. We recall that $\psi_K$ is defined in \eqref{solutionlibre}.
	\medskip
	
	In step $1$, we examine the quadratic expansion of $(u,v) \mapsto \langle \psi(T;(u,v),\varphi_1), \psi_K(t) \rangle$. It contains terms of the form (\ref{DL_noyau}).
	
	\subsection{State of the art}
	\subsubsection{Topological obstructions to exact controllability}
	In \cite{doi:10.1137/0320042}, Ball, Marsden and Slemrod proved obstructions to local exact controllability of linear PDEs with bilinear controls.
	For instance, if the multiplicative operators $\mu_{\ell}$ are bounded on $H^s_{(0)}(0,1)$, 
	then system \eqref{schr} is not exactly controllable in $\mathcal{S}\cap H^s_{(0)}(0,1)$, with controls 
	$u,v \in L^p_{\text{loc}}(\rr,\rr)$ and $p>1$.
	The fundamental reason behind is that, under these assumptions, the reachable set has empty interior in $H^{s}_{(0)}(0,1)$.
	The case of $L^1_{\text{loc}}$-controls ($p=1$) was incorporated in \cite{BOUSSAID2020108412} and extensions to nonlinear equations were proved in \cite{doi:10.1137/18M1215207,chambrion:hal-01901819}.
	After the important work \cite{doi:10.1137/0320042}, different notions of controllability were studied for system \eqref{Schro_single} such as
	\begin{itemize}\vspace{-0.2 cm}
		\item[-] \textit{exact controllability} in more regular spaces, on which  $\mu_{\ell}$ do not define bounded operators:
		note that in Theorem \ref{lintest}, for $m=0$, 
	 $\mu_{\ell}$ do not define bounded operators on $H^{2p+3}_{(0)}(0,1)$ because
		$\restriction{\mu_{\ell}^{(2p+1)}}{\{0,1\}} \neq 0$ (see \eqref{riemann:lebesgue}, because of assumption \eqref{hyp_linearisé_controlable}),\vspace{-0.2 cm}
		\item[-] \textit{approximate controllability}.
	\end{itemize}
	\subsubsection{Exact controllability in more regular spaces, by linear test} 
	For system \eqref{Schro_single}, local exact controllability was first proved in \cite{BEAUCHARD2005851,BEAUCHARD2006328} with Nash-Moser techniques, 
	to deal with an apparent derivative loss problem and then in \cite{beauchard2010local} with a classical inverse mapping theorem, 
	thanks to a regularizing effect. By grafting other ingredients onto this core strategy, global (resp.\ local) exact controllability in regular spaces was proved for different models
	in \cite{morancey2013globalexactcontrollability1d,article} (resp.\ \cite{bournissou2021local}). 
	\subsubsection{Single-input systems: quadratic obstructions to STLC} 
	In \cite{CORON2006103}, Coron denied the $L^{\infty}$-small-time local controllability result for system \eqref{Schro_single}, with a particular dipolar moment $\mu_{\ell}$, thanks to a drift. In \cite{beauchard2013local}, Beauchard and Morancey gave general assumptions on $\mu_{\ell}$ to deny $L^{\infty}-$STLC of \eqref{Schro_single}. In \cite{bournissou2021quadratic}, Bournissou proved that this drift also occurs with small control in $W^{-1,\infty}$. Quadratic terms have also been used to create obstructions to the controllability of other single-input systems, for example in \cite{beauchard2017quadratic} for ODEs, in \cite{Fr_d_ric_Marbach_2018} for the Burgers' equation, \cite{BEAUCHARD202022} for the heat equation, \cite{coron2020smalltime} for KdV and \cite{nguyen2023localcontrollabilitykortewegdevries} for KdV-Neumann. \\
	
	With the exception of a non-physical PDE designed for - see \cite[Section $5$ and $6$]{BEAUCHARD202022},
	for single-input systems, the quadratic terms generally do not recover small-time controllability.
	\subsubsection{Single-input systems: power series expansion of order 2 or 3} 
	Nevertheless, a power series expansion of order $2$ allows us to recover a direction lost \textbf{in large time}: this strategy is used in \cite{beauchard2013local} for \eqref{Schro_single}. This method is also used for other equations, such as KdV, in \cite{Cerpa2009}. If the order $2$ vanishes, a power series of order $3$ can be used to recover the STLC, for example in \cite{Coron2004}, for KdV. If the order $2$ doesn't cancel out, but the term of order $3$ is strong enough, this expansion can also give the STLC (see \cite{bournissou2022smalltime}, for equation \eqref{Schro_single}).
	\subsubsection{Large time and small-time approximate controllability} 
	The first results of global approximate controllability of bilinear Schrödinger equations 
	were obtained in large time (see \cite{Boscain_2012,Chambrion2008ControllabilityOT,ervedozapuel,nersesyan,Sigalotti2009GenericCP}).
	For particular systems, a large time is indeed necessary for the approximate controllability (see \cite{beauchard2014minimaltimebilinearcontrol,articletkc}).
	\\
	Small-time approximate controllability between eigenstates for Schrödinger equations on the torus is proved by Duca and Nersesyan in \cite{duca2021bilinearcontrolgrowthsobolev},
	by means of an infinite-dimensional geometric control approach (saturation argument).
	Related results have been subsequently established in \cite{boscain:hal-04496433,chambrion2022smalltime,duca2024smalltimecontrollabilitynonlinearschrodinger}.
	\subsubsection{Multi-input ODE: Lie brackets technics for a multi-input ODE} 
	In Section $2$, we propose a toy model: a finite-dimensional control-affine system. Thus, we give a new proof of a particular case of the Sussmann's ${S}(\theta)$ condition (see \cite{doi:10.1137/0325011}). This proof relies on a new representation formula of the state (inspired by the Magnus formula), in terms of Lie brackets. We use the same quadratic brackets as in \cite{RogerBrockett2013}. We will discuss these articles in more detail later in the paper.
	\subsection{Small-time continuously approximately reachable vector}
	In this section, we recall a result from \cite{bournissou2022smalltime} that will be a key tool in this article.
	\begin{defi}[Concatenation] Let $0<T_1<T_2$, $u:[0,T_1]\to\rr$ and $\tilde{u}:[0,T_2]\to\rr$. One defines the function $u\#\tilde{u}:[0,T_1+T_2]\to\rr$ as
		$$u\#\tilde{u}:=u\mathbb{1}_{(0,T_1)}+\tilde{u}(\cdot-T_1)\mathbb{1}_{(T_1,T_1+T_2)}.$$
	\end{defi}
	Let $(X,\left\|\cdot\right\|_X)$ be a Banach space over $\rr$. Let $(E_T,\left\|\cdot\right\|_{E_T})$ be a family of normed vector spaces of functions defined on $[0,T]$ for $T>0$. Assume that for all $T_1,T_2>0$, for all $u\in E_{T_1}$, $\tilde{u}\in E_{T_2}$, $u\#\tilde{u}\in E_{T_1+T_2}$ and the following inequality holds
	$$\left\|u\#\tilde{u}\right\|_{E_{T_1+T_2}}\leqslant\left\|u\right\|_{E_{T_1}}+\left\|\tilde{u}\right\|_{E_{T_2}}.$$ 
	
	Finally, let $(\mathscr{F}_T)_{T>0}$ be a family of functions from $X\times E_T$ to $X$ for $T>0$. The goal is to prove that the map $\mathscr{F}_T$ in locally onto. To do this, we use the following definition; this is an equivalent of the notion of tangent vector introduced by Kawski in \cite{kawski2}, but in infinite-dimensional case.
	\begin{defi}[Small-time continuously approximately reachable vector] A vector $\xi\in X$ is called a small-time continuously approximately reachable vector if there exists a continuous map $\Xi:[0,+\infty[\to X$ with $\Xi(0)=\xi$ such that for all $T>0$, there exist $C,\rho,s>0$ and a continuous map $z\in(-\rho,\rho)\mapsto u_z\in E_T$ such that
		$$\forall z\in(-\rho,\rho), \qquad \left\|\mathscr{F}_T(0,u_z)-z\Xi(T)\right\|_X\leqslant C|z|^{1+s} \qquad\text{with}\qquad \left\|u_z\right\|_{E_T}\leqslant C|z|^s.$$
	\end{defi}
	\begin{theorem}\label{STCLvecteurtangent} Assume the following hypotheses hold.
		\begin{enumerate}
			\item[$(A_1)$] For all $T>0$, $\mathscr{F}_T:X\times E_T\to X$ is of class $\mathcal{C}^2$ on a neighborhood of $(0,0)$ with $\mathscr{F}_T(0,0)=0$.
			\item[$(A_2)$] For all $x\in X$, $T\in\rr_+\mapsto\mathrm{d}\mathscr{F}_T(0,0)\cdot(x,0)\in X$ can be continuously extended at zero with $\mathrm{d}\mathscr{F}_0(0,0)\cdot(x,0)=x$.
			\item[$(A_3)$] For all $T_1,T_2>0$, for all $x\in X$, for all $u\in E_{T_1}$ and $v\in E_{T_2}$, 
			$$\mathscr{F}_{T_1+T_2}(x,u\#v)=\mathscr{F}_{T_2}(\mathscr{F}_{T_1}(x,u),v).$$
			\item[$(A_4)$] The space $H:=\text{Im}(\mathrm{d}\mathscr{F}_T(0,0)\cdot (0,\cdot))$ doesn't depend on time, is closed and of finite codimension $n$.
			\item[$(A_5)$] There exists $\mathcal{M}$, a supplementary of $H$ that admits a basis $(\xi_i)_{1\leqslant i\leqslant n}$ of small-time continuously approximately reachable vectors.
		\end{enumerate}
		Then, for all $T>0$, $\mathscr{F}_T$ is locally onto from zero: for all $\eta>0$, there exists $\delta>0$ such that for all $x_f\in X$ with $\left\|x_f\right\|_X<\delta$, there exists $u\in E_T$ with $\left\|u\right\|_{E_T}<\eta$ such that 
		$$\mathscr{F}_T(0,u)=x_f.$$
	\end{theorem}
	\begin{rmq} Using a change of function argument, this theorem is also true around another equilibrium than $(0,0)$.
	\end{rmq}
	\subsection{Plan of the paper}
	The paper is organized in the following way: in Section $2$, we present a proof of a theorem of STLC thanks to quadratic terms for control-affine system  in the finite-dimensional case, as a toy-model for the bilinear Schrödinger equation. In Section $3$, we present classical properties about the bilinear Schrödinger equation. Finally, in Section $4$, we give the proof of the main theorem of this article. Some elements of proof are developed in the appendix.
	\section{The finite-dimensional case}
	\subsection{Framework and notations}
	One considers the control-affine system
	\begin{equation}\label{affine-syst}x'(t)=f_0(x(t))+u(t)f_1(x(t))+v(t)f_2(x(t)),
	\end{equation}where the state $x(t)\in\rr^d$, the controls are scalar-input functions $u(t),v(t)\in\rr$ and $f_0,f_1$ and $f_2$ are analytic vector fields on a neighborhood of $0$ in $\rr^d$, such that $f_0(0)=0$. The last hypothesis ensures that $0$ is an equilibrium of the free control-affine system (i.e.\ with $(u,v)\equiv 0)$. \\\\\indent For each $0<T_1<T$, $u,v\in L^1((T_1,T),\rr)$, there exists a unique maximal mild solution to \eqref{affine-syst} with initial data $p\in\rr^d$ at time $T_1$, which we will denote by $x(\cdot;(u,v),p)$. We will sometimes note $x(\cdot;(u,v),p,T_1)$. As we are interested in small-time and small controls, the solution is well-defined up to time $T$. The following concepts were introduced by Beauchard and Marbach in \cite{beauchard2017quadratic}, in 2018. 
	\begin{defi}[$W^{m,\infty}_0$-STLC]
		Let $m\in\nn$. We say that system \eqref{affine-syst} is $W^{m,\infty}_0-$STLC when, for every $T,\varepsilon>0$, there exists $\delta>0$, such that, for every $x_f\in B(0,\delta)$, there exist $u,v\in W^{m,\infty}_0((0,T),\rr)$ with $\left\|(u,v)\right\|_{W^{m,\infty}}\leqslant \varepsilon$ and $x(T;(u,v),0)=x_f$.
	\end{defi}
	\begin{defi}[Smooth-STLC] The system \eqref{affine-syst} is smooth-STLC if it is $W^{m,\infty}_0-$STLC, for every $m\in\nn$.
	\end{defi}	
	
	We use the definitions and notations of Beauchard and Marbach in \cite{beauchard2024unified}.
	Let $X:= \{X_0,X_1, X_2\}$ be a set of three non-commutative indeterminates.
	\begin{defi}[Free algebra]
		\label{def:free-algebra}
		We consider $\mathcal{A}(X)$ the free algebra generated by $X$ over the field $\rr$, 
		i.e.\ the unital associative algebra of polynomials of the indeterminates $X_0$, $X_1$ and $X_2$.
	\end{defi}
	\begin{defi}[Free Lie algebra]
		For $a,b\in\mathcal{A}(X)$, one defines the Lie bracket of $a$ and $b$ as $[a,b]:= ab - ba$, also called $\text{ad}_a(b)$. One defines by induction on $n\in\nn$, $\text{ad}_a^{n+1}(b)=[a,\text{ad}_a^n(b)]$.
		This operation is anti-symmetric and satisfies the Jacobi identity: for all $a,b,c\in\mathcal{A}(X)$,
		$$[a,[b,c]]+[c,[a,b]]+[b,[c,a]]=0.$$
		Let $\mathcal{L}(X)$ be the free Lie algebra generated by $X$ over the field $\rr$, 
		i.e.\ the smallest linear subspace of $\mathcal{A}(X)$ containing $X$ and stable by the Lie bracket $[\cdot,\cdot]$.
	\end{defi}
	\begin{defi}[Iterated brackets]
		Let $\Br(X)$ be the free magma over $X$, i.e.\ the set of iterated brackets of elements of $X$, 
		defined by induction as: $X_0, X_1, X_2 \in \Br(X)$ and if $a, b \in \Br(X)$, then the ordered pair $(a,b)$ belongs to $\Br(X)$. 
		
		There is a natural evaluation mapping $\ev$ from $\Br(X)$ to $\mathcal{L}(X)$ defined by induction by 
		$\ev(X_i):= X_i$ for $i=0,1,2$ and $\ev((a, b)):= [\ev(a), \ev(b)]$. 
	\end{defi}
	\begin{defi}[Length and homogeneous layers within $\mathcal{L}(X)$]  For $b\in \Br(X)$, $|b|$ denotes the length of $b$. For $i\in\llbracket 0,2\rrbracket$, for $b \in \Br(X)$, $n_i(b)$ denotes the number of occurrences of the indeterminate $X_i$ in $b$. We will use the notation: $n(b)=n_1(b)+n_2(b)=|b|-n_0(b)$. For $i\in\nn$, $S_i(X)$ is the vector subspace of $\mathcal{L}(X)$ defined by
		$$S_i(X):= \spn\{\ev(b), \ b\in\Br(X), \ n(b)=i\}.$$
	\end{defi}
	For example, with $b=(((X_1,(X_0,X_2)),X_2),(X_1,X_2))$, $|b|=6$, $n_0(b)=1$, $n_1(b)=2$, $n_2(b)=3$ and $n(b)=5$.
		\begin{defi}[Left and right factors] For $b\in\Br(X)$ with 
		$|b|> 1$, $b$ can be written in a unique way as $b = (b_1, b_2)$, with $b_1, b_2\in\Br(X)$. We use the notations
		$\lambda(b) = b_1$ and $\mu(b) = b_2$, which define maps $\lambda,\mu :\Br(X)\setminus X\to\Br(X)$.
	\end{defi}
	For example, with $b=((X_1,X_2),((X_2,X_0),X_0))$, one has $\lambda(b)=(X_1,X_2)$ and $\mu(b)=((X_2,X_0),X_0)$.
	\begin{defi}[Bracket integration $b0^\nu$] \label{def:0nu}
		For $b \in \Br(X)$ and $\nu \in \nn$, we use the unconventional short-hand $b 0^\nu$ to denote the right-iterated bracket $(\dotsb(b, X_0), \dotsc, X_0)$, where $X_0$ appears $\nu$ times.
	\end{defi}
	For example, if $b=((X_1,X_2),(X_0,X_2))$, then $b0^2=((((X_1,X_2),(X_0,X_2)),X_0),X_0)$.
	\begin{defi}[Lie bracket of vector fields]
		Let $f,g:\Omega\to\rr^d$ be two smooth vector fields on an open subset $\Omega$ of $\rr^d$. One defines
		\begin{equation}
			[f,g]:x\in\Omega\mapsto \mathrm{D}g_x f(x) - \mathrm{D}f_x g(x).
		\end{equation}
	\end{defi}
	\begin{defi}[Evaluated Lie bracket]
		\label{Def:evaluated_Lie_bracket}
		Let $f_0, f_1,f_2:\Omega\to\rr^d$ be smooth vector fields on an open subset $\Omega$ of $\rr^d$ and  $f=\{f_0,f_1,f_2\}$.
		
		For $b \in \mathcal{L}(X)$, we define $f_b:=\Lambda(b)$, where $\Lambda:\mathcal{L}(X) \to \mathcal{C}^\infty(\Omega,\rr^d)$ is the unique homomorphism of Lie algebras such that $\Lambda(X_i)=f_i$, for $i\in\llbracket0,2\rrbracket$.
		
		We will write $f_b$ instead of $f_{\ev(b)}$ when $b \in \Br(X)$. Finally, for $\mathcal{N}\subset\Br(X)$, we use the notation
		\begin{equation}
			\mathcal{N}(f)(0):= \spn \{ f_b(0),\ b \in \mathcal{N} \} \subseteq \rr^d.
		\end{equation}
	\end{defi}
	For example, if $b=(((X_1,X_2),X_0),(X_1,X_0))$, one has $f_b:=[[[f_1,f_2],f_0],[f_1,f_0]]$.
	\subsection{An adapted basis of the free Lie algebra}
	\begin{defi}[Hall set]
		A {Hall set} is a subset $\B$ of $\Br(X)$ endowed with a total order $<$ such that
		\begin{itemize}
			\item $X \subset \B$,
			\item for all $b_1, b_2 \in \Br(X)$, $(b_1, b_2) \in \B$ iff $b_1, b_2 \in \B$, $b_1 < b_2$ and either $b_2 \in X$ or $\lambda(b_2) \leqslant b_1$, 
			\item for all $b_1, b_2 \in \B$ such that $(b_1,b_2) \in \B$ then $b_1 < (b_1,b_2)$.
		\end{itemize}
	\end{defi}
	\begin{theorem}[Viennot, \cite{Krob1987}]
		\label{thm:viennot}
		Let $\B \subset \Br(X)$ be a Hall set. 
		Then $\ev(\B)$ is a basis of $\mathcal{L}(X)$.
	\end{theorem}
	For a Hall set $\B$ and $A \subset \mathbb{N}$, we denote by $\B_A$:
	$$\B_A:=\{ b \in \B, \ n(b) \in A \}.$$ When $A$ is a singleton, we write $\B_N$ instead of $\B_{\{N\}}$. 
	\par The definition of a Hall set is also an algorithm for its construction.
	Indeed, the subsets $\B_N$ of a Hall set $\B$ can be constructed by induction on $N$. 
	One may start, for example,  with $\B_0=\{X_0\}$ and 
	$\B_1=\{ X_1 0^{\nu_1}, X_2 0^{\nu_2}, \ \nu_1, \nu_2 \in \mathbb{N}\}$ with the following order
	$$\forall k\in \mathbb{N}, \qquad	X_10^k<X_20^k<X_10^{k+1}<X_20^{k+1}<\cdots<X_0.$$
	which is compatible with the three axioms above. 
	For $N \geqslant 2$, to find all Hall elements $b \in \B_N$ given $\B_{\llbracket 1 , N-1 \rrbracket}$,
	one adds first all $(a, b)$ with $a \in \B_{N-1}$, $b \in X$ and $a < b$. 
	Then for each bracket $b = (b_1, b_2) \in \B_{\llbracket 1 , N-1 \rrbracket}$,
	one adds all the $(a, b)$ with $a \in \B_{N-n(b)}$ and $b_1 \leqslant a < b$. 
	Finally, one inserts the newly generated elements of $\B_N$ into an ordering, maintaining the condition that $a < (a, b)$.
	With this construction, one obtains the following statement.
	\begin{prop}\label{Prop:Bs} There exists a Hall set $\mathcal{B}$ such that $X_0$ is maximal, 
		\begin{equation}\label{base-s1}\mathcal{B}_1=\left\lbrace M_j^1:=X_10^j, \quad M_j^2:=X_20^j;\quad j\in\nn\right\rbrace,\end{equation}
		and $\mathcal{B}_2=\mathcal{B}_{2,good}\cup\mathcal{B}_{2,bad}$ with 
		\begin{equation}\label{base-s2bad}\mathcal{B}_{2,bad}=\left\lbrace W_{j,l}^1:=(M_{j-1}^1,M_j^1)0^l,\quad W_{j,l}^2:=(M_{j-1}^2,M_j^2)0^l;\quad j\in\nn^* ,l\in\nn \right\rbrace,\end{equation} and
		\begin{equation}\label{base-s2good}\mathcal{B}_{2,good}=\left\lbrace C_{j,l}:=(-1)^j\left(M_{\lfloor\frac{j+1}{2}\rfloor}^1,M_{\lfloor\frac{j}{2}\rfloor}^2\right)0^l;\quad j,l\in\nn\right\rbrace.\end{equation}
	\end{prop}
	We used the following notation.
	\begin{nota} For $b=(b_1,b_2)\in\Br(X)$ and $j\in\nn$, we use the notation $(-1)^jb$ as $(-1)^jb=(b_2,b_1)$ if $j$ is odd and $(-1)^jb=b$ if $j$ is even.
	\end{nota}
	\begin{rmq}
		When $l=0$, we will write $W_j^1,W_j^2,C_j$ instead of $W_{j,0}^1,W_{j,0}^2$ and $C_{j,0}$.
	\end{rmq}
	\subsection{A quadratic sufficient condition for STLC}
	The purpose of this section is to prove the following theorem.
	\begin{theorem}\label{Main_DF} Let $d,L\in\nn^*$, $f_0,f_1,f_2$ be analytic vector fields on a neighborhood of $0$ in $\rr^d$, such that $f_0(0)=0$. Assume that $r:=\dim\left(S_1(f)(0)\right)<d$ and that there exist brackets $b_{r+1},\cdots,b_d\in\mathcal{B}_{2,good}$ such that $\underset{i\in\llbracket r+1,d\rrbracket}\max|b_i|=L$ and \begin{equation}\label{LARC}S_1(f)(0)\oplus \spn(f_{b_{r+1}}(0),\cdots,f_{b_d}(0))=\rr^d.\end{equation}
			Let $j\in\llbracket r+1,d\rrbracket$. Assume that
			\begin{equation}\label{compensationb2badj}\text{for all }b\in\mathcal{B}_{2,bad}\text{ such that }|b|\leqslant |b_j|,\quad f_b(0)\in{S}_1(f)(0).\end{equation}
		 Then, for every $m\in\nn$, the vector $f_{b_j}(0)$ is a small-time $W^{m,\infty}_0-$continuously approximately reachable vector. Thus, if
		\begin{equation}\label{compensationb2bad}\text{for all }b\in\mathcal{B}_{2,bad}\text{ such that }|b|\leqslant L,\quad f_b(0)\in{S}_1(f)(0),\end{equation}the control-affine system \eqref{affine-syst} is smooth-STLC.
	\end{theorem}
	\begin{rmq}
		The case $r=d$ corresponds to the linear test, see for instance \cite{coronbook}, or \cite[Theorem $1$]{beauchard2017quadratic} for the smooth-STLC case.
	\end{rmq}
	A very simple system for which the theorem applies is the following one
	$$\left\lbrace\begin{array}{ccc}
		x_1'&=&u\\
		x_2'&=&x_1v
	\end{array}\right..$$ 
	Indeed, some computations on Lie brackets give: $f_{M_0^1}(0)=e_1$ and $f_{C_0}(0)=e_2$. Thus, the LARC \eqref{LARC} is verified. Moreover, $\{b\in\mathcal{B}_{2,bad}, \ |b|\leqslant |C_0|\}=\emptyset$. Then, the condition \eqref{compensationb2bad} is verified. Thus, the system is smooth-STLC.
	\\\par
	Theorem \ref{Main_DF} gives a sufficient condition for smooth-STLC; this condition is not necessary: indeed, consider the system 	$$\left\lbrace\begin{array}{ccc}
		x_1'&=&u\\
		x_2'&=&v\\
		x_3'&=&x_1^2+x_2^2+\alpha x_1x_2
	\end{array}\right.,$$ with $\alpha\in\rr^*$. One can prove that ${S}_1(f)(0)=\spn(e_1,e_2)$ and $f_{C_1}(0)=\alpha e_3$. Then, the LARC \eqref{LARC} is verified. Moreover, $f_{W_1^1}(0)=2 e_3\notin {S}_1(f)(0)$ and $|W_1^1|=|C_1|.$ Then, the hypothesis \eqref{compensationb2bad} is not verified. However, if $|\alpha|>2$, one can prove that the system is smooth-STLC.
	\\\par Theorem \ref{Main_DF} is a corollary of Sussmann's $S(\theta)$-condition (see Appendix \ref{sussman}). We propose another proof of Theorem \ref{Main_DF},
	relying on a representation formula of the state, presented in Propositions \ref{Prop:Magnus} and \ref{Prop:rep_form}.
	The advantage of this alternative proof strategy is that it can be adapted to the infinite dimensional framework of PDEs:
	the representation formula does not hold in its full generality for the state of the Schrödinger PDE, 
	but its leading terms can be extracted from an expansion of the solution (see Proposition \ref{quad-fini}),
	so that we can conclude with a similar proof strategy.
	\\\\\indent In \cite{RogerBrockett2013}, the author is interested in particular linear and quadratic ODEs systems of the form $\left\lbrace\begin{array}{cccc}x'&=&Ax+Bu,& \\w_i'&=&{}^txD_ix,& \ 1\leqslant i\leqslant r\end{array}\right.,$ with $D_i$, well-chosen symmetric matrices. He proves a small-time global controllability result. For that, Brockett shows that the matrices $D_i$ are associated with good brackets, see \cite[Theorem $3.4$]{RogerBrockett2013}. The good quadratic Lie brackets involved are elements shaped as $(M_{\nu}^1,X_2)$, $\nu\in\nn$.  They are linear combination of elements of our base of $\mathcal{B}_{2,good}$.
	\subsection{Some tools for the proof}
	\begin{defi}[Coordinates of the second kind]
		Let $\mathcal{B} \subset \Br(X)$ be a Hall set.
		The coordinates of the second kind associated with $\B$ is the unique family
		$(\xi_b)_{b\in\B}$ of functionals $\rr_+ \times L^1_{\mathrm{loc}}(\rr_+,\rr)^2 \to \rr$ defined by induction as: 
		for every $t>0$ and $u, v \in L^1((0,t),\rr)$,
		\begin{itemize}
			\item $\xi_{X_0}(t,(u,v)):= t$, \ $\xi_{X_1}(t,(u,v)): = u_1(t)$ \ and \ $\xi_{X_2}(t,(u,v)):= v_1(t)$,
			\item for $b \in \mathcal{B} \setminus X$, there exists  a unique couple $(b_1,b_2)$ of elements of $\mathcal{B}$ such that $b_1<b_2$ and a unique maximal integer $m\geqslant 1$ with $b=\ad_{b_1}^m (b_2)$ and then
			\begin{equation}
				\xi_{b}(t,(u,v)):=\frac{1}{m!} \int_0^t  \xi_{b_1}^m(s,(u,v)) \dot{\xi}_{b_2}(s,(u,v)) \dd s.
			\end{equation}
		\end{itemize}
	\end{defi}
	With this definition, one immediately obtains the following expressions.
	\begin{prop}[Coordinates of the second kind] The following equalities hold.
		\begin{enumerate}
			\item[1.] For every $j\in\nn$, for every $t>0$,
			\begin{equation}\label{egbase-s1}\xi_{M_j^1}(t,(u,v))=u_{j+1}(t),\qquad\xi_{M_j^2}(t,(u,v))=v_{j+1}(t).\end{equation}
			\item[2.] For every $j\in\nn^*$, $l\in\nn$, for every $t>0$,
			\begin{equation}\label{egbase-s2bad}\begin{gathered}\xi_{W_{j,l}^1}(t,(u,v))=\frac{1}{2}\int_0^t\frac{(t-s)^l}{l!}u_j^2(s)\ds,\\ \xi_{W_{j,l}^2}(t,(u,v))=\frac{1}{2}\int_0^t\frac{(t-s)^l}{l!}v_j^2(s)\ds.\end{gathered}\end{equation}
			\item[3.] For every $j,l\in\nn$, for every $t>0$,
			\begin{equation}\label{egbase-s2good}\xi_{C_{j,l}}(t,(u,v))=\int_0^t\frac{(t-s)^l}{l!}u_{\left\lfloor\frac{j}{2}\right\rfloor+1}(s)v_{\left\lfloor\frac{j+1}{2}\right\rfloor}(s)\ds.\end{equation}
		\end{enumerate}
	\end{prop}
	\begin{proof}
		The first two points are proved by Beauchard and Marbach in \cite[Proposition $3.7$]{beauchard2024unified}. For the last one, the definition by induction on $\xi$ and some computations lead to the result.
	\end{proof}
	\begin{prop}[Magnus formula]\label{Prop:Magnus}
		Let $M\in\nn$, $\delta,T>0$, $f_0,f_1,f_2:B(0,5\delta)\to\rr^d$ be analytic vector fields with T$\left\|f_0\right\|_{\mathcal{C}^0}\leqslant \delta$. There exist $\gamma,C>0$ such that, for every $u,v\in L^1((0,T),\rr^d)$ with $\left\|u\right\|_{L^1},\left\|v\right\|_{L^1}\leqslant \gamma$, $p\in\mathcal{B}(0,\delta)$ and $t\in[0,\gamma]$, 
		\begin{equation}\label{x=e^Z2+O}
			\left\|x(t;(u,v),p,0)-e^{\mathcal{Z}_M(t;f,(u,v))}e^{tf_0}p\right\|\leqslant C\left\|(u,v)\right\|_{L^1(0,t)}^{M+1},
		\end{equation}
		with 
		$$\mathcal{Z}_M(t;f,(u,v))(0) =\sum_{b \in \B_{\llbracket 1 ,M \rrbracket}} \eta_b(t,(u,v)) f_b(0),$$
		where the series is absolutely convergent and the following explicit expressions hold for every $t>0$ and $u,v \in L^1(0,t)$:
		\begin{itemize}
			\item[1.] if $b\in\mathcal{B}_1$, then $\eta_b(t,(u,v))=\xi_b(t,(u,v))$,
			\item[2.] if $b \in \B_{2,good}$, then
			\begin{equation}\label{pollution}\eta_b(t,(u,v))=\xi_b(t,(u,v))+ \sum_{j+k=n_0(b)} \gamma^b_{j,k} u_{j+1}(t) v_{k+1}(t),
			\end{equation}
			where $\gamma^b_{j,k}\in \rr$.
		\end{itemize}
	\end{prop}
	From this proposition, one can deduce the following one, thanks to Lemma \ref{gronwallmag}.
	\begin{crl}[Magnus formula 2] \label{Prop:rep_form}
		Let $M\in\nn^*$, $\delta,T>0$, $f_0, f_1, f_2:B(0,2\delta)\to\rr^d$  be  analytic vector fields with $f_0(0)=0$ and $T\left\|f_0\right\|_{\infty}\leqslant\delta$. For $u,v\in L^1((0,T),\rr)$, as $\|(u,v)\|_{L^1} \to 0$,
		\begin{equation} \label{x=Z2+O}
			x(t;(u,v),0)= \mathcal{Z}_M(t;f,(u,v))(0) +\mathcal{O}\left(  \|(u,v)\|_{L^1(0,t)}^{M+1}  + \left\|x(t;u,v)\right\|^{1+\frac{1}{M}}   \right).
		\end{equation}
	\end{crl}
	These representation formulas rely on \cite{Beauchard_2023}. The path for obtaining it is presented in appendix.
	\begin{rmq}[Homogeneity property]
		Let $\bar{u}, \bar{v} \in L^1(0,1)$, $\lambda,\mu \in \rr$, $T>0$, 
		$u: t \in (0,T) \mapsto \lambda \bar{u}(\frac tT)$ and  
		$v: t \in (0,T) \mapsto \mu \bar{v}(\frac tT)$. Then, for all $b\in\mathcal{B}$,
		\begin{equation}\label{homogeneous}
			\eta_b(T,(u,v))=\lambda^{n_1(b)}\mu^{n_2(b)} T^{|b|} \eta_b(1,(\bar{u},\bar{v})).
		\end{equation}
	\end{rmq}
	\begin{lm}
		For every $M\in\nn^*$, there exists $C_M > 0$ such that, for every $T> 0$, $u,v\in L^1((0,T),\rr)$, $b\in\mathcal{B}$ with $n(b)\leqslant M$ and $t\in[0,T]$,
		\begin{equation}\label{sizeetab}
			|\eta_b(t,(u,v))| \leqslant\frac{C_M}{|b|!} t^{n_0(b)}\left\|(u,v)\right\|^{n(b)}_{L^1(0,t)}.
		\end{equation}
	\end{lm}
	\begin{proof}
		This lemma is proved in \cite[Proposition 52]{Beauchard_2023}.
	\end{proof}
	\subsection{A quadratic moment problem}
	To prove Theorem \ref{Main_DF}, we solve quadratic moment problems.
	\begin{prop}[Moment problems] \label{momentfini}
		Let $B$ be a finite subset of $\mathcal{B}_1\cup\mathcal{B}_{2,\text{good}}$ and $b_0 \in B\cap\mathcal{B}_{2,good}$.
		There exist $u,v \in\mathcal{C}^\infty_c((0,1),\rr)$ such that,
		\begin{equation}\label{momentfini1}
			\text{for every }b \in B,\qquad \eta_b(1,(u,v))=\delta_{b,b_0}.
		\end{equation}
	\end{prop}
	\begin{proof}
		One may assume that
		$B=\left\lbrace M_{\nu}^1,\ M_{\nu}^2, \ C_{\mu,\nu}, \ (\mu,\nu)\in\llbracket 0,\mu^*\rrbracket\times\llbracket 0,\nu^*\rrbracket\right\rbrace.$ One considers $N=\mu^*+\nu^*+1$.  We are looking for $u$ and $v$ in the form of $u=\phi^{(N)}$ and $v=\psi^{(N)}$, with $\phi,\psi\in\mathcal{C}^{\infty}_c((0,1),\rr)$. One recalls that for $b\in\mathcal{B}_1$, $\eta_b=\xi_b$. Then, for every $\mu\in\llbracket 0,\mu^*\rrbracket$, $\eta_{M_{\mu}^1}(1,(u,v))=u_{\mu+1}(1)=0$ and $\eta_{M_{\mu}^2}(1,(u,v))=v_{\mu+1}(1)=0$, using \eqref{egbase-s1}. Thus, \eqref{momentfini1} is already verified for $b\in B\cap\mathcal{B}_1$. 
		
		Thanks to the choice of $N$, the equations \eqref{egbase-s2good} and  \eqref{pollution} give: for all $(\mu,\nu)\in\llbracket 0,\mu^*\rrbracket\times\llbracket 0,\nu^*\rrbracket$,  $$\eta_{C_{\mu,\nu}}(1,(u,v))=\xi_{C_{\mu,\nu}}(1,(u,v))=\int_0^1\frac{(1-t)^{\nu}}{\nu!}u_{\lfloor\frac{\mu}{2}\rfloor+1}(t)v_{\lfloor\frac{\mu+1}{2}\rfloor}(t)\dt.$$
		Using integrations by parts, one has
		$$\eta_{C_{\mu,\nu}}(1,(u,v))=(-1)^{N-\lfloor\frac{\mu+1}{2}\rfloor}\int_0^1\psi(t)\left(\frac{(1-\cdot)^{\nu}}{\nu!}\phi^{(N-\lfloor\frac{\mu}{2}\rfloor-1)}\right)^{(N-\lfloor\frac{\mu+1}{2}\rfloor)}(t)\dt.$$
		Using the Leibniz formula, one gets: $\eta_{C_{\mu,\nu}}(1,(u,v))=\displaystyle\int_0^1\psi(t)f_{\mu,\nu}(t)\dt$, with
		$$f_{\mu,\nu}:t\in(0,1)\mapsto\sum_{k=0}^{\min(N-\lfloor\frac{\mu+1}{2}\rfloor,\nu)}\binom{N-\lfloor\frac{\mu+1}{2}\rfloor}{k}\frac{(-1)^{k+N-\lfloor\frac{\mu+1}{2}\rfloor}}{(\nu-k)!}(1-t)^{\nu-k}\phi^{(2N-1-\mu-k)}(t).$$
		\textit{Step 1: There exists $\phi\in\mathcal{C}^{\infty}_c((0,1),\rr)$ s.t. the family $\mathcal{F}:=\left(f_{\mu,\nu}\right)_{\substack{0\leqslant\mu\leqslant\mu^*\\0\leqslant\nu\leqslant\nu^*}}$ is linearly independent on $(0,1)$.}  Let $m\geqslant\nu^*+2$ and $\chi\in\mathcal{C}^{\infty}_c((0,1),\rr)$ such that $\chi\equiv 1$ on $(\frac 14,\frac 34)$. One defines $\phi:t\mapsto e^{t^m}\chi(t)\in\mathcal{C}^{\infty}_c((0,1),\rr)$. By induction, one can prove that: for all $l\in\nn$, there exists $P_l\in\mathbb{R}[X]$ with $\deg(P_l)=l(m-1)$ such that, for all $t\in(\frac 14,\frac 34)$, $\phi^{(l)}(t)=e^{t^m}P_l(t)$. Then, 
		\begin{equation}\label{deg} \forall(\mu,\nu)\in\llbracket0,\mu^*\rrbracket\times\llbracket 0,\nu^*\rrbracket, \quad \deg\left(\restriction{f_{\mu,\nu}(t)e^{-t^m}}{(\frac 14,\frac 34)}\right)=\nu+(m-1)(2N-1-\mu).\end{equation}
		Thanks to the choice of $m$, all the degrees \eqref{deg} are different. Then, $\left(t\mapsto f_{\mu,\nu}(t)e^{-t^m}\right)_{\substack{0\leqslant\mu\leqslant\mu^*\\0\leqslant\nu\leqslant\nu^*}}$ is linearly independent on $(0,1)$. Consequently, one obtains the result for $\mathcal{F}$.
		\\\\\textit{Step 2: Construction of $u$ and $v$.} One chooses $u:=\phi^{(N)}$, with the function $\phi$ obtained at step $1$. By construction of $\phi$, there exists $\psi\in\mathcal{C}^\infty_c((0,1),\rr)$, solution to the following moment problems
		$$\text{for every }(\mu,\nu) \in\llbracket 0,\mu^*\rrbracket\times\llbracket 0,\nu^*\rrbracket, \quad
		\int_0^1\psi(t)f_{\mu,\nu}(t) dt = \delta_{C_{\mu,\nu},b_0}.$$
		We conclude with $v:=\psi^{(N)}$.
	\end{proof}
	\subsection{Tangent vector}
	Let $m\geqslant 0$. The purpose of this section is to prove the following proposition.
	\begin{prop}\label{vecteurtangent}Let $j\in\llbracket r+1,d\rrbracket$. Assume that \eqref{LARC} and \eqref{compensationb2badj} hold. The vector $f_{b_j}(0)$ is a small-time $W^{m,\infty}_0-$continuously approximately reachable vector associated with vector variations $e^{\frac{\cdot}{2}H_0}f_{b_j}(0)$.
	\end{prop}
	 Taking into account the assumption \eqref{LARC}, we can consider $\mathbb{P}$, the linear projection on $\spn(f_{b_i}(0))_{r+1\leqslant i\leqslant d}$ parallel to ${S}_1(f)(0)$. The proof is divided in two steps. In the first one, we prove that the system can move in the direction lost at the linear order $f_{b_j}(0)$, in projection. 
	\begin{prop}Let $j\in\llbracket r+1,d\rrbracket$. Assume that \eqref{LARC} and \eqref{compensationb2badj} hold. Let $q_j:=\left\lfloor\frac{n_0(b_j)}{2}\right\rfloor$, $s_j:=\frac{1}{4(|b_j|+m)}$, $\alpha_j=\frac{3}{8}+\frac{m}{8(|b_j|+m)}$ and $\delta_j:=\alpha_j+(q_j+2)s_j$. For all $T_1>0$, there exist $C,\rho>0$ and a continuous map $z\in\rr\mapsto (u_z,v_z)\in W^{m,\infty}_0(0,T_1)^2$ such that
		\begin{equation}\label{dimfinieproj}\forall z\in(-\rho,\rho),\quad \left\|\mathbb{P}\left(x(T_1;(u_z,v_z),0,0)\right)-zf_{b_j}(0)\right\|\leqslant C|z|^{1+s_j}.
		\end{equation}
		The size of the controls is given by: for all $k\geqslant 0$, for all $r\in[1,+\infty]$,  there exists $C>0$, such that 
		\begin{equation}\label{dimfinicontrol}\forall z\in(-\rho,\rho), \qquad \left\|u_z\right\|_{W^{k,r}},\left\|v_z\right\|_{W^{k,r}}\leqslant C|z|^{\alpha_j+s_j\left(\frac 1r-k\right)}.\end{equation}
		Finally, for all $z\in(-\rho,\rho)$,
		\begin{equation}\label{dimfinierest}\left\|x(T_1;(u_z,v_z),0,0)\right\|\leqslant C|z|^{\delta_j}.\end{equation}
		Note that $\delta_j>\frac 12$.
	\end{prop}
	\begin{proof} Let $T_1>0$ and $\rho=T_1^{\frac{1}{s_j}}$. One defines the brackets set $B=\left\lbrace M_j^1, M_j^2, \ j\in\llbracket 0,q_j\rrbracket\right\rbrace\cup\left\lbrace b\in\mathcal{B}_{2,good}, \ |b|\leqslant |b_j|\right\rbrace$.
		Thanks to Proposition \ref{momentfini}, we consider $\bar{u},\bar{v}\in\mathcal{C}^{\infty}_c(0,1)$ such that $\eta_{b}(1,(\bar{u},\bar{v}))=\delta_{b,b_j}$,  for all $b\in B$. We define, for $z\in\rr^*$, 
		$$u_z,v_z:t\in[0,T_1]\mapsto\text{sgn}(z)|z|^{\alpha_j}\bar{u}\left(\frac{|z|^{s_j}-T_1+t}{|z|^{s_j}}\right),|z|^{\alpha_j}\bar{v}\left(\frac{|z|^{s_j}-T_1+t}{|z|^{s_j}}\right).$$
		For all $z\in(-\rho,\rho)$, one has \begin{equation}\label{suppp}\text{Supp}\left(\left(u_z\right)_{q_j+1}\right),\text{Supp}\left(\left(v_z\right)_{q_j+1}\right)\subset (T_1-|z|^{s_j},T_1)\subset (0,T_1).\end{equation}
		For all $k\geqslant 0$, $ r\in [1,+\infty[$, using the Poincaré's inequality, the support condition \eqref{suppp} and a change of variables, we get for all $z\in(-\rho,\rho)\setminus\left\lbrace 0\right\rbrace$,
		$$\left\|u_z\right\|_{W^{k,r}}^r\leqslant  C\int_0^{1}\left||z|^{\alpha_j-ks_j}\bar{u}^{(k)}(\sigma)\right|^r|z|^{s_j}\mathrm{d}\sigma\leqslant C|z|^{\left(\alpha_j-ks_j\right)r+s_j}\left\|\bar{u}^{(k)}\right\|_{L^r}^r.$$ The inequality with $r=+\infty$ is proved in the same way and we get \eqref{dimfinicontrol}. 
		\\\\We prove \eqref{dimfinierest}; we estimate the linear term. Note that, for all $z\in(-\rho,\rho)$, for all $i\in\{1,2\}$, for all $k\in\nn$, the homogeneity property \eqref{homogeneous} gives
		\begin{equation}\label{etaa}\begin{split}\eta_{M_k^i}(T_1,(u_z,v_z))&=\eta_{M_k^i}\left(|z|^{s_j},\left(\text{sgn}(z)|z|^{\alpha_j}\bar{u}\left(\frac{\cdot}{|z|^{s_j}}\right),|z|^{\alpha_j}\bar{v}\left(\frac{\cdot}{|z|^{s_j}}\right)\right)\right)\\&=\text{sgn}(z)^{\delta_{1,i}}|z|^{\alpha_j+|M_k^i|s_j}\eta_{M_k^i}(1,(\bar{u},\bar{v})).\end{split} \end{equation}
		Consequently, the definition of $(\bar{u},\bar{v})$ and \eqref{etaa} gives: for all $z\in(-\rho,\rho)$, for all $i\in\{1,2\}$, for all $k\in\llbracket 0,q_j\rrbracket$, $\eta_{M_k^i}(T_1,(u_z,v_z))=0$. Using \eqref{sizeetab} and \eqref{dimfinicontrol} with $r=1$,
		\begin{equation}\label{Z1}\mathcal{Z}_1(T_1;f,(u_z,v_z))(0)=\sum_{i=1}^2\sum_{k=q_j+1}^{+\infty}\eta_{M_k^i}(T_1,(u_z,v_z))f_{M_k^i}(0)=\mathcal{O}\left(|z|^{\alpha_j+(q_j+2)s_j}\right).
		\end{equation}
		The Magnus formula \eqref{x=Z2+O} with $M=1$ leads to
		$$x(T_1;(u_z,v_z),0,0)=\mathcal{Z}_1(T_1;f,(u_z,v_z))(0)+\mathcal{O}\left(\left\|(u_z,v_z)\right\|_{L^1}^2+\left\|x(T_1;(u_z,v_z),0,0)\right\|^2\right).$$
		Using the size of the controls given by \eqref{dimfinicontrol} with $(k,r)=(0,1)$ and \eqref{Z1}, one has
		$$x(T_1;(u_z,v_z),0,0)=\mathcal{O}\left(|z|^{\delta_j}+|z|^{2\alpha_j+2s_j}+\left\|x(T_1;(u_z,v_z),0,0)\right\|^2\right).$$
		As $\frac 12<\delta_j\leqslant 2\alpha_j+2s_j$, one gets \eqref{dimfinierest}.
		\\\\Now, we prove \eqref{dimfinieproj}. By definition,
		\begin{equation}\label{Z2global}\mathcal{Z}_2(T_1;f,(u_z,v_z))(0)=\mathcal{Z}_1(T_1;f,(u_z,v_z))(0)+\sum_{b\in\mathcal{B}_2}\eta_b(T_1,(u_z,v_z))f_b(0)=\mathcal{O}\left(|z|^{\delta_j}\right),\end{equation}
		using \eqref{Z1}, the estimates \eqref{sizeetab} and \eqref{dimfinicontrol} with $(k,r)=(0,1)$. Then, using the map $\mathbb{P}$, the hypothesis \eqref{compensationb2badj} on the short brackets of $\mathcal{B}_{2,bad}$ and the homogeneity property \eqref{homogeneous}, we obtain
		\begin{equation*}\mathbb{P}\left(\mathcal{Z}_2 (T_1;f,(u_z,v_z))(0)\right)=\sum_{\substack{b\in\{\mathfrak{b}\in\mathcal{B}_{2,good},\ |\mathfrak{b}|\leqslant |b_j|\}\\ \cup\{\mathfrak{b}\in\mathcal{B}_2,\ |\mathfrak{b}|>|b_j| \}}}\text{sgn}(z)^{n_1(b)}|z|^{2\alpha_j+|b|s_j}\eta_b(1,(\bar{u},\bar{v}))\mathbb{P}\left(f_b(0)\right).\end{equation*} 
		Using the definition of $(\bar{u},\bar{v})$ and the fact that $2\alpha_j+|b_j|s_j=1$, one gets
		\begin{equation}\label{dimfiniequadpol}\mathbb{P}\left(\mathcal{Z}_2 (T_1;f,(u_z,v_z))(0)\right)=zf_{b_j}(0)+\mathcal{O}(|z|^{1+s_j}).\end{equation}
		Then, using the Magnus representation formula \eqref{x=e^Z2+O} with $M=2$, Lemma \ref{gronwallmag} and the projection $\mathbb{P}$, one has
		\begin{equation}\begin{gathered}\label{Z2finn}\mathbb{P}\left(x(T_1;(u_z,v_z),0,0)\right)=\mathbb{P}\left(\mathcal{Z}_2(T_1;f,(u_z,v_z))(0)\right)\\+\mathcal{O}\left(\left\|\mathcal{Z}_2(T_1;f,(u_z,v_z))(0)\right\|_{\infty}\left\|D\mathcal{Z}_2(T_1;f,(u_z,v_z))\right\|_{\infty}+\left\|(u_z,v_z)\right\|_{L^1}^3\right).\end{gathered}\end{equation}
		Finally, note that 
		$$D\mathcal{Z}_2(T_1;f,(u_z,v_z))=\sum_{b\in\mathcal{B}_{\llbracket 1,2\rrbracket}}\eta_b(T_1,(u_z,v_z))Df_b.$$
		Then, with the same steps than for \eqref{Z2global}, one gets
		\begin{equation}\label{DZ2}
			\left\|D\mathcal{Z}_2(T_1;f,(u_z,v_z))\right\|_{\infty}=\mathcal{O}\left(|z|^{\delta_j}\right).
		\end{equation}
		Finally, using \eqref{Z2global}, \eqref{dimfiniequadpol}, \eqref{DZ2} and \eqref{dimfinicontrol} with $(k,r)=(0,1)$ in \eqref{Z2finn}, one obtains \eqref{dimfinieproj}, as $1+s_j\leqslant 2\delta_j\leqslant 3\alpha_j+ 3s_j$.
	\end{proof}
	One defines $H_0=Df_0(0)$. 	In the second step of the proof, we correct the linear part of the solution.
	\begin{prop} Let $j\in\llbracket r+1,d\rrbracket$. Assume that \eqref{LARC} and \eqref{compensationb2badj} hold. For all $0<T_1<T<T_1+1$, there exist $C,\rho>0$ and a continuous map $z\mapsto (U_z,V_z)\in W^{m,\infty}_0(0,T)^2$ such that
		\begin{equation}\label{vecteurtangentbis}\forall z\in(-\rho,\rho), \qquad \left\|x(T;(U_z,V_z),0,0)-ze^{(T-T_1)H_0}f_{b_j}(0)\right\|\leqslant C|z|^{1+s_j},\end{equation}
		with the following estimate on the family of controls
		\begin{equation}\label{dimfiniecontrolfinal}\left\|U_z\right\|_{W^{m,\infty}},\left\|V_z\right\|_{W^{m,\infty}}\leqslant C|z|^{\frac 14}.\end{equation}
	\end{prop}
	\begin{proof} Let $0<T_1<T<T_1+1$, one defines, for $z\in\rr$,
		$$U_z,V_z:=u_z\mathbb{1}_{[0,T_1]}+\tilde{u}_z\mathbb{1}_{[T_1,T]},v_z\mathbb{1}_{[0,T_1]}+\tilde{v}_z\mathbb{1}_{[T_1,T]},$$
		where $u_z,v_z\in W^{m,\infty}_0(0,T_1)$ are the controls defined in the previous proposition and $\tilde{u}_z,\tilde{v}_z\in W^{m,\infty}_0(T_1,T)$ are the controls that correct the linear part of the solution, i.e.\ the controls that drive the solution from $x(T_1;(u_z,v_z),0,0)$ to $\mathbb{P}_{S_1(f)(0)}\left(ze^{(T-T_1)H_0}f_{b_j}(0)\right)$ in projection on ${S}_1(f)(0)$, that is to say
		$$\mathbb{P}_{{S}_1(f)(0)}\left(x(T;(\tilde{u}_z,\tilde{v}_z),x(T_1;(u_z,v_z),0,0),T_1)\right)=\mathbb{P}_{{S}_1(f)(0)}\left(ze^{(T-T_1)H_0}f_{b_j}(0)\right).$$
		Note that $\mathbb{P}_{{S}_1(f)(0)}=\operatorname{Id}-\mathbb{P}$. Then, 
		\begin{equation}\label{cont1}\left\|(\tilde{u}_z,\tilde{v}_z)\right\|_{W^{m,\infty}}\leqslant C\left(\left\|x(T_1;(u_z,v_z),0,0)\right\|+|z|\left\|e^{(T-T_1)H_0}f_{b_j}(0)\right\|\right)\leqslant C|z|^{\delta_j},\end{equation}
		the last estimate is given by \eqref{dimfinierest}.	Moreover, the control estimate \eqref{dimfinicontrol} with $(k,r)=(m,\infty)$ gives
		\begin{equation}\label{cont2}\left\|(u_z,v_z)\right\|_{W^{m,\infty}}\leqslant C|z|^{\alpha_j-ms_j}\leqslant C|z|^{\frac 14}.\end{equation}
		The equations \eqref{cont1} and \eqref{cont2} give \eqref{dimfiniecontrolfinal}, as $\delta_j>\frac 12$. Furthermore, by construction,
		\begin{equation}\begin{gathered}\label{tang1}
			\left\|x(T;(U_z,V_z),0,0)-ze^{(T-T_1)H_0}f_{b_j}(0)\right\|=\\	\left\|\mathbb{P}\left(x(T;(U_z,V_z),0,0)-ze^{(T-T_1)H_0}f_{b_j}(0)\right)\right\|.\end{gathered}
		\end{equation}
		We apply Lemma \ref{gronwalldecomp} with $p=x(T_1;(u_z,v_z),0,0)$ and $(u,v)=(\tilde{u}_z,\tilde{v}_z)$ to obtain
		\begin{equation}\begin{gathered}\label{decomtang}
			x(T;(U_z,V_z),0,0)=x\left(T;(0,0),x(T_1;(u_z,v_z),0,0),T_1\right)+x(T;(\tilde{u}_z,\tilde{v}_z),0,T_1)\\
			+\mathcal{O}\left(\left\|x(T_1;(u_z,v_z),0,0)\right\|\left\|(\tilde{u}_z,\tilde{v}_z)\right\|_{L^{\infty}}+\left\|x(T_1;(u_z,v_z),0,0)\right\|^2\right).\end{gathered}
		\end{equation}
		The equations \eqref{dimfinierest}, \eqref{cont1}, \eqref{tang1} and \eqref{decomtang} give
		\begin{equation}\begin{gathered}\label{tang2}\left\|x(T;(U_z,V_z),0,0)-ze^{(T-T_1)H_0}f_{b_j}(0)\right\|\leqslant \left\|\mathbb{P}(x(T;(\tilde{u}_z,\tilde{v}_z),0,T_1))\right\|\\+ \left\|\mathbb{P}\left(x(T;(0,0),x(T_1;(u_z,v_z),0,0),T_1)-ze^{(T-T_1)H_0}f_{b_j}(0)\right)\right\|+\mathcal{O}\left(|z|^{2\delta_j}\right).\end{gathered}
		\end{equation}
		We designate the second term of the right-hand side of \eqref{tang2} by $A$ and Lemma \ref{devlpode} gives with $p=x(T_1;(u_z,v_z),0,0)$
		$$A=\left\|\mathbb{P}\left(e^{(T-T_1)H_0}x(T_1;(u_z,v_z),0,0)-ze^{(T-T_1)H_0}f_{b_j}(0)\right)\right\|+\mathcal{O}\left(\left\|x(T_1;(u_z,v_z),0,0)\right\|^2\right).$$
		For all $b\in\mathcal{B}_1$, $H_0f_b(0)=Df_0(0)f_b(0)-Df_b(0)f_0(0)=f_{b0}(0)$, with $b0\in\mathcal{B}_1$. Then, ${S}_1(f)(0)$ is stable by $e^{(T-T_1)H_0}.$ Consequently,
		$$e^{(T-T_1)H_0}\left(\operatorname{Id}-\mathbb{P}\right)\left(x(T_1;(u_z,v_z),0,0)\right)\in{S}_1(f)(0)=\ker(\mathbb{P}),$$ and, with \eqref{dimfinierest}, we obtain
		\begin{equation*}\begin{split}A&=\left\|\mathbb{P}\left(e^{(T-T_1)H_0}\mathbb{P}\left(x(T_1;(u_z,v_z),0,0)\right)-ze^{(T-T_1)H_0}f_{b_j}(0)\right)\right\|+\mathcal{O}\left(|z|^{2\delta_j}\right),\\
				A&=\mathcal{O}\left(\left\|\mathbb{P}\left(x(T_1;(u_z,v_z),0,0)\right)-zf_{b_j}(0)\right\|+|z|^{2\delta_j}\right).\end{split}\end{equation*}
		Finally, the equation \eqref{dimfinieproj} gives
		\begin{equation}\label{tangA}A=\mathcal{O}\left(|z|^{1+s_j}\right),\end{equation}
		as $1+s_j\leqslant 2\delta_j.$ To obtain \eqref{vecteurtangentbis}, we finally estimate the first term of  the right-hand size of \eqref{tang2}, let's say $B$. 
		Using once again the Magnus formula \eqref{x=e^Z2+O} with $M=2$, Lemma \ref{gronwallmag} and the projection $\mathbb{P}$, one has
		\begin{equation}\begin{gathered}\label{Z2fin}\mathbb{P}\left(x(T;(\tilde{u}_z,\tilde{v}_z),0,T_1)\right)=\mathbb{P}\left(\mathcal{Z}_2(T;f,(\tilde{u}_z,\tilde{v}_z))(0)\right)\\+\mathcal{O}\left(\left\|\mathcal{Z}_2(T;f,(\tilde{u}_z,\tilde{v}_z))(0)\right\|_{\infty}\left\|D\mathcal{Z}_2(T;f,(\tilde{u}_z,\tilde{v}_z))\right\|_{\infty}+\left\|(\tilde{u}_z,\tilde{v}_z)\right\|_{L^1}^3\right).\end{gathered}\end{equation}
		Again, \eqref{sizeetab} and \eqref{cont1} gives
		\begin{equation*}\begin{gathered}	\left\|\mathbb{P}\left(\mathcal{Z}_2(T;f,(\tilde{u}_z,\tilde{v}_z))\right)\right\|_{\infty}=\mathcal{O}\left(\left\|(\tilde{u}_z,\tilde{v}_z)\right\|_{L^1}^2\right)=\mathcal{O}\left(\left\|(\tilde{u}_z,\tilde{v}_z)\right\|_{W^{m,\infty}}^2\right)=\mathcal{O}\left(|z|^{2\delta_j}\right),\\
				\left\|\mathcal{Z}_2(T;f,(\tilde{u}_z,\tilde{v}_z))\right\|_{\infty},\left\|D\mathcal{Z}_2(T;f,(\tilde{u}_z,\tilde{v}_z))\right\|_{\infty}=\mathcal{O}\left(\left\|(\tilde{u}_z,\tilde{v}_z)\right\|_{W^{m,\infty}}\right)=\mathcal{O}\left(|z|^{\delta_j}\right).\end{gathered}
		\end{equation*}
		Therefore, thanks to the estimates \eqref{Z2fin}, one has
		\begin{equation}\label{tangB}B=\mathcal{O}\left(|z|^{2\delta_j}+\left\|(\tilde{u}_z,\tilde{v}_z)\right\|_{L^1}^3\right)=\mathcal{O}(|z|^{1+s_j}),\end{equation} using \eqref{cont1} and $1+s_j\leqslant 2\delta_j$.
		To conclude, the equations \eqref{tang2}, \eqref{tangA} and \eqref{tangB} lead to the wondered inequality \eqref{vecteurtangentbis}.
	\end{proof}
	We are now in a position to write the proof of the main result of this section.
	\begin{proof}[Proof of Proposition \ref{vecteurtangent}]
		One considers $T>0$ and $T_1:=\frac{T}{2}$. Then, the previous proposition gives the result.
	\end{proof}
	\subsection{Proof of Theorem \ref{Main_DF}}
	Now, we can write the proof of Theorem \ref{Main_DF}.
	\begin{proof}[Proof of Theorem \ref{Main_DF}]
		We use Theorem \ref{STCLvecteurtangent}. More precisely, we consider $X=\rr^d$ and $E_T=W^{m,\infty}_0(0,T)^2$. Moreover, for all $T>0$,
		$$\mathscr{F}_T:(0,(u,v))\mapsto x(T;(u,v),0).$$ We have to check all the assumptions to obtain the controllability result.
		\begin{enumerate}
			\item[$(A_1)$] This is well-known that the end-point map is regular around the equilibrium.
			\item[$(A_2)$] The differential at $(0,(0,0))$ is given $\mathrm{d}\mathscr{F}_T(0,(0,0))(x_0,(\bar{u},\bar{v}))=Y(T)$ where $Y$ is the solution to the linearized system  
			$$\left\lbrace\begin{array}{crl}Y'(t)&=&Df_0(0)Y(t)+\bar{u}(t)f_1(0)+\bar{v}(t)f_2(0)\\Y(0)&=&x_0
			\end{array}\right..$$ Then, for all $x_0\in\rr^d$, $T\mapsto\mathrm{d}\mathscr{F}_T(0,(0,0))(x_0,(0,0))=e^{TDf_0(0)}x_0$ is continuous on $\rr$ and $\mathrm{d}\mathscr{F}_0(0,(0,0))(x_0,(0,0))=x_0$.
			\item[$(A_3)$] This point is linked to the semi-group property of the solution.
			\item[$(A_4)$] By hypothesis, $H=\text{Im}(\mathscr{F}_T(0,0)\cdot(0,\cdot))$ is a closed subspace (finite dimension), this is the reachable set of the linearized system around $0$. This space doesn't depend on $T$ and its codimension is finite.
			\item[$(A_5)$] Finally, using \eqref{compensationb2bad}, Proposition \ref{vecteurtangent} shows that, for all $j\in\llbracket r+1,d\rrbracket$, $f_{b_j}(0)$ is a small-time $W^{m,\infty}_0$-continuously approximately reachable vector. Then, as a supplementary space of $H$ is given by $\spn\left(\left(f_{b_j}(0)\right)_{r+1\leqslant j\leqslant d}\right)$, this condition is verified.
		\end{enumerate}
		By Theorem \ref{STCLvecteurtangent}, the multi-input control-affine system \eqref{affine-syst} is $W^{m,\infty}_0$-STLC. We obtain the desired result.
	\end{proof}
	\section{Preliminaries on the Schrödinger equation}
	In all the document, we will note $\omega_j:=\lambda_j-\lambda_1$ and $\nu_j:=\lambda_K-\lambda_j$, with $j\geqslant 1$, $K\geqslant 2$. For $\ph\in L^2(0,1)$, $t\in\rr$, one defines
	$$e^{-iAt}\ph=\displaystyle\sum_{k=1}^{+\infty}\langle\ph,\ph_k\rangle e^{-i\lambda_kt}\ph_k.$$
	
	One recalls that $\lambda_j$, $\ph_j$ are defined in \eqref{vp}. For the rest of this paper, $n\geqslant 1$, $p,m\geqslant 0$, $K\geqslant 2$ are fixed integers. Let us discuss about the well-posedness of \eqref{schr}.
	\subsection{Well-posedness}
	\begin{theorem}[Well-posedness] Let $T>0$, $\mu_1,\mu_2$ satisfying $\mathbf{(H)_{reg}}$, $u,v\in H^m_0((0,T),\rr)$, $\psi_0\in H^{2(p+m)+3}_{(0)}(0,1)$ and $f\in H^m_0\left((0,T),H^{2p+3}\cap H^{2p+1}_{(0)}(0,1)\right)$. There exists a unique weak solution to the following equation
		\begin{equation}\label{schro}\left\lbrace\begin{array}{lr} i\partial_t\psi=-\partial_x^2\psi-(u\mu_1+v\mu_2)\psi-f, &  (0,T)\times (0,1),\\
				\psi(\cdot,0)=\psi(\cdot,1)=0, & (0,T),\\
				\psi(0,\cdot)=\psi_0, &(0,1),\end{array}\right.
		\end{equation}
		{i.e.}\ a function $\psi\in\mathcal{C}^m\left([0,T],H^{2p+3}_{(0)}(0,1)\right)$ such that the following equality holds in $H^{2p+3}_{(0)}$ for every $t\in[0,T]$
		\begin{equation}\label{duhamel}
			\psi(t)=e^{-iAt}\psi_0+i\int_0^te^{-iA(t-s)}((u(s)\mu_1+v(s)\mu_2)\psi(s)+f(s))\mathrm{d}s.
		\end{equation}
		Moreover, $\psi(T)\in H^{2(p+m)+3}_{(0)}$ and the following estimates hold: for every $R>0$, there exists $C=C(R,\mu_1,\mu_2,T)>0$ such that, if $\left\|u\right\|_{H^m},\left\|v\right\|_{H^m}<R$, 
		\begin{equation}\label{dep-continue} 
			\left\|\psi\right\|_{\mathcal{C}^m\left([0,T],H^{2p+3}_{(0)}\right)}\leqslant C\left(\left\|\psi_0\right\|_{H^{2(p+m)+3}_{(0)}}+\left\|f\right\|_{H^m\left((0,T),H^{2p+3}\right)}\right),
		\end{equation}
		\begin{equation}\label{dep-continuebis}
			\left\|\psi(T)\right\|_{H^{2(p+m)+3}_{(0)}}\leqslant C\left(\left\|\psi_0\right\|_{H^{2(p+m)+3}_{(0)}}+\left\|f\right\|_{H^m\left((0,T),H^{2p+3}\right)}\right).
		\end{equation}
	\end{theorem}
	See \cite[Theorem $2.1$]{bournissou2021local} for details about the proof.
	\begin{lm}\label{bndope}
		Let $k\in\nn$ and $\mu\in W^{2k,\infty}(0,1)$. Then, $\mu$ is a bounded operator on $H^{2k}_{(0)}(0,1)$ iff all the derivatives of $\mu$ of odd order less than or equal to $2k-3$ vanish at the boundary.
	\end{lm}
	\begin{proof}Let $\mu$ be such a function.
		For all $\ph\in H^{2k}_{(0)}(0,1)$, for all $l\in\{0,\cdots,k-1\}$, one has, thanks to Leibniz formula
		$$\left(\mu\ph\right)^{(2l)}=\sum_{j=0}^{l}\binom{2l}{2j}\mu^{(2j)}\ph^{(2(l-j))}+\sum_{j=0}^{l-1}\binom{2l}{2j+1}\mu^{(2j+1)}\ph^{(2(l-j)-1)}.$$
		The first sum vanishes at $x=0,1$ because $\restriction{\ph^{(2(l-j))}}{\{0,1\}}=0$. As  $\restriction{\mu^{(2j+1)}}{\{0,1\}}=0$ for $2j+1\leqslant 2l-1\leqslant 2k-3$, the second sum is $0$ at $x=0,1$. Conversely, let $\ell\in\{0,\cdots,k-2\}$. Considering $\ph_{\ell}\in H^{2k}_{(0)}(0,1)$ such that $\ph_{\ell}^{(2(k-j)-3)}(0)=\delta_{j,\ell}$ for $0\leqslant j\leqslant k-2$, we obtain
		$$0=\left(\mu\ph_{\ell}\right)^{(2k-2)}(0)=\sum_{j=0}^{k-2}\binom{2k-2}{2j+1}\mu^{(2j+1)}(0)\delta_{j,\ell}=\binom{2k-2}{2\ell+1}\mu^{(2\ell+1)}(0).$$
		Using the same strategy with $x=1$, we obtain the result.
	\end{proof}
	\begin{rmq}\label{stabilité} Let $\mu$ be a function satisfying $(\mathbf{H_{reg}})$. Then, $\ph\mapsto \mu\ph$ is a continuous mapping from $H^{2p+3}\cap H^{2p+1}_{(0)}$ to $H^{2p+3}\cap H^{2p+1}_{(0)}$. However, $\mu^{(2p+1)}(0),\mu^{(2p+1)}(1)$ \textit{a priori} don't vanish, so this application doesn't preserve $H^{2p+3}_{(0)}(0,1)$.  This problem is circumvented by \eqref{dep-continue} and \eqref{dep-continuebis}.
	\end{rmq}
The following statement gives the dependency of the solution with respect to the initial condition. This is the adaptation of Lemma \ref{gronwalldecomp}, with the Schrödinger equation.
	\begin{prop}
		Let $T> 0$, $\mu_1,\mu_2$ satisfying $\mathbf{(H)_{reg}}$ and $\psi_0\in H^{2(p+m)+3}_{(0)}(0,1)$. For all $R>0$, there exists $C=C(T,\mu_1,\mu_2,R)>0$ such that for all $u,v\in H^m_0(0,T)$ with $\left\|u\right\|_{H^m},\left\|v\right\|_{H^m}<R$, one has
		\begin{equation}\begin{gathered}\label{dep-continue2}
			\left\|\psi(T;(u,v),\psi_0+\ph_1)-\psi(T;(u,v),\ph_1)-\psi(T;(0,0),\psi_0)\right\|_{H_{(0)}^{2(p+m)+3}}\\
			\leqslant C\left\|(u,v)\right\|_{H^m}\left\|\psi_0\right\|_{H_{(0)}^{2(p+m)+3}}.\end{gathered}
		\end{equation}
	\end{prop}
	\begin{proof}
		One defines for $T>0$, 
		$$\Lambda:t\in[0,T]\mapsto\psi(t;(u,v),\psi_0+\ph_1)-\psi(t;(u,v),\ph_1)-\psi(t;(0,0),\psi_0).$$
		The function $\Lambda$ is the solution to the following bilinear Schrödinger equation
		$$\left\lbrace\begin{array}{lr}
			i\partial_t\Lambda=-\partial_x^2\Lambda -(u\mu_1+v\mu_2)\Lambda - (u\mu_1+v\mu_2)\psi(\cdot;(0,0),\psi_0),&(0,T)\times (0,1),\\
			\Lambda(\cdot,0)=\Lambda(\cdot,1)=0,&(0,T),\\
			\Lambda(0,\cdot)=0,&(0,1).
		\end{array}\right.$$
		Then, the inequality \eqref{dep-continuebis} applied to this equation gives
		\begin{equation}\label{alg-m=0}\left\|\Lambda(T)\right\|_{H^{2(p+m)+3}_{(0)}}\leqslant C\left\|(u\mu_1+v\mu_2)\psi(\cdot;(0,0),\psi_0)\right\|_{H^m\left((0,T),H^{2p+3}\right)}.\end{equation}
		When $m\geqslant 1$, the Sobolev space $H^m(0,T)$ has an algebra structure and one obtains the following inequality
		$$\left\|\Lambda(T)\right\|_{H^{2(p+m)+3}_{(0)}}\leqslant C\left\|(u,v)\right\|_{H^m}\underset{i\in\llbracket 1,2\rrbracket}\max\left\|\mu_i\psi(\cdot;(0,0),\psi_0)\right\|_{H^m\left((0,T),H^{2p+3}\right)}.$$ 
		Once again, we use the algebra structure of the Sobolev space $H^{2p+3}(0,1)$ to obtain
		$$\left\|\Lambda(T)\right\|_{H^{2(p+m)+3}_{(0)}}\leqslant C\underset{i\in\llbracket 1,2\rrbracket}\max\left\|\mu_i\right\|_{H^{2p+3}}\left\|(u,v)\right\|_{H^m}\left\|\psi(\cdot;(0,0),\psi_0)\right\|_{H^m\left((0,T),H^{2p+3}_{(0)}\right)}.$$ 
		Noticing that $\psi(t;(0,0),\psi_0)=e^{-iAt}\psi_0$,
		$$\left\|\psi(\cdot;(0,0),\psi_0)\right\|_{H^m\left((0,T),H^{2p+3}_{(0)}\right)}\leqslant C\left\|\psi_0\right\|_{H^{2(p+m)+3}_{(0)}}.$$ This leads to the conclusion. Finally, let's consider the case where $m=0$. The following estimation holds
		$$\left\|u\mu_1\psi(\cdot;(0,0),\psi_0)\right\|^2_{L^2((0,T),H^{2p+3})}=\int_0^T|u(t)|^2\left\|\mu_1\psi(t;(0,0),\psi_0)\right\|_{H^{2p+3}}^2\dt.$$
		Using the algebra structure of $H^{2p+3}(0,1)$,
		$$\left\|u\mu_1\psi(\cdot;(0,0),\psi_0)\right\|^2_{L^2((0,T),H^{2p+3})}\leqslant C\left\|\mu_1\right\|^2_{H^{2p+3}} \int_0^T|u(t)|^2\left\|\psi(t;(0,0),\psi_0)\right\|_{H^{2p+3}_{(0)}}^2\dt.$$
		The same estimate can be obtained with the term in $\mu_2v$. Thus, the equation \eqref{alg-m=0} gives
		$$\left\|\Lambda(T)\right\|_{H^{2p+3}_{(0)}}\leqslant C\underset{i\in\llbracket 1,2\rrbracket}\max\left\|\mu_i\right\|_{H^{2p+3}}\left\|(u,v)\right\|_{L^2}\left\|\psi(\cdot;(0,0),\psi_0)\right\|_{L^{\infty}\left((0,T),H^{2p+3}_{(0)}\right)}.$$ 
		The inequality \eqref{dep-continue} (with $m=0$) gives the conclusion.
	\end{proof}
	\subsection{Expansion of the solution}
	We want to make an asymptotic development of the solution $\psi$, with small controls.
	Let $u,v\in H^m_0(0,T)$ be two fixed controls. The first-order term $\Psi\in\mathcal{C}^m\left([0,T],H^{2p+3}_{(0)}(0,1)\right)$ is the solution to the linearized system of \eqref{schr} around the free trajectory $(\psi_1,(u,v)\equiv 0)$, {i.e.}\
	\begin{equation}\hspace{-0.45 cm}\label{schro:lin}\left\lbrace\begin{array}{lr} i\partial_t\Psi=-\partial_x^2\Psi-(u\mu_1+v\mu_2)\psi_1, & (0,T)\times(0,1),\\
			\Psi(\cdot,0)=\Psi(\cdot,1)=0, & (0,T),\\
			\Psi(0,\cdot)=0, &(0,1).\end{array}\right.
	\end{equation}
	Using \eqref{duhamel}, the solution is given by: $\forall t\in [0,T]$, 
	\begin{equation}\label{schro:lin:exp} \Psi(t)=i\sum_{j=1}^{+\infty}\left(\left\langle\mu_1\ph_1,\ph_j\right\rangle\int_0^tu(s)e^{i\omega_js}\mathrm{d}s+\left\langle\mu_2\ph_1,\ph_j\right\rangle\int_0^tv(s)e^{i\omega_js}\mathrm{d}s\right)\psi_j(t).
	\end{equation}
	\begin{defi}[Weak norms of Sobolev spaces] For any integer $k\in\nn^*$, the negative $H^{-k}(0,T)$ space is endowed with the norm: $\forall u\in H^{-k}(0,T)$, 
		$$\left\|u\right\|_{H^{-k}(0,T)}:=|u_1(T)|+\left\|u_k\right\|_{L^2(0,T)}.$$
	\end{defi}
	
	\begin{prop}[Weak estimates] Let $T> 0$, $ k,p\in\nn$ be such that $k\leqslant p$, $\mu_1,\mu_2$ be two functions satisfying $\mathbf{(H)_{reg}}$ and $u,v\in H^m_0(0,T)$. For every $R>0$,
		there exists a constant $C=C(T,\mu_1,\mu_2,R)>0$ such that, if $\left\|u\right\|_{H^m},\left\|v\right\|_{H^m}<R$ and $u_2(T)=\cdots=u_{k+1}(T)=0$, $v_1(T)=\cdots=v_{k+1}(T)=0$, the following estimates hold: for all $l=-(k+1),\cdots,m$, 
		\begin{equation}\label{error:lin}\left\|(\psi(T;(u,v),\ph_1)-\psi_1-\Psi)(T)\right\|_{H^{2(p+l)+3}_{(0)}}\leqslant C\left\|(u,v)\right\|_{H^l}\left\|(u,v)\right\|_{H^m}.\end{equation} 
	\end{prop}
	\begin{proof}
		This point is proved in \cite[Proposition $4.5$]{bournissou2021quadratic}.
	\end{proof}
	In the framework of this article, $\mathbf{(H)_{lin,K,1}}$ implies $\langle \Psi(T), \varphi_K\rangle \equiv 0$, for every controls $(u,v)$. To determine the evolution of $\psi$ in this direction, we refine the approximation and extend the development of the solution to the quadratic term.
	The second-order term $\xi\in\mathcal{C}^m\left([0,T],H^{2p+3}_{(0)}(0,1)\right)$ is the solution to the following system
	\begin{equation}\label{schro:lin}\left\lbrace\begin{array}{lr} i\partial_t\xi=-\partial_x^2\xi-(u\mu_1+v\mu_2)\Psi,& (0,T)\times(0,1),\\
			\xi(\cdot,0)=\xi(\cdot,1)=0, & (0,T),\\
			\xi(0,\cdot)=0, &(0,1).\end{array}\right.
	\end{equation}
	We will sometimes note $\xi(\cdot;(u,v))$ for more precision. The idea is that $\psi(T;(u,v),\ph_1)\simeq \psi_1(T)+\Psi(T)+\xi(T).$ Thus, 
	$$\left\langle\psi(T;(u,v),\ph_1),\psi_K(T)\right\rangle\simeq0+0+\left\langle\xi(T),\psi_K(T)\right\rangle,$$ the first term being $0$ thanks to orthogonality ($K\neq 1$) and the second one by hypothesis $\mathbf{(H)_{lin,K,1}}$.
	For $1\leqslant \alpha,\beta\leqslant 2$, one defines
	\begin{equation}\label{noyaux2}
		h_{\alpha,\beta}:(t,s)\in[0,T]^2\mapsto-\sum_{j=1}^{+\infty}\langle\mu_{\alpha}\ph_1,\ph_j\rangle\langle\ph_j,\mu_{\beta}\ph_K\rangle e^{i\nu_jt+i\omega_js}.
	\end{equation}
	We finally use the notation, for $u,v\in L^2((0,T),\rr)$,
	\begin{equation}\begin{gathered}\label{defft}\mathcal{F}_T^1(u):=\int_0^Tu(t)\left(\int_0^th_{1,1}(t,s)u(s)\ds\right)\dt, \\ \mathcal{F}_T^2(v):=\int_0^Tv(t)\left(\int_0^th_{2,2}(t,s)v(s)\ds\right)\dt,\end{gathered}
	\end{equation}
	\begin{equation}\label{defgt}\mathcal{G}_T(u,v):=\int_0^Tu(t)\left(\int_0^th_{2,1}(t,s)v(s)\ds\right)\dt+\\\int_0^Tv(t)\left(\int_0^th_{1,2}(t,s)u(s)\ds\right)\dt.
	\end{equation}
	The formulas \eqref{duhamel} and \eqref{schro:lin:exp} lead to
	\begin{equation}\label{schro:quad}
		\left\langle\xi(T),\psi_K(T)\right\rangle=\mathcal{F}_T^1(u)+\mathcal{G}_T(u,v)+\mathcal{F}_T^2(v).
	\end{equation}
	\begin{lm}\label{lmreste}
		Assume that $\mu_1,\mu_2$ satisfy $\mathbf{(H)_{reg}}$. Then, as $\left\|(u,v)\right\|_{L^2}\to 0$,
		\begin{equation}\label{reste}
			\left\| \psi(\cdot;(u,v),\ph_1)-\psi_1-\Psi-\xi\right\|_{L^{\infty}\left((0,T),H^{2p+3}_{(0)}\right)}= \mathcal{O}\left(\left\|u\right\|_{L^2(0,T)}^3+\left\|v\right\|_{L^2(0,T)}^3\right).
		\end{equation}
	\end{lm}
	\begin{proof}
		\textit{First step: linear remainder}. One defines $\Lambda:t\in[0,T]\mapsto \psi(t;(u,v),\ph_1)-\psi_1(t)$. The function $\Lambda$ is solution to the following PDE
		$$\left\lbrace\begin{array}{lr}
			i\partial_t\Lambda=-\partial_x^2\Lambda -(u\mu_1+v\mu_2)\Lambda - (u\mu_1+v\mu_2)\psi_1, &(0,T)\times (0,1),\\
			\Lambda(\cdot,0)=\Lambda(\cdot,1)=0, &(0,T),\\
			\Lambda(0,\cdot)=0, &(0,1).
		\end{array}\right.$$ 
		As $\psi_1(t)\in H^{2p+3}_{(0)}(0,1)$ and $\mu_1,\mu_2$ verify $\mathbf{(H)_{reg}}$, Remark \ref{stabilité} ensures that $(u(t)\mu_1+v(t)\mu_2)\psi_1(t)\in H^{2p+3}\cap H^{2p+1}_{(0)}(0,1)$. 
		Then, we can use the regularization inequality given by \eqref{dep-continue} with $m=0$ to obtain
		\begin{equation}\label{reste-fond}
			\left\|\Lambda\right\|_{L^{\infty}\left((0,T),H^{2p+3}_{(0)}\right)}\leqslant C\left\|(u\mu_1+v\mu_2)\psi_1\right\|_{L^2\left((0,T),H^{2p+3}\right)}.
		\end{equation}
		Note that 
		$$\left\|u\mu_1\psi_1\right\|_{L^2\left((0,T),H^{2p+3}\right)}^2=\int_0^T|u(t)|^2\left\|\mu_1\ph_1e^{-i\lambda_1t}\right\|^2_{H^{2p+3}}\dt=\left\|\mu_1\ph_1\right\|_{H^{2p+3}}^2\left\|u\right\|_{L^2(0,T)}^2.$$
		We can obtain the same estimate for $\left\|v\mu_2\psi_1\right\|_{L^2\left((0,T),H^{2p+3}\right)}$. Then, thanks to \eqref{reste-fond}, one gets
		$$\left\|\psi(\cdot;(u,v),\ph_1)-\psi_1\right\|_{L^{\infty}\left((0,T),H^{2p+3}_{(0)}\right)}=\mathcal{O}\left(\left\|u\right\|_{L^2(0,T)}+\left\|v\right\|_{L^2(0,T)}\right).$$
		\textit{Second step: quadratic remainder}. 
		Using the same strategy, one defines $\Lambda:t\in[0,T]\mapsto \psi(t;(u,v),\ph_1)-\psi_1(t)-\Psi(t)$. The function $\Lambda$ is solution to the following PDE
		$$\left\lbrace\begin{array}{lr}
			i\partial_t\Lambda=-\partial_x^2\Lambda -(u\mu_1+v\mu_2)\left(\psi(\cdot;(u,v),\ph_1)-\psi_1\right),& (0,T)\times(0,1), \\
			\Lambda(\cdot,0)=\Lambda(\cdot,1)=0, &(0,T),\\
			\Lambda(0,\cdot)=0, &(0,1).
		\end{array}\right.$$
		By the same way, one can use the estimate \eqref{dep-continue} with $m=0$ in this equation to obtain
		$$\left\|\Lambda\right\|_{L^{\infty}\left((0,T),H^{2p+3}_{(0)}\right)}\leqslant C\left\|(u\mu_1+v\mu_2)\left(\psi(\cdot;(u,v),\ph_1)-\psi_1\right)\right\|_{L^2((0,T),H^{2p+3})}.$$
		Computing this norm and using the algebra structure of the Sobolev space $H^{2p+3}(0,1)$,
		\begin{equation}\begin{gathered}\label{alg-sob}\left\|u\mu_1\left(\psi(\cdot;(u,v),\ph_1)-\psi_1\right)\right\|_{L^2\left((0,T),H^{2p+3}\right)}\\\leqslant C\left\|\psi(\cdot;(u,v),\ph_1)-\psi_1\right\|_{L^{\infty}\left((0,T),H^{2p+3}_{(0)}\right)}\left\|\mu_1\right\|_{H^{2p+3}} \left\|u\right\|_{L^2(0,T)}.\end{gathered}\end{equation}
		Then, using the inequality proved in the first step in \eqref{alg-sob}, we obtain 
		$$\left\|\psi(\cdot;(u,v),\ph_1)-\psi_1-\Psi\right\|_{L^{\infty}\left((0,T),H^{2p+3}_{(0)}\right)}=\mathcal{O}\left(\left\|u\right\|_{L^2(0,T)}^2+\left\|v\right\|_{L^2(0,T)}^2\right).$$
		\textit{Third step: cubic remainder}. Once again, one defines $\Lambda:t\in[0,T]\mapsto \psi(t;(u,v),\ph_1)-\psi_1(t)-\Psi(t)-\xi(t)$. The function $\Lambda$ is solution to the following PDE
		$$\left\lbrace\begin{array}{lr}
			i\partial_t\Lambda=-\partial_x^2\Lambda -(u\mu_1+v\mu_2)\left(\psi(\cdot;(u,v),\ph_1)-\psi_1-\Psi\right),& (0,T)\times(0,1), \\
			\Lambda(\cdot,0)=\Lambda(\cdot,1)=0, &(0,T),\\
			\Lambda(0,\cdot)=0, &(0,1).
		\end{array}\right.$$
		One uses the estimate \eqref{dep-continue} with $m=0$ in this equation to obtain
		$$\left\|\Lambda\right\|_{L^{\infty}\left((0,T),H^{2p+3}_{(0)}\right)}\leqslant C\left\|(u\mu_1+v\mu_2)\left(\psi(\cdot;(u,v),\ph_1)-\psi_1-\Psi\right)\right\|_{L^2((0,T),H^{2p+3})}.$$
		Using the algebra structure of the Sobolev space $H^{2p+3}(0,1)$,
		\begin{equation}\begin{gathered}\label{alg-sob2}\left\|u\mu_1\left(\psi(\cdot;(u,v),\ph_1)-\psi_1-\Psi\right)\right\|_{L^2\left((0,T),H^{2p+3}\right)}\\\leqslant C\left\|\psi(\cdot;(u,v),\ph_1)-\psi_1-\Psi\right\|_{L^{\infty}\left((0,T),H^{2p+3}_{(0)}\right)}\left\|\mu_1\right\|_{H^{2p+3}} \left\|u\right\|_{L^2(0,T)}.\end{gathered}\end{equation}
		Then, using the inequality proved in the second step in \eqref{alg-sob2}, we obtain 
		$$\left\|\psi(\cdot;(u,v),\ph_1)-\psi_1-\Psi-\xi\right\|_{L^{\infty}\left((0,T),H^{2p+3}_{(0)}\right)}=\mathcal{O}\left(\left\|u\right\|_{L^2(0,T)}^3+\left\|v\right\|_{L^2(0,T)}^3\right).$$
	\end{proof}
	\subsection{Control in projection}
	The following result is adapted from a result proved by Bournissou in \cite{bournissou2021local}. This theorem gives the controllability in projection on $\mathcal{H}$, defined in \eqref{h}. More precisely, the statement is the following.
	\begin{theorem}\label{control-projection}
		Let $p,m,k\geqslant 0$, $K\geqslant 2$, such that $k\leqslant p$, $\mu_1,\mu_2$ functions satisfying $\mathbf{(H)_{reg}}$ and $\mathbf{(H)_{lin,K,2}}$. Then, the bilinear Schrödinger equation \eqref{schr} is $H^m_0-$STLC in projection on $\mathcal{H}$ around the ground state in $H^{2(p+m)+3}_{(0)}$. More precisely, for any $T_1<T$, there exist $\delta,C>0$ and a $\mathcal{C}^1$ map $\Gamma_{T_1,T}: \mathcal{V}_{T_1}\times\mathcal{V}_{T}\to H^m_0((T_1,T),\rr)^2$, where 
		$$\mathcal{V}_{T_1}:=\left\lbrace \psi_0\in\mathcal{S}\cap H^{2(p+m)+3}_{(0)}(0,1),\quad \left\|\psi_0-\psi_1(T_1)\right\|_{H^{2(p+m)+3}_{(0)}}<\delta\right\rbrace$$
		$$\mathcal{V}_{T}:=\left\lbrace \psi_f\in\mathcal{H}\cap H^{2(p+m)+3}_{(0)}(0,1),\quad \left\|\psi_f-\psi_1(T)\right\|_{H^{2(p+m)+3}_{(0)}}<\delta\right\rbrace$$
		such that, $\Gamma_{T_1,T}(\psi_1(T_1),\psi_1(T))=(0,0)$ and for every $\psi_0,\psi_f\in\mathcal{V}_{T_1}\times\mathcal{V}_{T}$ , the solution to \eqref{schr} on $(T_1,T)$ with initial data $\psi_0$ in $t=T_1$ and controls $u,v:=\Gamma_{T_1,T}(\psi_0,\psi_f)$ satisfies $$\mathbb{P}_{\mathcal{H}}\left(\psi(T;\Gamma_{T_1,T}(\psi_0,\psi_f),\psi_0)\right)=\psi_f,$$ with the following boundary conditions 
		$$u_2(T)=\cdots=u_{k+1}(T)=0, \quad \text{and} \quad v_2(T)\cdots=v_{k+1}(T)=0.$$
		Finally, for all $l\in\left\lbrace -(k+1),\cdots,m\right\rbrace$,
		\begin{equation}\label{estim-simult}
			\left\|u\right\|_{H^l}\leqslant C\left(\left\|\psi_0-\psi_1(T_1)\right\|_{H^{2(p+l)+3}_{(0)}}+\left\|\psi_f-\psi_1(T)\right\|_{H^{2(p+l)+3}_{(0)}}\right).
		\end{equation}
		Moreover, if $\langle\mu_1\ph_1,\ph_1\rangle\neq 0$, one can ensure $v_1(T)=0$.
	\end{theorem}
	\begin{rmq}
		The hypothesis $\langle\mu_1\ph_1,\ph_1\rangle\neq 0$ allows us to solve the moment problems (given by the linear test) in the direction $\ph_1$ with the control $u$. Indeed, since $\omega_1=0$,
		$$\langle\Psi(T),\psi_1(T)\rangle=i\langle\mu_1\ph_1,\ph_1\rangle u_1(T)+i\langle\mu_2\ph_1,\ph_1\rangle v_1(T).$$
		Thus, one can impose the border condition $v_1(T)=0$. This property is also verified by the nonlinear system thanks to the iteration on the Banach fixed-point theorem.
	\end{rmq}
	\section{Proof of the main theorem}\label{proofmain}
	The proof of the main theorem is divided in the following steps.
	\begin{enumerate}
		\item[1.] In the first step, we use the quadratic expansion of the solution in the direction $\psi_K(T)$ (given by Proposition \ref{quad-fini}) to move along the real direction lost $i^n\ph_K$.
		\item[2.] In the second step, we correct linearly the system in the other directions $\mathcal{H}$, thanks to Theorem \ref{control-projection}. Then, we examine the impact of this second step on the correction made in the first one. A priori, the second step can destroy the work of the first and we use weak estimates to prove that this is not the case.
		\item[3.] We use a trick introduced by Kawski in \cite{hermeskawski} and executed by Bournissou in \cite{bournissou2022smalltime} to move along the other real direction lost $i^{n+1}\ph_K$.
		\item[4.] We apply Theorem \ref{STCLvecteurtangent} (based on the Brouwer fixed-point theorem) to obtain STLC.
	\end{enumerate}
	
	In order to extract the leading terms of the dynamic of the system (as in the finite-dimensional case, with the Magnus-type representation formula), we manipulate the expression given by \eqref{schro:quad}. This is the purpose of the following subsection. 
	\subsection{Asymptotic estimates on the quadratic term of the solution}
	\begin{nota}
		For  $T>0$, $j\geqslant 1$ and $u,v\in L^2(0,T)$, we use the notation $$I^j_T(u,v):=\int_0^Tu(t)e^{i\nu_jt}\left(\int_0^tv(s)e^{i\omega_js}\ds\right)\dt, \qquad \mathcal{I}_T(u,v):=\int_0^Tu(t)v(t)e^{i\omega_Kt}\dt.$$
	\end{nota}
	We define $q:=\lfloor\frac n2\rfloor$. This notation will be used throughout the rest of the article. As a remainder, the sequence $(c_j)_{j\geqslant 1}$ is defined in Section \ref{dipolar} as $c_j=\left\langle\mu_1\ph_1,\ph_j\right\rangle\left\langle \mu_1\ph_j,\ph_K\right\rangle$. The goal of this subsection is to prove the following proposition.
	\begin{prop}\label{quad-fini} Let $u,v\in L^2((0,T),\rr)$ be such that $u_i(T)=v_i(T)=0$ for $1\leqslant i\leqslant q+1$.  Furthermore, assume that $\mathbf{(H)_{reg}}$, $\mathbf{(H)_{quad,K,1}}$, $\mathbf{(H)_{quad,K,2}}$ and $\mathbf{(H)_{quad,K,3}}$ hold. Then,
		\begin{equation}\begin{gathered}\label{equationclef}\left\langle\xi(T),\psi_K(T)\right\rangle=-iA_{q+1}^1\int_0^Tu_{q+1}^2(t)e^{i\omega_Kt}\dt+i^n\gamma_n\int_0^Tu_{q+1}(t)v_{n-q}(t)e^{i\omega_Kt}\dt\\-iA_{q+1}^2\int_0^Tv_{q+1}^2(t)e^{i\omega_Kt}\dt+R_{T,1}^{q+1}(u)+\rho_T^{q+1}(u,v)+R_{T,2}^{q+1}(v).\end{gathered}\end{equation}
		with 
		$$R_{T,1}^{q+1}(u):=\displaystyle(-1)^q\sum_{j=1}^{+\infty}c_j\omega_j^{q+1}\nu_j^{q+1}I^j_T(u_{q+1},u_{q+1}),$$ $$R_{T,2}^{q+1}(v):=\displaystyle(-1)^q\sum_{j=1}^{+\infty}\tilde{c_j}\omega_j^{q+1}\nu_j^{q+1}I^j_T(v_{q+1},v_{q+1}),$$
		$$\rho_T^{q+1}(u,v):=(-i)^{n-1}\left(\sum_{j=1}^{+\infty}d_j\omega_j^{n-q}\nu_j^{q+1}I^j_T(u_{q+1},v_{n-q})+\sum_{j=1}^{+\infty}\tilde{d_j}\omega_j^{q+1}\nu_j^{n-q}I^j_T(v_{n-q},u_{q+1})\right).$$
	\end{prop}
	The proof is divided into several lemmas. The first one, strongly inspired by \cite[Proposition $5.1$]{bournissou2021quadratic}, focuses on the first term on the right-hand side of the decomposition given in \eqref{schro:quad}.
	\begin{lm}\label{propdvlpft}
		Let $T>0$, $0\leqslant l\leqslant q$ and assume that $u\in L^2((0,T),\rr)$ such that $u_i(T)=0$ for $2\leqslant i\leqslant l+1$. If $\mathbf{(H)_{reg}}$ holds, then 
		\begin{equation}\begin{gathered}\label{dev1}
			\mathcal{F}_T^1(u)=-u_1(T)\sum_{j=1}^{+\infty}c_je^{i\nu_jT}\left(\int_0^Tu(s)e^{i\omega_js}\ds\right)+\frac{u_1(T)^2}{2}e^{i\omega_KT}\sum_{j=1}^{+\infty}c_j\\-i\sum_{k=1}^{l+1}A_k^1\int_0^Tu_k^2(t)e^{i\omega_Kt}\dt+R_{T,1}^{l+1}(u).
	\end{gathered}	\end{equation}
	\end{lm}
	\begin{proof} We prove the statement by a finite induction on $l\in\llbracket 0,q\rrbracket$.
		If the formula is true for a fixed $l\leqslant q-1$, then, for all $u\in L^2((0,1),\rr)$ such that $u_i(T)=0$ for $2\leqslant i\leqslant l+2$,  an integration by parts gives
		\begin{equation}\label{u-u1}(-1)^lR_{T,1}^{l+1}(u)=A+B,\end{equation}
		with
		$$A=-i\sum_{j=1}^{+\infty}c_j\omega_j^{l+1}\nu_j^{l+2}I^j_T(u_{l+2},u_{l+1}), \hspace{1 cm} B:=-\sum_{j=1}^{+\infty}c_j\omega_j^{l+1}\nu_j^{l+1}\mathcal{I}_T(u_{l+2},u_{l+1}).$$
		Then, with one integration by parts, we obtain
		\begin{equation}\label{u-u2}A=-i\sum_{j=1}^{+\infty}c_j\omega_j^{l+1}\nu_j^{l+2}\int_0^Tu_{l+2}^2(t)e^{i\omega_Kt}\dt+(-1)^lR_{T,1}^{l+2}(u).
		\end{equation}
		\begin{equation}\label{u-u3}B=\frac{i\omega_K}{2}\sum_{j=1}^{+\infty}c_j\omega_j^{l+1}\nu_j^{l+1}\int_0^Tu_{l+2}^2(t)e^{i\omega_Kt}\dt.\end{equation}
		Using \eqref{u-u1}, \eqref{u-u2}, \eqref{u-u3} and the induction hypothesis, we conclude. The initialization is proved by the same manipulations, but there are boundary terms. The starting point is the equality given by \eqref{noyaux2} and \eqref{defft},
		$$\mathcal{F}_T^1(u)=-\sum_{j=1}^{+\infty}c_jI^j_T(u,u).$$ 	All series converge thanks to the hypothesis $\lfloor\frac{n}{2}\rfloor\leqslant p$ (see Remark \ref{asymptotique}). 
	\end{proof}
	One recalls that, for any integer $k\in\nn^*$, the negative $H^{-k}(0,T)$ space is endowed with the norm: $\forall u\in H^{-k}(0,T)$, 
	$\left\|u\right\|_{H^{-k}(0,T)}:=|u_1(T)|+\left\|u_k\right\|_{L^2(0,T)}.$ Using the expression of Lemma \ref{propdvlpft}, one can obtain the following estimate.
	\begin{crl} Let $T>0$ and assume that $u\in L^2((0,T),\rr)$ such that $u_i(T)=0$ for $2\leqslant i\leqslant q+1$. If $\mathbf{(H)_{reg}}$ and $\mathbf{(H)_{quad,K,1}}$ hold, then 
		\begin{equation}\label{ineq1}
			\left|\mathcal{F}_T^1(u)\right| =\mathcal{O}\left(\left|u_1(T)\right|^2+\left\|u_{q+1}\right\|_{L^2}^2\right)=\mathcal{O}\left(\left\|u\right\|_{H^{-(q+1)}}^2\right).
		\end{equation}
	\end{crl}
	\begin{proof}
		Using Lemma \ref{propdvlpft} with $l=q$, we obtain
		\begin{equation*}\begin{gathered}		\mathcal{F}_T^1(u)=-u_1(T)\sum_{j=1}^{+\infty}c_je^{i\nu_jT}\left(\int_0^Tu(s)e^{i\omega_js}\ds\right)+\frac{u_1(T)^2}{2}e^{i\omega_KT}\sum_{j=1}^{+\infty}c_j\\-iA_{q+1}^1\int_0^Tu_{q+1}^2(t)e^{i\omega_Kt}\dt+R_{T,1}^{q+1}(u).
	\end{gathered}	\end{equation*}
		Note that $A_{q+1}^1=0$ when $n$ is odd, but not necessarily when $n$ is even. We can manipulate the first term: using integrations by parts,
		$$\sum_{j=1}^{+\infty}c_je^{i\nu_jT}\left(\int_0^Tu(s)e^{i\omega_js}\ds\right)=\sum_{j=1}^{+\infty}c_je^{i\nu_jT}\left(u_1(T)e^{i\omega_jT}+(-i\omega_j)^{q+1}\int_0^Tu_{q+1}(s)e^{i\omega_js}\ds\right).$$
		Finally, the term $R_{T,1}^{q+1}(u)$ is estimated by $\left\|u\right\|_{H^{-(q+1)}}^2$ because, for all $j\geqslant 1$, 
		$$\left| I_T^j(u_{q+1},u_{q+1})\right|\leqslant \left\|u_{q+1}\right\|_{L^1}^2.$$ 
		This formula leads to the result. 
	\end{proof}
	Now, we can follow the same approach with the crossed terms. More precisely, we focus on the second term on the right-hand side of the decomposition given by \eqref{schro:quad}.
	\begin{lm}\label{propdvlpgt} 
		Let $T>0$ and assume that $u,v\in L^2((0,T),\rr)$ such that $u_i(T)=0$ for $2\leqslant i\leqslant q+1$ and  $v_i(T)=0$ for $1\leqslant i\leqslant q+1$. If $\mathbf{(H)_{reg}}$ holds, then 
		\begin{equation}\begin{gathered}\label{dev}
			\mathcal{G}_T(u,v)=\sum_{k=1}^{q}(-1)^k\gamma_{2k}\mathcal{I}_T(u_{k+1},v_k)+i\sum_{k=0}^{n-q-1}(-1)^k\gamma_{2k+1}\mathcal{I}_T(u_{k+1},v_{k+1})\\-u_1(T)\sum_{j=1}^{+\infty}d_je^{i\nu_jT}\left(\int_0^Tv(s)e^{i\omega_js}\ds\right)+\rho_T^{q+1}(u,v).\end{gathered}
		\end{equation}
	\end{lm}
	\begin{proof} 
		We deal with the case where $n$ is odd. Thanks to \eqref{noyaux2} and \eqref{defgt}, we obtain $\mathcal{G}_T(u,v)=-(A+B)$ with
		$$A=\sum_{j=1}^{+\infty}d_jI^j_T(u,v), \qquad\qquad B=\sum_{j=1}^{+\infty}\tilde{d_j}I^j_T(v,u).$$
		First, we prove by a (finite) induction on $0\leqslant l\leqslant q$ that: for all $u,v\in L^2((0,T),\rr)$ such that $u_i(T)=0$ for $2\leqslant i\leqslant l+1$,
		\begin{equation}\begin{gathered}\label{A}
			A=u_1(T)\sum_{j=1}^{+\infty}d_je^{i\nu_jT}\left(\int_0^Tv(s)e^{i\omega_js}\ds\right)+(-1)^{l+1}\sum_{j=1}^{+\infty}d_j\omega_j^{l+1}\nu_j^{l+1}I^j_T(u_{l+1},v_{l+1})\\+\sum_{k=0}^l(-1)^{k+1}\sum_{j=1}^{+\infty}d_j\omega_j^k\nu_j^k\mathcal{I}_T(u_{k+1},v_k)+i\sum_{k=0}^l(-1)^{k+1}\sum_{j=1}^{+\infty}d_j\omega_j^k\nu_j^{k+1}\mathcal{I}_T(u_{k+1},v_{k+1}).
		\end{gathered}\end{equation}
		For $l=0$, with an integration by parts on $A$, we get
		$$A=u_1(T)\sum_{j=1}^{+\infty}d_je^{i\nu_jT}\left(\int_0^Tv(s)e^{i\omega_js}\ds\right)-\sum_{j=1}^{+\infty}d_j\mathcal{I}_T(u_1,v)-i\sum_{j=1}^{+\infty}d_j\nu_jI^j_T(u_1,v).$$
		Another integration on $v$ by parts gives
		\begin{equation*}\begin{gathered}
			A=u_1(T)\sum_{j=1}^{+\infty}d_je^{i\nu_jT}\left(\int_0^Tv(s)e^{i\omega_js}\ds\right)-\sum_{j=1}^{+\infty}d_j\omega_j\nu_jI^j_T(u_1,v_1)\\-\sum_{j=1}^{+\infty}d_j\mathcal{I}_T(u_1,v)-i\sum_{j=1}^{+\infty}d_j\nu_j\mathcal{I}_T(u_1,v_1).
		\end{gathered}\end{equation*}
		For the induction step, it suffices to do two integrations by parts in the formula given by the induction hypothesis. Similarly, we prove that for $0\leqslant l\leqslant q$, for all $u,v\in L^2((0,T),\rr)$ such that $v_i(T)=0$ for $1\leqslant i\leqslant l+1$,
		\begin{equation}\begin{gathered}\label{B}\hspace{-0.14 cm}
			B=\sum_{k=0}^l(-1)^k\sum_{j=1}^{+\infty}\tilde{d_j}\omega_j^k\nu_j^k\mathcal{I}_T(u_{k+1},v_k)+i\sum_{k=0}^l(-1)^k\sum_{j=1}^{+\infty}\tilde{d_j}\omega_j^{k+1}\nu_j^k\mathcal{I}_T(u_{k+1},v_{k+1})\\+(-1)^{l+1}\sum_{j=1}^{+\infty}\tilde{d_j}\omega_j^{l+1}\nu_j^{l+1}I^j_T(v_{l+1},u_{l+1}).
			\end{gathered}	\end{equation}
		Using $\mathcal{G}_T(u,v)=-(A+B)$, \eqref{A} and \eqref{B} with $l=q$, we conclude. To obtain the result in the even case, we manipulate the expression in the same way, with $v_q$ instead of $v_{q+1}$. Once again, all the series converge thanks to the hypothesis $\lfloor\frac{n}{2}\rfloor\leqslant p$ (see Remark \ref{asymptotique}).
	\end{proof}
	Using the expression given by Lemma \ref{propdvlpgt}, we get the following estimate.
	\begin{crl}
		Let $T>0$, $u,v\in L^2((0,T),\rr)$ be such that, $u_i(T)=0$ for $2\leqslant i\leqslant q+1$, $v_i(T)=0$ for $1\leqslant i\leqslant q+1$.  If $\mathbf{(H)_{reg}}$ and $\mathbf{(H)_{quad,K,3}}$ hold,
		\begin{equation}\label{ineq3}\left|\mathcal{G}_T(u,v)\right|=\mathcal{O}\left(\left\|u\right\|_{H^{-(q+1)}}\left\|v\right\|_{H^{-(n-q)}}\right).\end{equation}
	\end{crl}
	\begin{proof}
		We deal with the case where $n$ is odd, so $n=2q+1$. First, using Lemma \ref{propdvlpgt} and $\mathbf{(H)_{quad,K,3}}$, we obtain
		\begin{equation*}\begin{gathered}
			\mathcal{G}_T(u,v)=-u_1(T)\sum_{j=1}^{+\infty}d_je^{i\nu_jT}(-i\omega_j)^{q+1}\left(\int_0^Tv_{q+1}(s)e^{i\omega_js}\ds\right)+\rho_T^{q+1}(u,v)\\+i(-1)^q\gamma_n\mathcal{I}_T(u_{q+1},v_{q+1}).\end{gathered}
		\end{equation*}
		This equality leads to the result, by definition of the norm in a negative Sobolev space, as 
		$$\left|\rho_T^{q+1}(u,v)\right|\leqslant C\left\|u\right\|_{H^{-(q+1)}}\left\|v\right\|_{H^{-(q+1)}},\quad \left|\mathcal{I}_T(u_{q+1},v_{q+1})\right|\leqslant C\left\|u\right\|_{H^{-(q+1)}}\left\|v\right\|_{H^{-(q+1)}}.$$ We can prove this statement in the same way when $n$ is even.
	\end{proof}
With the two Lemmas \ref{propdvlpft} and \ref{propdvlpgt}, we are now able to prove the main result of this subsection.
	\begin{proof}[Proof of Proposition \ref{quad-fini}]
		Once again, we deal with the case where $n$ is odd. Then, using the hypothesis $\mathbf{(H)_{quad,K,1}}$, Lemma \ref{propdvlpft} with $l=q$ gives
		\begin{equation}\label{ftu}\mathcal{F}_T^1(u)=R_{T,1}^{q+1}(u).
		\end{equation}
		Similarly, the hypothesis $\mathbf{(H)_{quad,K,2}}$ leads to
		\begin{equation}\label{ftv}\mathcal{F}_T^2(v)=R_{T,2}^{q+1}(v).
		\end{equation}
		Using Lemma \ref{propdvlpgt} and $\mathbf{(H)_{quad,K,3}}$, one has
		\begin{equation}\label{gtuv}\mathcal{G}_T(u,v)=i(-1)^q\gamma_n\int_0^Tu_{q+1}(t)v_{q+1}(t)e^{i\omega_Kt}\dt+\rho_T^{q+1}(u,v).
		\end{equation}
		Using the equations \eqref{schro:quad}, \eqref{ftu}, \eqref{ftv} and \eqref{gtuv}, we obtain \eqref{equationclef}.
	\end{proof}
	\subsection{A concatenation lemma}
	For all the rest of the document, we assume that $\mathbf{(H)_{reg}}$, $\mathbf{(H)_{lin,K,1}}$,  $\mathbf{(H)_{lin,K,2}}$,  $\mathbf{(H)_{quad,K,1}}$,  $\mathbf{(H)_{quad,K,2}}$,  $\mathbf{(H)_{quad,K,3}}$ and  $\mathbf{(H)_{quad,K,4}}$ hold. The goal of the following lemma is to examine the interaction between the first and the second step. 
	\begin{lm}[A composition lemma] Assume that $\langle\mu_1\ph_1,\ph_1\rangle\neq 0$.
		Let $0<T_1<T$, $u,v\in L^2(0,T_1)$ be such that $u_i(T_1)=v_i(T_1)=0$ for $1\leqslant i\leqslant q+1$. Let $(\tilde{u},\tilde{v})=\Gamma_{T_1,T}(\psi(T_1;(u,v),\ph_1),\psi_1(T))$, where $\Gamma$ is defined in Theorem \ref{control-projection}. Finally, one defines $U=u\mathbb{1}_{(0,T_1)}+\tilde{u}\mathbb{1}_{(T_1,T)}$ and $V=v\mathbb{1}_{(0,T_1)}+\tilde{v}\mathbb{1}_{(T_1,T)}$. Then,
		\begin{equation}\begin{gathered}
			\label{oddeven}
			\left\langle\xi(T;(U,V)),\psi_K(T)\right\rangle-\left\langle\xi(T_1;(u,v)),\psi_K(T_1)\right\rangle=\mathcal{O}\Big(\left\|(\tilde{u},\tilde{v})\right\|_{H^{-(q+1)}}^2\\+\left\|(u,v)\right\|_{H^{-(q+1)}}\left\|(\tilde{u},\tilde{v})\right\|_{H^{-(q+1)}}+\left\|\tilde{u}\right\|_{H^{-(q+1)}}\left\|\tilde{v}\right\|_{H^{-(n-q)}}\Big).\end{gathered}
		\end{equation}
	\end{lm}
	\begin{proof}Using the formula \eqref{schro:quad}, we get
		\begin{equation}\label{lemmecompo}	\hspace{-0.1 cm}\langle\xi(T;(U,V)),\psi_K(T)\rangle-\langle\xi(T_1;(u,v)),\psi_K(T_1)\rangle=\langle\xi(T;(\tilde{u},\tilde{v})),\psi_K(T)\rangle+G_{T_1,T}\left(U,V\right),
		\end{equation}
		where $G_{T_1,T}$ is the bilinear form given by
		\begin{equation}\begin{gathered}\label{lemmecompo2}
			G_{T_1,T}(U,V)=\int_{T_1}^T\tilde{u}(t)\left(\int_0^{T_1}h_{1,1}(t,s)u(s)\ds+\int_0^{T_1}h_{2,1}(t,s)v(s)\ds\right)\dt\\+\int_{T_1}^T\tilde{v}(t)\left(\int_0^{T_1}h_{1,2}(t,s)u(s)\ds+\int_0^{T_1}h_{2,2}(t,s)v(s)\ds\right)\dt.\end{gathered}
		\end{equation}
		Each term of \eqref{lemmecompo2} can be written as a sum of product of two integrals, for example
		\begin{equation*}\int_{T_1}^T\tilde{u}(t)\left(\int_0^{T_1}h_{1,1}(t,s)u(s)\ds\right)\dt=-\sum_{j=1}^{+\infty}c_j\left(\int_{T_1}^T\tilde{u}(t)e^{i\nu_jt}\dt\right)\left(\int_0^{T_1}u(s)e^{i\omega_js}\ds\right).\end{equation*}
		Using integrations par parts, one can write the expression as
		$$-\sum_{j=1}^{+\infty}c_j\left(\tilde{u}_1(T)e^{i\nu_jT}+(-i\nu_j)^{q+1}\int_{T_1}^T\tilde{u}_{q+1}(t)e^{i\nu_jt}\dt\right)(-i\omega_j)^{q+1}\int_0^{T_1}u_{q+1}(s)e^{i\omega_js}\ds.$$
		By Cauchy--Schwarz's inequality,
		$$\left|\int_{T_1}^T\tilde{u}(t)\left(\int_0^{T_1}h_{1,1}(t,s)u(s)\ds\right)\dt\right|\leqslant C\left\|(u,v)\right\|_{H^{-(q+1)}}\left\|(\tilde{u},\tilde{v})\right\|_{H^{-(q+1)}}.$$
		We obtain the same estimation with the other terms of \eqref{lemmecompo2}. Thus, 
		\begin{equation}\label{lemmecompo7}
			G_{T_1,T}(U,V)=\mathcal{O}\left(\left\|(u,v)\right\|_{H^{-(q+1)}}\left\|(\tilde{u},\tilde{v})\right\|_{H^{-(q+1)}}\right).
		\end{equation}
		Moreover,
		\begin{equation}\label{lemmecompo8}\langle\xi(T;(\tilde{u},\tilde{v})),\psi_K(T)\rangle=\mathcal{F}_T^1(\tilde{u})+\mathcal{F}_T^2(\tilde{v})+\mathcal{G}_T(\tilde{u},\tilde{v}).
		\end{equation}
		Using \eqref{lemmecompo}, \eqref{lemmecompo7},  \eqref{lemmecompo8}, \eqref{ineq1} and \eqref{ineq3} (thanks to the boundary conditions given by the STLC in projection Theorem \ref{control-projection}), we obtain the result. The assumption $\lfloor\frac{n}{2}\rfloor\leqslant p$ gives the convergence of all the series written (see Remark \ref{asymptotique}).
	\end{proof}
	\subsection{Motion along $\ph_K$ and $i\ph_K$}
	The next proposition implements step $1$ of the proof strategy explained at the beginning of Section \ref{proofmain}.
	\begin{prop} We denote $s_n:=\frac{1}{4(n+m+2)}$ and $\alpha_n:=\frac{3}{8}+\frac{m}{8(n+m+2)}$ 
		For all $T_1>0$, there exist $C,\rho>0$ and a continuous map $z\mapsto (u_z,v_z)$ from $\rr$ to $H^m_0(0,T_1)^2$ such that
		\begin{equation}\label{expansion}
			\forall z\in(-\rho,\rho), \qquad \left|\left\langle\psi(T_1;(u_z,v_z),\ph_1),\psi_K(T_1)\right\rangle-i^nz\right|\leqslant C |z|^{1+s_n}.
		\end{equation}
		The size of the controls is given by: for all $k\in\zz_{\geqslant -(q+1)}$, for all $ r\in [1,+\infty]$, there exists $C>0$ such that
		\begin{equation}\label{control:size}
			\forall z\in(-\rho,\rho), \qquad  \left\|u_z\right\|_{W^{k,r}},\left\|v_z\right\|_{W^{k,r}}\leqslant C|z|^{\alpha_n+s_n(\frac 1r-k)}.
		\end{equation}
		Finally, for all $\ep\in(0,\frac 34)$ , there exists $C>0$ so that for all $z\in(-\rho,\rho)$,
		\begin{equation}\label{point:size}\left\|\psi(T_1;(u_z,v_z),\ph_1)-\psi_1(T_1)\right\|_{H^{2(p+l)+3}_{(0)}}\leqslant C|z|^{\tau_l}, \qquad l=-(q+1),\cdots,m,
		\end{equation}
		with $\tau_l:=\alpha_n+s_n\left(\frac{3}{4}-\ep-l\right)>\frac{1}{4}.$ Note that $\tau_{-(q+1)}>\frac 12$ if $\ep<\frac 14$.
	\end{prop}
	\begin{proof}
		Let $T_1>0$ and  $\rho=T_1^{\frac{1}{s_n}}$. Thanks to $\mathbf{(H)_{quad,K,4}}$, we considers $\bar{u},\bar{v}\in\mathcal{C}^{\infty}_c(\rr,\rr)^2$ such that $\text{Supp}(\bar{u}),\text{Supp}(\bar{v})\subset (0,1)$ and  
		$\displaystyle\int_0^1\bar{u}(t)\bar{v}^{(2q+1-n)}(t)\mathrm{d}t=\frac{1}{\gamma_n}.$
		Then, we define for $z\in\rr^*$, 
		$$u_z,v_z:t\in[0,T_1]\mapsto \text{sgn}(z)|z|^{\alpha_n}\bar{u}^{(q+1)}\left(\frac{t}{|z|^{s_n}}\right),|z|^{\alpha_n}\bar{v}^{(q+1)}\left(\frac{t}{|z|^{s_n}}\right).$$
		For all $z\in(-\rho,\rho)$, one has \begin{equation}\label{support}\text{Supp}\left(\left(u_z\right)_{q+1}\right),\text{Supp}\left(\left(v_z\right)_{q+1}\right)\subset (0,|z|^{s_n})\subset (0,T_1).\end{equation} By definition, 
		\begin{equation}\label{derivee:nulle}\text{for all }z\in(-\rho,\rho), \ \text{for all }i\in\llbracket 1,q+1\rrbracket, \quad \left(u_z\right)_i(T_1)=\left(v_z\right)_i(T_1)=0.\end{equation}
		Then, we can compute the norm of the controls.
		For all $k\in\zz_{\geqslant -(q+1)}$, $ r\in [1,+\infty[$, using, the notation $u^{(k)}=u_{-k}$ if $k<0$ and the Poincaré's inequality, we get for all $z\in(-\rho,\rho)\setminus\left\lbrace 0\right\rbrace$,
		$$\left\|u_z\right\|_{W^{k,r}}^r\leqslant C\left\|u_z^{(k)}\right\|^r_{L^r}=C\int_0^{1}\left||z|^{\alpha_n-ks_n}\bar{u}^{(q+1+k)}\left(\frac{t}{|z|^{s_n}}\right)\right|^r\dt.$$
		Then, using a change of variables, one obtains
		$$\left\|u_z\right\|_{W^{k,r}}\leqslant C|z|^{\left(\alpha_n-ks_n\right)+\frac{s_n}{r}}\left\|\bar{u}^{(q+1+k)}\right\|_{L^r}.$$ The inequality with $r=+\infty$ is similarly proved. We prove the inequality for $v_z$ in the same way to get \eqref{control:size}. Then, the mapping can be extended to $0$. The continuity is given by the previous inequality. We use the expansion \eqref{equationclef}, the support condition \eqref{support} and \eqref{derivee:nulle}  to obtain, for every $z\in(-\rho,\rho)$,
		\begin{equation*}\begin{gathered}\left\langle\xi(T_1),\psi_K(T_1)\right\rangle=i^n\gamma_n\int_0^{|z|^{s_n}}\left(u_z\right)_{q+1}(t)\left(v_z\right)_{n-q}(t)e^{i\omega_Kt}\dt+\rho_{T_1}^{q+1}(u_z,v_z)+R_{T_1,1}^{q+1}(u_z)\\+R_{T_1,2}^{q+1}(v_z)-iA_{q+1}^1\int_0^{|z|^{s_n}}\left(u_z\right)_{q+1}(t)^2e^{i\omega_Kt}\dt-iA_{q+1}^2\int_0^{|z|^{s_n}}\left(v_z\right)_{q+1}(t)^2e^{i\omega_Kt}\dt.\end{gathered}\end{equation*}
		Note that, for all $k\in\llbracket 0,q+1\rrbracket$, one has  $\left(u_z\right)_k=\text{sgn}(z)|z|^{\alpha_n+ks_n}\bar{u}^{(q+1-k)}\left(\frac{\cdot}{|z|^{s_n}}\right)$ and $\left(v_z\right)_k=|z|^{\alpha_n+ks_n}\bar{v}^{(q+1-k)}\left(\frac{\cdot}{|z|^{s_n}}\right)$. Then, with the change of variables $t=|z|^{s_n}\sigma$, one has
		\begin{equation}\begin{gathered}\label{motionn}
				\left\langle\xi(T_1),\psi_K(T_1)\right\rangle=i^n\gamma_n\text{sgn}(z)|z|^{2\alpha_n+(n+2)s_n}\int_0^1\bar{u}(\sigma)\bar{v}^{(2q+1-n)}(\sigma)e^{i\omega_K|z|^{s_n}\sigma}\mathrm{d}\sigma+\rho_{T_1}^{q+1}(u_z,v_z)\\+R_{T_1,1}^{q+1}(u_z)+R_{T_1,2}^{q+1}(v_z)-i|z|^{2\alpha_n+(2q+3)s_n}\int_0^1\left(A_{q+1}^1\bar{u}(\sigma)^2+A_{q+1}^2\bar{v}(\sigma)^2\right)e^{i\omega_K|z|^{s_n}\sigma}\mathrm{d}\sigma.\end{gathered}
		\end{equation}
		Note that, if $n$ is even then $n=2q$.
		Else, $n$ is odd, then, $\mathbf{(H)_{quad,K,1}}$ and $\mathbf{(H)_{quad,K,2}}$ gives $A_{q+1}^1=A_{q+1}^2=0$. In all the cases, we obtain
		\begin{equation}\label{rtt}
			|z|^{2\alpha_n+(2q+3)s_n}\left|\int_0^1\left(A_{q+1}^1\bar{u}(\sigma)^2+A_{q+1}^2\bar{v}(\sigma)^2\right)e^{i\omega_K|z|^{s_n}\sigma}\mathrm{d}\sigma\right|=\mathcal{O}\left(|z|^{2\alpha_n+(n+3)s_n}\right).
		\end{equation}
		By definition of $R_{T_1,1}^{q+1}$, $\left|R_{T_1,1}^{q+1}(u_z)\right|\leqslant \displaystyle\sum_{j=1}^{+\infty}\left|c_j\nu_j^{q+1}\omega_j^{q+1}\right|\left|I^j_{T_1}\left(\left(u_z\right)_{q+1},\left(u_z\right)_{q+1}\right)\right|$.
		Then, using the explicit formulation of $I_{T_1}^j$ and \eqref{control:size} with $(k,r)=(-(q+1),1)$, 
		\begin{equation}\label{rtu}\left|R_{T_1,1}^{q+1}(u_z)\right|\leqslant C\left\|u_z\right\|_{W^{-(q+1),1}}^2\leqslant C|z|^{2\alpha_n+2(q+2)s_n}=\mathcal{O}\left(|z|^{2\alpha_n+(n+3)s_n}\right).\end{equation}
		Similarly, 
		\begin{equation}\label{rtv}\left|R_{T_1,2}^{q+1}(v_z)\right|=\mathcal{O}\left(|z|^{2\alpha_n(n+3)s_n}\right).\end{equation}
		Finally, one obtains with the same arguments
		\begin{equation}\label{rhotuv}
			\left|\rho_{T_1}^{q+1}(u_z,v_z)\right|\leqslant C\left\|u_z\right\|_{W^{-(q+1),1}}\left\|v_z\right\|_{W^{-(n-q),1}}=\mathcal{O}\left(|z|^{2\alpha_n+(n+3)s_n}\right).\end{equation}
		Using \eqref{rtt}, \eqref{rtu}, \eqref{rtv} and \eqref{rhotuv} in \eqref{motionn}, noticing that $2\alpha_n+(n+2)s_n=1$ and using the expansion $e^{i\omega_K|z|^{s_n}\sigma}=1+\mathcal{O}(|z|^{s_n})$, one has
		$$\left\langle\xi(T_1),\psi_K(T_1)\right\rangle=i^n\gamma_nz\int_0^1\bar{u}(\sigma)\bar{v}^{(2q+1-n)}(\sigma)\mathrm{d}\sigma+\mathcal{O}\left(|z|^{1+s_n}\right)=i^nz+\mathcal{O}\left(|z|^{1+s_n}\right),$$
		by definition of $\bar{u},\bar{v}$. Thus, using the error estimates \eqref{reste} and the hypothesis $\mathbf{(H)_{lin,K,1}}$, we obtain
		$$|\left\langle\psi(T_1;(u_z,v_z),\ph_1),\psi_K(T_1)\right\rangle-i^nz|=\mathcal{O}\left(|z|^{1+s_n}+\left\|(u_z,v_z)\right\|_{L^2}^3\right).$$
		Using \eqref{control:size} with $(k,r)=(0,2)$, we obtain \eqref{expansion}, as $1+s_n\leqslant 3\alpha_n+\frac{3}{2}s_n$.
		
		For the estimate \eqref{point:size}, we use \eqref{error:lin} with $k=q$ and \eqref{control:size} with $(k,r)\in\{(l,2),(m,2)\}$ to obtain the existence of $C>0$ such that, for all $l=-(q+1),\cdots,m$, 
		\begin{equation}\label{error}\left\|\psi(T_1;(u_z,v_z),\ph_1)-\psi_1(T_1)-\Psi(T_1)\right\|_{H^{2(p+l)+3}_{(0)}}\leqslant C|z|^{2\alpha_n+s_n(1-m-l)}.\end{equation} 
		Then, we estimate the linear term in weak norms. For all $j\geqslant 2$, for all $k\geqslant -(q+1)$, integrations by parts give (the same notation as previously is used for $k<0$) gives
		\begin{equation}\label{ipplim}\left|\int_0^{T_1}u_z(t)e^{i\omega_j(t-T_1)}\dt\right|= \left|\omega_j^{-k}\int_0^{T_1}u_z^{(k)}(t)e^{i\omega_j(t-T_1)}\dt\right|\leqslant C_k\omega_j^{-k}|z|^{\alpha_n+s_n(1-k)},\end{equation}
		thanks the inequality \eqref{control:size}, with $(k,r)=(k,1)$. This inequality is true with $j=1$ because the left-hand size is zero. Let $k\in\llbracket -(q+1),m\rrbracket$ and $r\in[k,k+1]$. There exists $\theta\in[0,1]$ such that $r=k+\theta$. Then, using \eqref{ipplim}
		\begin{equation*}\begin{split}\left|\int_0^{T_1}u_z(t)e^{i\omega_j(t-T_1)}\dt\right|&=\left|\int_0^{T_1}u_z(t)e^{i\omega_j(t-T_1)}\dt\right|^{1-\theta}\left|\int_0^{T_1}u_z(t)e^{i\omega_j(t-T_1)}\dt\right|^{\theta}
				\\&\leqslant\left(C_k\omega_j^{-k}|z|^{\alpha_n+s_n(1-k)}\right)^{1-\theta}\left(C_{k+1}\omega_j^{-(k+1)}|z|^{\alpha_n+s_n(1-(k+1))}\right)^{\theta}\\&\leqslant \underset{k\in\llbracket -(q+1),m\rrbracket}\max\left(\max(1,C_k)\max(1,C_{k+1})\right)\omega_j^{-r}|z|^{\alpha_n+s_n(1-r)}.
		\end{split}\end{equation*}
		Then, the inequality \eqref{ipplim} is true for $k$ real; there exists $C>0$, uniform in $k$, such that, for every $j\in\nn^*$ and $k\in [-(q+1),m+1]$,
		\begin{equation}\label{lininter}\left|\int_0^{T_1}u_z(t)e^{i\omega_j(t-T_1)}\dt\right|\leqslant C\omega_j^{-k}|z|^{\alpha_n+s_n(1-k)}.\end{equation}
		This inequality is also true with $v_z$. We want to apply this inequality with $k=l$. Nevertheless, the series diverges; we introduce a non-integer perturbation: let $\ep\in(0,3/4)$ and $l\in\left\lbrace-(q+1),\cdots,m\right\rbrace$. With $k =\varepsilon+l+\frac 14$, one obtains
		\begin{equation}\label{linear:proof}
			\left\|\Psi(T_1)\right\|_{H^{2(p+l)+3}_{(0)}}\leqslant C\left(\sum_{j=1}^{+\infty}\frac{1}{j^{1+4\varepsilon}}\right)^{\frac 12}|z|^{\tau_l},
		\end{equation}
		thanks to Remark \ref{asymptotique} and the estimate \eqref{lininter}. Using the equation \eqref{error}  and \eqref{linear:proof}, we obtain the desired inequality, because $\tau_l\leqslant 2\alpha_n+s_n(1-m-l)$.
	\end{proof}
	The following statement represents step $2$ of the proof strategy explained at the beginning of Section \ref{proofmain}.
	\begin{prop}\label{tangentiphk} Assume that $\left\langle\mu_1\ph_1,\ph_1\right\rangle\neq0$.
		The vector $i^n\ph_K$ is a small-time $H^m_0$-continuously approximately reachable vector associated with vector variations $i^n\psi_K(T)$. More precisely, for all $T > 0$, there exist $C,\rho> $0 and a continuous map $z\in\rr\mapsto(U_z,V_z)\in H_0^m(0,T)^2$ such that
		\begin{equation}
			\forall z\in(-\rho,\rho), \qquad \left\|\psi(T;(U_z,V_z),\ph_1)-\psi_1(T)-i^nz\psi_K(T)\right\|_{H^{2(p+m)+3}_{(0)}}\leqslant C|z|^{1+\frac{1}{5}s_n},
		\end{equation}
		with the following size estimate on the family of controls
		\begin{equation}\label{control:size2}
			\left\|U_z\right\|_{H^m},\left\|V_z\right\|_{H^m}\leqslant C|z|^{\frac 14}.
		\end{equation}
	\end{prop}
	\begin{proof}
		Let $0<T_1<T$, one defines controls that allow to move along $i^n\ph_K$ and correcting the linear part, thanks to hypothesis $\mathbf{(H)_{lin,K,2}}$. More precisely, for $z\in\rr$, 
		$$U_z,V_z:=u_z\mathbb{1}_{[0,T_1]}+\tilde{u}_z\mathbb{1}_{[T_1,T]},v_z\mathbb{1}_{[0,T_1]}+\tilde{v}_z\mathbb{1}_{[T_1,T]},$$
		where $u_z,v_z$ are the controls defined by the previous proposition and $\tilde{u}_z,\tilde{v}_z$ are the controls given by the control in projection Theorem \ref{control-projection} with $k=q$, {i.e.}\
		$$\tilde{u}_z,\tilde{v}_z=\Gamma_{T_1,T}(\psi(T_1;(u_z,v_z),\ph_1),\psi_1(T)).$$
		Then, for all $z\in(-\rho,\rho)$ (with $\rho=\min(\rho_1,1)$, where $\rho_1$ is given by the previous proposition with $T_1$),
		$$\left\|U_z\right\|_{H^m},\left\|V_z\right\|_{H^m}\leqslant \left\|(u_z,v_z)\right\|_{H^m}+\left\|(\tilde{u}_z,\tilde{u}_z)\right\|_{H^m}.$$ The first term is estimated by \eqref{control:size} with $(k,r)=(m,2)$. For the second one, we use the simultaneous estimates \eqref{estim-simult} and \eqref{point:size} to obtain, for all $l\in\llbracket -(q+1),m\rrbracket$, 
		\begin{equation}\label{control:size:tilde}
			\left\|(\tilde{u}_z,\tilde{v}_z)\right\|_{H^l}\leqslant C\left\|\psi(T_1;(u_z,v_z),\ph_1)-\psi_1(T_1)\right\|_{H^{2(p+l)+3}_{(0)}} \leqslant C|z|^{\tau_l}.
		\end{equation}
		Noticing that $\tau_m> \frac 14$ and $\alpha_n+s_n(\frac 12-m)>\frac{1}{4}$, one has \eqref{control:size2}. For the motion along $i^n\ph_K$, by construction $$\mathbb{P}_{\mathcal{H}}(\psi(T;(U_z,V_z),\ph_1))=\psi_1(T)=\mathbb{P}_{\mathcal{H}}(\psi_1(T)+i^nz\psi_K(T)),$$  where $\mathcal{H}$ is defined in \eqref{h}. Thus, we just estimate $|\left\langle\psi(T;(U_z,V_z),\ph_1),\psi_K(T)\right\rangle-i^nz|$. Using the triangular inequality,
		\begin{equation}\begin{gathered}
			|\left\langle\psi(T;(U_z,V_z),\ph_1),\psi_K(T)\right\rangle-i^nz|\leqslant |\left\langle\psi(T_1;(u_z,v_z),\ph_1),\psi_K(T_1)\right\rangle-i^nz|\\+|\left\langle\psi(T;(U_z,V_z),\ph_1),\psi_K(T)\right\rangle-\left\langle\psi(T_1;(u_z,v_z),\ph_1),\psi_K(T_1)\right\rangle|.\end{gathered}
		\end{equation}
		The first term is estimated by \eqref{expansion}. To study the second one, we use the equations \eqref{oddeven} and \eqref{reste}. Then, we get
		\begin{equation}
			\begin{gathered}
				|\left\langle\psi(T;(U_z,V_z),\ph_1),\psi_K(T)\right\rangle-\left\langle\psi(T_1;(u_z,v_z),\ph_1),\psi_K(T_1)\right\rangle|=\mathcal{O}\left(\left\|(\tilde{u}_z,\tilde{v}_z)\right\|_{H^{-(q+1)}}^2\right.\\\left.+\left\|(u_z,v_z)\right\|_{H^{-(q+1)}}\left\|(\tilde{u}_z,\tilde{v}_z)\right\|_{H^{-(q+1)}}+\left\|\tilde{u}_z\right\|_{H^{-(q+1)}}\left\|\tilde{v}_z\right\|_{H^{-(n-q)}}+\left\|(U_z,V_z)\right\|_{L^2}^3\right).\\
			\end{gathered}
		\end{equation}
		We choose $0<\varepsilon<\frac{1}{48}$. We estimate these terms.
		\begin{enumerate}
			\item[1.]The error term: using \eqref{control:size} with $(k,r)=(0,2)$ and \eqref{control:size:tilde} with $l=0$ and because $\frac{1}{5}s_n\leqslant\frac{1}{16}$,
			\begin{equation*}\begin{split}\left\|(U_z,V_z)\right\|_{L^2}^3&\leqslant C|z|^{3\alpha_n+\frac{3}{2}s_n}+C|z|^{3\tau_0}\leqslant C|z|^{1+\frac{1}{5}s_n}+C|z|^{\frac{17}{16}+\left(\frac{1}{16}-3\varepsilon s_n\right)+\frac{9}{4}s_n}\\
					&\leqslant  C|z|^{1+\frac{1}{5}s_n}+C|z|^{\frac{17}{16}+\left(\frac{1}{16}-3\varepsilon\right)}\leqslant C|z|^{1+\frac{1}{5}s_n}+C|z|^{\frac{17}{16}}\leqslant C|z|^{1+\frac{1}{5}s_n}.
				\end{split}
			\end{equation*}
			\item[2.] Using \eqref{control:size:tilde} with $l=-(q+1)$,
			$\left\|(\tilde{u}_z,\tilde{v}_z)\right\|_{H^{-(q+1)}}^2\leqslant C|z|^{2\tau_{-(q+1)}}.$
			We obtain the result because
			\begin{enumerate}
				\item[a.] if $n$ is odd, 
				$2\tau_{-(q+1)}=2\alpha_n+s_n\left(n+2+\frac 12-2\varepsilon\right)= 1+s_n\left(\frac{1}{2}-2\varepsilon\right)\geqslant 1+\frac{1}{5}s_n.$
				\item[b.] if $n$ is even, 
				$2\tau_{-(q+1)}=2\alpha_n+s_n\left(n+2+\frac 32-2\varepsilon\right)\geqslant 1+\frac{1}{5}s_n.$
			\end{enumerate}
			\item[3.] Once again, using \eqref{control:size} with $(k,r)=(-(q+1),2)$ and \eqref{control:size:tilde} with $l=-(q+1)$,
			$$\left\|(u_z,v_z)\right\|_{H^{-(q+1)}}\left\|(\tilde{u}_z,\tilde{v}_z)\right\|_{H^{-(q+1)}}\leqslant C|z|^{\alpha_n+s_n(\frac{1}{2}+q+1)+\tau_{-(q+1)}}.$$
			With the same arguments, we obtain the result.
			\item[4.] If $n$ is even, there is a last term: using \eqref{control:size:tilde} with $l=-q$ and $l=-(q+1)$, $$\left\|\tilde{u}_z\right\|_{H^{-(q+1)}}\left\|\tilde{v}_z\right\|_{H^{-q}}\leqslant C|z|^{\tau_{-(q+1)}+\tau_{-q}}.$$ Furthermore,
			$\tau_{-(q+1)}+\tau_{-q}=2\alpha_n+s_n\left(n+2+\frac 12 -2\varepsilon\right)=1+s_n\left(\frac 12-2\ep\right)\geqslant 1+\frac 15 s_n.$
		\end{enumerate}
		Finally, the continuity of $z\in\rr\mapsto (u_z,v_z)\in H^m_0(0,T_1)^2$ is given by the previous proposition and the continuity of $z\in\rr\mapsto (\tilde{u}_z,\tilde{v}_z)\in H^m_0(T_1,T)^2$ results from the regularity of $\Gamma_{T_1,T}$ and the regularity of the Schrödinger equation, with respect to the controls.
	\end{proof}
	Now, we have to show that it is possible to move in the direction $i^{n+1}\ph_K$: this is the step $3$ of the proof strategy explained at the beginning of Section \ref{proofmain}.
	\begin{prop}\label{tangentphk} Assume that $\left\langle\mu_1\ph_1,\ph_1\right\rangle\neq0$.
		The vector $i^{n+1}\ph_K$ is a small-time $H^m_0$-continuously approximately reachable vector associated with vector variations $i^{n+1}\psi_K(T)$. More precisely, there exists $T^*>0$, such that, for all $T\in(0,T^*)$, there exist $C,\rho> $0 and a continuous map $z\mapsto(U_z,V_z)$ from $\rr$ to $H_0^m(0,T)^2$ such that,
		$$\forall z\in(-\rho,\rho), \qquad \left\|\psi(T;(U_z,V_z),\ph_1)-\psi_1(T)-i^{n+1}z\psi_K(T)\right\|_{H^{2(p+m)+3}_{(0)}}\leqslant C|z|^{1+\frac{1}{5}s_n},$$
		with the following size estimate on the family of controls
		$$\left\|U_z\right\|_{H^m},\left\|V_z\right\|_{H^m}\leqslant C|z|^{\frac 14}.$$
	\end{prop}
	\begin{proof} We consider $(u_z,v_z)_{z\in\rr}$ the family of controls associated with $i^n\ph_K$ in Proposition \ref{tangentiphk}.  
		First, we show that there exists $C>0$ such that for all $(\alpha,\beta)\in\rr^2$, small enough,
		\begin{equation}\label{keypoint}\left\|\psi(3T;(u_{\alpha,\beta},v_{\alpha,\beta}),\ph_1)-\psi_1(3T)-i^n\left(\beta e^{2i\omega_KT}+\alpha\right)\psi_K(3T)\right\|_{H^{2(p+m)+3}_{(0)}}\leqslant C|(\alpha,\beta)|^{1+\frac{1}{5}s_n},
		\end{equation}
		with $u_{\alpha,\beta}=u_{\alpha}\#0_{[0,T]}\#u_{\beta}$ and $v_{\alpha,\beta}=v_{\alpha}\#0_{[0,T]}\#v_{\beta}$. 
		Indeed, by Proposition \ref{tangentiphk}, there exist $C,\rho>0$ such that, for all $\alpha\in(-\rho,\rho)$,
		\begin{equation}
			\label{motionphin}\left\|\psi(T;(u_{\alpha},v_{\alpha}),\ph_1)-\psi_1(T)-i^n\alpha\psi_K(T)\right\|_{H^{2(p+m)+3}_{(0)}}\leqslant C|\alpha|^{1+\frac{1}{5}s_n},
		\end{equation}
		\begin{equation}\label{motionphin2}\left\|(u_{\alpha},v_{\alpha})\right\|_{H^m}\leqslant C|\alpha|^{\frac 14}.\end{equation}
		In the interval $[T,2T]$, the evolution of the system is free ($(u,v)\equiv 0$). Consequently, the solution is given by $\psi(2T)=e^{-iA(2T-T)}\psi(T)$. Using that $e^{-iAT}$ is an isometry from $H^{2(p+m)+3}_{(0)}$ to $H^{2(p+m)+3}_{(0)}$, we obtain
		\begin{equation}\label{motionphkint}\left\|\psi(2T;(u_{\alpha}\#0_{[0,T]},v_{\alpha}\#0_{[0,T]}),\ph_1)-\psi_1(2T)-i^n\alpha\psi_K(2T)\right\|_{H^{2(p+m)+3}_{(0)}}\leqslant C|\alpha|^{1+\frac{1}{5}s_n}.
		\end{equation}
		Let $\lambda\in\cc$, $\psi_0\in H^{2(p+m)+3}_{(0)}(0,1)$ and $u,v\in H^m_0((0,T),\rr)$. Using the uniqueness of the solution to \eqref{schro} (with $f=0$), one has: $\psi(\cdot;(u,v),\lambda\psi_0) =\lambda\psi(\cdot;(u,v),\psi_0).$
		This property, applied with $\lambda:=e^{2i\lambda_1T}$ and the semi-group property of the bilinear Schrödinger equation lead to $$\psi(3T;(u_{\alpha,\beta},v_{\alpha,\beta}),\ph_1)e^{2i\lambda_1T}=\psi\left(T;(u_{\beta},v_{\beta}),\psi(2T;(u_{\alpha}\#0_{[0,T]},v_{\alpha}\#0_{[0,T]}),\ph_1)e^{2i\lambda_1T}\right).$$ 
		Then, the estimate \eqref{dep-continue2} with  $\psi_0=\psi(2T;(u_{\alpha}\#0_{[0,T]},v_{\alpha}\#0_{[0,T]}),\ph_1)e^{2i\lambda_1T}-\ph_1$ gives
		\begin{equation*}\begin{gathered}\left\|\psi(3T;(u_{\alpha,\beta},v_{\alpha,\beta}),\ph_1)e^{2i\lambda_1T}-\psi(T;(u_{\beta},v_{\beta}),\ph_1)-\psi(T;(0,0),\psi_0)\right\|_{H^{2(p+m)+3}_{(0)}} \\\leqslant C\left\|(u_{\beta},v_{\beta})\right\|_{H^m}\left\|\psi(2T;(u_{\alpha}\#0_{[0,T]},v_{\alpha}\#0_{[0,T]}),\ph_1)-\psi_1(2T)\right\|_{H^{2(p+m)+3}_{(0)}}.\end{gathered}
		\end{equation*}
		Using the estimation \eqref{motionphin2} with $\alpha=\beta$ and  \eqref{motionphkint}, one gets
		\begin{equation}\label{compo-flot}\left\|\psi(3T;(u_{\alpha,\beta},v_{\alpha,\beta}),\ph_1)e^{2i\lambda_1T}-\psi(T;(u_{\beta},v_{\beta}),\ph_1)-e^{-iAT}\left(\psi_0\right)\right\|_{H^{2(p+m)+3}_{(0)}}\leqslant C|\beta|^{\frac 14}|\alpha|.
		\end{equation}
		Using \eqref{motionphkint} in \eqref{compo-flot}, we obtain
		\begin{equation}\hspace{-0.123 cm}
			\begin{gathered}\label{compo-flot2}\left\|\psi(3T;(u_{\alpha,\beta},v_{\alpha,\beta}),\ph_1)-\psi(T;(u_{\beta},v_{\beta}),\ph_1)e^{-2i\lambda_1T}-e^{-iAT}\left(i^n\alpha\psi_K(2T)\right)\right\|_{H^{2(p+m)+3}_{(0)}}\\\leqslant C\left(|\beta|^{\frac 14}|\alpha|+|\alpha|^{1+\frac{1}{5}s_n}\right).
			\end{gathered}
		\end{equation}
		Finally, as $e^{-iAT}(i^n\alpha\psi_K(2T))=i^n\alpha\psi_K(3T)$, using \eqref{motionphin} with $\beta$ instead of $\alpha$, we obtain \eqref{keypoint}. Thus, for $T \in\left(0,\frac{\pi}{2\omega_K}\right)$ and $z\in(-\rho,\rho)$, taking $\beta=\frac{z}{\sin(2\omega_KT)}$ and $\alpha=-\beta\cos(2\omega_kT)$, we conclude.
	\end{proof}
	\subsection{Conclusion}
	Now, we can easily write the proof of the main theorem of this article; this is the last step of the proof strategy explained at the beginning of Section \ref{proofmain}.
	\begin{proof}[Proof of Theorem \ref{maintheorem}]
		We use Theorem \ref{STCLvecteurtangent}. More precisely, we define $X=H^{2(p+m)+3}_{(0)}(0,1)$ and $E_T=H^m_0(0,T)^2$. Moreover, for all $T>0$,
		$$\mathscr{F}_T:(\psi_0,(u,v))\mapsto\psi(T;(u,v),\psi_0).$$ We have to check all the assumptions of Theorem \ref{STCLvecteurtangent} to obtain the controllability result. 
		\begin{enumerate}
			\item[$(A_1)$] This is known that the end-point map is regular around the equilibrium $(\ph_1,
			(0,0))$ (see \cite[Proposition $4.2$]{bournissou2021local} for $\mathcal{C}^1$). 
			\item[$(A_2)$] This point is given by \cite[Proposition $4.2$]{bournissou2021local}.
			\item[$(A_3)$] Using the uniqueness of the solution to \eqref{schro}, one can prove the following semi-group property: for all $T_1,T_2>0$, $\psi_0\in H^{2(p+m)+3}_{(0)}(0,1)$, $u,v\in H^m_0((0,T_1),\rr)$ and $\tilde{u},\tilde{v}\in H^m_0((0,T_2),\rr)$, 
			$$\psi(T_1+T_2; (u\#\tilde{u},v\# \tilde{v}), \psi_0) = \psi(T_2;(\tilde{u},\tilde{v}), \psi(T_1; (u,v), \psi_0)).$$
			\item[$(A_4)$] By hypothesis $\mathbf{(H)_{lin,K,2}}$ and by Theorem \ref{control-projection}, $H=\mathcal{H}$ is a closed subspace, this is the reachable set of the linearized system around the ground state. This space doesn't depend on $T$ and its real codimension in $L^2(0,1)$ is equal to $2$.
			\item[$(A_5)$] Finally, as $\langle\mu_1\ph_1,\ph_1\rangle\neq 0$, Propositions \ref{tangentiphk} and \ref{tangentphk} show that $\ph_K$ and $i\ph_K$ are small-time $H_0^m$-continuously approximately reachable vectors. Then, as a supplementary space of $H$ is given by $\spn_{\rr}(\ph_K,i\ph_K)$, this condition is verified.
		\end{enumerate}
		By Theorem \ref{STCLvecteurtangent}, the multi-input bilinear Schrödinger equation \eqref{schr} is $H^m_0$-Small-Time Locally Controllable around the ground state in $H^{2(p+m)+3}_{(0)}(0,1)$.
	\end{proof}
	\appendix
	\section{Postponed proofs}	\subsection{Existence of $\mu_1,\mu_2$ verifying the hypotheses}\label{existe}
	This section is inspired by \cite{bournissou2022smalltime,bournissou2021quadratic}. We recall that the operator $A$ is defined in \eqref{laplacien}.
	\begin{lm}\label{crochetlie}
		Let $\mu\in\mathcal{C}_c^{\infty}(0,1)$. For all $n\in\nn$, for all $f\in\mathcal{C}^{\infty}([0,1])$, 
		\begin{equation}\underline{\ad}^n_A(\mu)f=\sum_{k=0}^n\alpha_k^n\mu^{(2n-k)}f^{(k)},\end{equation}
		where $(\alpha_k^n)_{k\in\llbracket 0,n\rrbracket}$ are defined by induction as, $\alpha_0^0:=1$ and, 
		$$\forall k\in\llbracket 1,n\rrbracket, \quad \alpha^{n+1}_k:=2\alpha_{k-1}^n+\alpha_k^n, \qquad \alpha_0^{n+1}:=\alpha_0^n,\qquad \alpha^{n+1}_{n+1}:=2\alpha_n^n.$$
		Moreover, $\displaystyle\sum_{k=0}^n\alpha_k^n(-1)^k=(-1)^n$ and $\displaystyle\sum_{k=0}^nk\alpha_k^n(-1)^k=2n(-1)^n$.
	\end{lm}
	This lemma is proved by induction in \cite[Proposition A.3, \textit{Step} $1$]{bournissou2021quadratic}.
	\begin{lm}\label{lmcroc} For all $n\in\nn$, there exist a constant $C>0$ and a quadratic form $Q_n$ such that, for every $\mu_1,\mu_2\in\mathcal{C}^{\infty}_c(0,1)$,
		\begin{enumerate}
			\item[1.] The odd good quadratic brackets are estimated as
			$$\gamma_{2n+1}=-\langle [\underline{\ad}_A^{n+1}(\mu_1),\underline{\ad}_A^n(\mu_2)]\ph_1,\ph_K\rangle= 2(-1)^n\langle\mu_1^{(4n+2)}\mu_2\ph_1,\ph_K\rangle+Q_n(\mu_1,\mu_2),$$
			with $\left|Q_n(\mu_1,\mu_2)\right|\leqslant\left\|\mu_1\right\|_{H^{4n+1}}\left\|\mu_2\right\|_{L^2}.$
			\item[2.] The even good quadratic brackets are estimated as $$\gamma_{2n}=\langle [\underline{\ad}_A^n(\mu_1),\underline{\ad}_A^n(\mu_2)]\ph_1,\ph_K\rangle= 4n(-1)^n\langle \mu_1^{(4n-2)}\mu'_2,\ph_1'\ph_K-\ph_K'\ph_1\rangle+Q_n(\mu_1,\mu_2),$$
			with $\left|Q_n(\mu_1,\mu_2)\right|\leqslant\left\|\mu_1\right\|_{H^{4n-3}}\left\|\mu_2\right\|_{H^1}.$
		\end{enumerate}
	\end{lm}
	\begin{proof} The equalities between the brackets and the series defined by $\gamma_i$ are proved in Proposition \ref{crochet3}.
		Let $\mu_1,\mu_2\in\mathcal{C}^{\infty}_c(0,1)$. We start with the odd case: for all $f\in\mathcal{C}^{\infty}([0,1])$, $n\geqslant 0$, using \eqref{crochetlie}, one has
		\begin{equation}\begin{gathered}\label{liedvlp}[\underline{\ad}_A^n(\mu_2),\underline{\ad}_A^{n+1}(\mu_1)]f=\sum_{k=0}^{n+1}\sum_{i=0}^n\sum_{l=0}^i\binom{i}{l}\alpha_k^{n+1}\alpha_i^n\mu_1^{(2(n+1)+i-k-l)}\mu_2^{(2n-i)}f^{(k+l)}\\-\sum_{k=0}^n\sum_{i=0}^{n+1}\sum_{l=0}^i\binom{i}{l}\alpha_k^n\alpha_i^{n+1}\mu_2^{(2n+i-k-l)}\mu_1^{(2(n+1)-i)}f^{(k+l)}.\end{gathered}
		\end{equation} Moreover, for all $a,b\in\nn$, for all $f\in\mathcal{C}^{\infty}([0,1])$, there exists $C>0$, such that
		$$\left|\langle \mu_1^{(a)}\mu_2^{(b)},f\rangle\right|\leqslant\left\|\mu_1\right\|_{H^{a+b}}\left\|\mu_2\right\|_{L^2}.$$
		Thus, thanks to \eqref{liedvlp}, one obtains
		\begin{multline}\label{lieerr}\left|\langle[\underline{\ad}_A^n(\mu_2),\underline{\ad}_A^{n+1}(\mu_1)]\ph_1,\ph_K\rangle-\alpha_0^{n+1}\sum_{i=0}^n\alpha_i^n\langle\mu_1^{(2(n+1)+i)}\mu_2^{(2n-i)}\ph_1,\ph_K\rangle\right.\\\left.+\alpha_0^n\sum_{i=0}^{n+1}\alpha_i^{n+1}\langle\mu_1^{(2(n+1)-i)}\mu_2^{(2n+i)}\ph_1,\ph_K\rangle\right|\leqslant \left\|\mu_1\right\|_{H^{4n+1}}\left\|\mu_2\right\|_{L^2}.
		\end{multline}
		Moreover, using integrations by parts, we get for all $i\in\llbracket 0,n\rrbracket$,
		\begin{equation}\label{lieerr2}\left|\langle\mu_1^{(2(n+1)\pm i)}\mu_2^{(2n\mp i)}\ph_1,\ph_K\rangle-(-1)^i\langle \mu_1^{(4n+2)}\mu_2\ph_1,\ph_K\rangle\right|\leqslant  \left\|\mu_1\right\|_{H^{4n+1}}\left\|\mu_2\right\|_{L^2}.
		\end{equation}
		Using \eqref{lieerr} and \eqref{lieerr2}, one has
		\begin{multline*}\left|\gamma_{2n+1}-\left(\alpha_0^{n+1}\sum_{i=0}^n\alpha_i^n(-1)^i-\alpha_0^n\sum_{i=0}^{n+1}\alpha_i^{n+1}(-1)^i\right)\langle\mu_1^{(4n+2)}\mu_2\ph_1,\ph_K\rangle\right|\\\leqslant \left\|\mu_1\right\|_{H^{4n+1}}\left\|\mu_2\right\|_{L^2}.\end{multline*}
		Finally, as $\displaystyle\sum_{i=0}^n\alpha_i^n(-1)^i=(-1)^n$ and $\alpha_0^n=1$, one obtains the result. Let us extract the leading term of the even good brackets. With same manipulations, one obtains for all $n\geqslant 1$,
		\begin{multline*}
			$$	\left|\gamma_{2n}-\alpha_0^n\displaystyle\sum_{i=0}^n\alpha_i^n\left\langle\left(\mu_1^{(2n-i)}\mu_2^{(2n+i)}-\mu_2^{(2n-i)}\mu_1^{(2n+i)}\right)\ph_1,\ph_K\right\rangle-\right.\\\left.\alpha_0^n\sum_{i=1}^ni\alpha_i^n\left\langle\left(\mu_1^{(2n-i)}\mu_2^{(2n+i-1)}-\mu_2^{(2n-i)}\mu_1^{(2n+i-1)}\right),\ph_1'\ph_K\right\rangle-\right.\\\left.\alpha_1^n\sum_{i=0}^n\alpha_i^n\left\langle\left(\mu_1^{(2n-i)}\mu_2^{(2n+i-1)}-\mu_2^{(2n-i)}\mu_1^{(2n+i-1)}\right),\ph_1'\ph_K\right\rangle\right|\leqslant\left\|\mu_1\right\|_{H^{4n-3}}\left\|\mu_2\right\|_{H^1}.$$
		\end{multline*}
	 Hence, using similar estimates, we obtain the simpler expression
		\begin{multline*}\left|\gamma_{2n}+2\alpha_0^n\left(\sum_{i=0}^ni\alpha_i^n(-1)^i\right)\langle \mu_1^{(4n-2)}\mu_2',(\ph_1\ph_K)'\rangle\right.\\\left.-\left(2\alpha_0^n\sum_{i=0}^ni\alpha_i^n(-1)^i+2\alpha_1^n\sum_{i=0}^n\alpha_i^n(-1)^i\right)\langle \mu_1^{(4n-2)}\mu_2',\ph_1'\ph_K\rangle\right|\leqslant\left\|\mu_1\right\|_{H^{4n-3}}\left\|\mu_2\right\|_{H^1}.\end{multline*}
		By induction, we obtain $\alpha_1^n=2n.$  The sums calculated in Lemma \ref{crochetlie} lead to the result.
	\end{proof}
	\begin{lm}\label{muint} Let $K\in\nn$, $K\geqslant 2$. There exist $\mu_1^n,\mu_{2}^n\in\mathcal{C}_c^{\infty}(0, 1)$ such that
		$$\langle\mu_1^n\ph_1,\ph_K\rangle=\langle\mu_2^n\ph_1,\ph_K\rangle=\gamma_1(\mu_1^n,\mu_2^n)=\cdots=\gamma_{n-1}(\mu_1^n,\mu_2^n)=0,$$
		and
		$$\gamma_n(\mu_1^n,\mu_2^n)=1.$$
	\end{lm}
	\begin{rmq}
		Choosing $\tilde{\mu_1^n}=-\mu_1^n$, one has $\gamma_n(\tilde{\mu_1^n},\mu_2^n)=-1$.
	\end{rmq}
	The proof of this lemma is the same as the proof of \cite[Theorem A.$4$]{bournissou2021quadratic}, using Lemma \ref{lmcroc} instead of \cite[Proposition A.$3$]{bournissou2021quadratic}. 
	
	The proof of the following theorem is mainly the same given by Bournissou in \cite[Theorem  A.$2$]{bournissou2022smalltime}.
	\begin{theorem} Let $n\geqslant 1$, $K\geqslant2 $, $p,m\geqslant 0$, such that $n\leqslant p+1$. There exist $\mu_1,\mu_2$ satisfying $\mathbf{(H)_{reg}}$,  $\mathbf{(H)_{lin,K,1}}$,  $\mathbf{(H)_{lin,K,2}}$,  $\mathbf{(H)_{quad,K,1}}$,  $\mathbf{(H)_{quad,K,2}}$,  $\mathbf{(H)_{quad,K,3}}$ and $\mathbf{(H)_{quad,K,4}}$.
	\end{theorem}
	In order to prove this theorem, we reformulate some hypotheses. More precisely, we will note
	\begin{equation}\label{hypref1}
		\text{Supp}(\mu_1),\text{Supp}(\mu_2)\subset [0,1), \end{equation}
	\begin{equation}\label{hypref2} \forall j\in\nn^*\setminus\left\lbrace K\right\rbrace,\quad \langle \mu_1\ph_1,\ph_j\rangle\langle\mu_2\ph_1,\ph_j\rangle\neq0,\end{equation}
	\begin{equation}\label{hypref3}\mu_1^{(2p+1)}(0)\mu_2^{(2p+1)}(0)\neq 0.\end{equation}
	Note that, if $\mu_1,\mu_2\in H^{2(p+m)+3}\cap H^{2p+1}_0((0,1),\rr)$ are two functions, one has
	$$\eqref{hypref1}, \eqref{hypref2},  \eqref{hypref3} \Rightarrow \mathbf{(H)_{reg}},\mathbf{(H)_{lin,K,2}}.$$
	Consequently, we prove the existence of $\mu_1,\mu_2\in H^{2(p+m)+3}\cap H^{2p+1}_0((0,1),\rr)$ verifying  $\mathbf{(H)_{lin,K,1}}$,  $\mathbf{(H)_{quad,K,1}}$,  $\mathbf{(H)_{quad,K,2}}$,  $\mathbf{(H)_{quad,K,3}}$, $\mathbf{(H)_{quad,K,4}}$, $\eqref{hypref1}$, $\eqref{hypref2}$ and $\eqref{hypref3}$.

	\begin{proof}[Idea of proof] Let $K\in\nn$, $K\geqslant 2$ and $\bar{x}\in(0,1)$ such that $\ph_K(\bar{x})=0$. As $\ph_1>0$ on $(0,1)$ and $\ph_K'(\bar{x})>0$, one may assume the existence of $\delta>0$ such that $\ph_1\ph_K>0$ on $(\bar{x},\bar{x}+\delta)$ and $\ph_1\ph_K<0$ on $(\bar{x}-\delta,\bar{x})$. Let $\eta\in(0,\bar{x}-\delta)$ be such that $\ph_1\ph_K\neq 0$ on $(0,\eta)$.
		\vspace{0.5 cm}\\\textit{Step 1: We prove the existence of functions $\mu_1,\mu_2\in H^{2(p+m)+3}\cap H^{2p+1}_0((0,1),\rr)$ satisfying $\mathbf{(H)_{lin,K,1}}$, $\mathbf{(H)_{quad,K,4}}$, $\eqref{hypref1}$, $\eqref{hypref2}$ and $\eqref{hypref3}$.} We use the same method than Bournissou in \cite{bournissou2022smalltime}; we consider
		$$\mathcal{E}:=\left\lbrace \mu_1,\mu_2\in H^{2(p+m)+3}(0,1),  \mu_1\equiv\mu_2\equiv 0 \text{ on }[\frac{\eta}{2},1], \text{ satisfying }\mathbf{(H)_{lin,K,1}}\right\rbrace\cap H^{2p+1}_0(0,1),$$
		$$\mathcal{U}:=\left\lbrace \mu_1,\mu_2\in\mathcal{E}, \quad\mu_1,\mu_2\text{ satisfy } \mathbf{(H)_{quad,K,4}}, \eqref{hypref2}\text{ and }\eqref{hypref3}\right\rbrace.$$		
		The set $\mathcal{E}$ is not empty. The purpose of the step is to prove that $\mathcal{U}$ is not empty. For that, we use the Baire theorem to prove that $\mathcal{U}$ is dense in $\mathcal{E}$.
		\vspace{0.5 cm}\\\textit{Step 2: We prove the existence of functions $\mu_1,\mu_2\in H^{2(p+m)+3}\cap H^{2p+1}_0((0,1),\rr)$ satisfying $\mathbf{(H)_{lin,K,1}}$, $\mathbf{(H)_{quad,K,3}}$, $\mathbf{(H)_{quad,K,4}}$, $\eqref{hypref1}$, $\eqref{hypref2}$ and $\eqref{hypref3}$.}
		We divide the support of $\mu_1$ and $\mu_2$, in two intervals: Supp$(\mu_1)=I\cup I_1$ and Supp$(\mu_2)=I\cup I_2$, with $I_1\cap I_2=\emptyset.$ We considers two functions $\mu_1,\mu_2$ as given in the previous step. If $\mathbf{(H)_{quad,K,3}}$ is not true, we construct, as in \cite{bournissou2022smalltime}, a perturbation, supported on $I$, satisfying $\mathbf{(H)_{quad,K,3}}$, without any consequence on the other properties.
		\vspace{0.5 cm}\\\textit{Step 3: Conclusion} we prove the theorem: we need to take into account the hypotheses $\mathbf{(H)_{quad,K,1}}$ and $\mathbf{(H)_{quad,K,2}}$. For example, if $\mathbf{(H)_{quad,K,1}}$ is not true, we construct, as in \cite{bournissou2022smalltime}, a perturbation, supported on $I_1$, satisfying $\mathbf{(H)_{quad,K,1}}$, without any consequence on the other properties (in particular, the properties $\mathbf{(H)_{quad,K,3}}$ and $\mathbf{(H)_{quad,K,4}}$ are not affected, because $\mu_2\equiv 0$ on $I_1$ (and $\mu_1\equiv 0$ on $I_2$)).
	\end{proof}
	\subsection{Lie brackets on the Schrödinger equation}
	\subsubsection{Definition of Lie brackets in infinite dimension}
	\begin{lm}\label{crochetdiminf}
		Let $k\in\nn$, $\mu_1\in W^{2k,\infty}(0,1)$ with all its derivatives of odd order less than or equal to $2k-3$ that vanish at the boundary. Then, for all $\ph\in H^{2k}_{(0)}(0,1)$, $\underline{\ad}_A^k(B_1)\ph$ is well-defined in $L^2(0,1)$. Moreover, 
		$$\underline{\ad}_A^k(B_1)\ph=\sum_{l=0}^k\alpha_l^k\mu_1^{(2k-l)}\ph^{(l)},$$ with the definitions of the scalar $(\alpha_l^k)_{0\leqslant l\leqslant k}$ given in \eqref{crochetlie}.
	\end{lm}
	\begin{proof} We prove this lemma by induction on $k$. The case $k=0$ is immediate, because $B_1$ is an operator on $L^2(0,1)$ when $\mu_1\in L^{\infty}(0,1)$. Assume that the result is true for a fixed $k$. We consider $\mu_1\in W^{2k+2,\infty}(0,1)$ such that $\restriction{{\mu_1}^{(2l+1)}}{\{0,1\}}=0$ for $0\leqslant l\leqslant k-1$. Let $\ph\in H^{2k+2}_{(0)}(0,1)$. Then, 
		\begin{enumerate}
			\item[1.] $A\ph\in H^{2k}_{(0)}(0,1)$,  so $\underline{\ad}_A^k(B_1)(A\ph)$ is well-defined, by the induction hypothesis.
			\item[2.] We now prove that $\underline{\ad}_A^k(B_1)\ph\in H^2_{(0)}(0,1)$. By induction hypothesis, one has the equality
			$\underline{\ad}_A^k(B_1)\ph=\displaystyle\sum_{l=0}^k\alpha_l^k\mu_1^{(2k-l)}\ph^{(l)}.$ Then, one has $\underline{\ad}_A^k(B_1)\ph\in H^2(0,1)$.
			Moreover, for all $l\in\llbracket 0,k\rrbracket$, 
			\begin{enumerate}
				\item[a.] if $l$ is odd, $l=2i+1$ and $\restriction{\mu_1^{(2k-2i-1)}\ph^{(l)}}{\{0,1\}}=0$ because $2k-2i-1\leqslant 2k-1$ so $\restriction{\mu_1^{(2k-2i-1)}}{\{0,1\}}=0$,
				\item[b.] otherwise, $l$ is even, $\restriction{\mu_1^{(2k-l)}\ph^{(l)}}{\{0,1\}}=0$ because $l\leqslant k$ and $\ph\in H^{2k+2}_{(0)}(0,1)$.
			\end{enumerate}
			Then, the bracket is well-defined in $L^2(0,1)$. By induction, we obtain the result.
		\end{enumerate}
	\end{proof}
	\begin{prop}\label{crochetdiminf1}
		Let $k\in\nn^*$, $\mu_1\in W^{4k-2,\infty}(0,1)$ with all its derivatives of odd order less than or equal to $4k-5$ that vanish at the boundary. Then, for all $\ph\in H^{3k-1}_{(0)}(0,1)$ the operator $[\underline{\ad}_A^{k-1}(B_1),\underline{\ad}_A^k(B_1)]\ph$ is well-defined in $L^2(0,1)$.
	\end{prop}
	\begin{proof}
		Let $\ph\in H^{3k-1}_{(0)}(0,1)$. Then $\ph\in H^{2k}_{(0)}(0,1)$ so $\underline{\ad}_A^k(B_1)\ph$ and $\underline{\ad}_A^{k-1}(B_1)\ph$ are well-defined by Proposition \ref{crochetdiminf} and one has to prove that
		\begin{enumerate}
			\item[1.] $\underline{\ad}_A^{k-1}(B_1)\ph\in H^{2k}_{(0)}(0,1)$. One recalls that $\underline{\ad}_A^{k-1}(B_1)\ph=\displaystyle\sum_{l=0}^{k-1}\alpha_l^{k-1}\mu_1^{(2k-l-2)}\ph^{(l)}$. According to the regularity of $\ph$ and $\mu_1$, this expansion gives $\underline{\ad}_A^{k-1}(B_1)\ph \in H^{2k}(0,1)$. Then, using Leibniz formula, for all $j\in\llbracket 0,k-1\rrbracket$, $l\in\llbracket 0,k-1\rrbracket$, $$\left(\mu_1^{(2k-l-2)}\ph^{(l)}\right)^{(2j)}=\displaystyle\sum_{i=0}^{2j}\binom{2j}{i}\mu_1^{(2k-l-2+i)}\ph^{(l+2j-i)}.$$ Each term of this sum vanish at $x=0,1$. Indeed,
			\begin{enumerate}
				\item[a.] if $i$ and $l$ have the same parity, $l+2j-i$ is even and $l+2j-i\leqslant k-1+2(k-1)=3k-3$ so $\restriction{\ph^{(l+2j-i)}}{\{0,1\}}=0$,
				\item[b.] otherwise, $2k-l-2+i$ is odd and $2k-l-2+i\leqslant 2k-2+2j-1\leqslant 4k-5$ so $\restriction{\mu_1^{(2k-l-2+i)}}{\{0,1\}}=0$.
			\end{enumerate}
			\item[2.] $\underline{\ad}_A^k(B_1)\ph\in H^{2k-2}_{(0)}(0,1)$. This is proven in the same way as point $1$.
		\end{enumerate}
		Consequently, the operator $\left[\underline{\ad}_A^{k-1}(B_1),\underline{\ad}_A^k(B_1)\right]$ is well-defined in $L^2(0,1)$.
	\end{proof}
	Then, if $2\lfloor\frac{n+1}{2}\rfloor-3\leqslant p-1$, Proposition \ref{crochetdiminf1} ensures that the brackets in \eqref{sommecrochet1} and \eqref{sommecrochet2} are well-defined in $L^2(0,1)$.
	\begin{prop}\label{crochetdiminf2}
		Let $k\in\nn$, $\mu_1,\mu_2\in W^{2k,\infty}(0,1)$  with all their derivatives of odd order less than or equal to $2k-3$ that vanish at the boundary Then, for all $\ph\in H^{k+\lfloor\frac{k+1}{2}\rfloor}_{(0)}(0,1)$, $\left[\underline{\ad}_A^{\lfloor\frac{k+1}{2}\rfloor}(B_1),\underline{\ad}_A^{\lfloor\frac k2\rfloor}(B_2)\right]\ph$  is well-defined in $L^2(0,1)$.
	\end{prop}
	This proposition can be proved in the same way as the previous one. Similarly, if $n-2\leqslant p-1$, the brackets in \eqref{sommecrochet3} are well-defined, thanks to Proposition \ref{crochetdiminf2}.
	\subsubsection{Computation of Lie brackets for the Schrödinger equation}
	\begin{lm}\label{lmcrochet} Let $k\in\nn$ and  $\mu\in W^{2k,\infty}(0,1)$ be such that $\restriction{\mu^{(2l+1)}}{\{0,1\}}=0$ for $0\leqslant l\leqslant k-2$. Then, for all $p,q\in\nn^*$, 
		$$\left\langle\underline{\ad}_{A}^k(\mu)\ph_p,\ph_q\right\rangle=(\lambda_p-\lambda_q)^k\left\langle \mu\ph_p,\ph_q\right\rangle.$$
	\end{lm}
	\begin{proof} Let $k\in\nn$. As $\ph_p,\ph_q\in H^{2k+2}_{(0)}(0,1)$, one has
		$$\left\langle\underline{\ad}_A^{k+1}(\mu)\ph_p,\ph_q\right\rangle=\left\langle[\underline{\ad}_A^k(\mu),A]\ph_p,\ph_q\right\rangle=\left\langle(\lambda_p I_d-A)\underline{\ad}_A^k(\mu)\ph_p,\ph_q\right\rangle.$$
		Using the symmetry property of $A$,
		$$\left\langle\underline{\ad}_A^{k+1}(\mu)\ph_p,\ph_q\right\rangle=\left\langle\underline{\ad}_A^k(\mu)\ph_p,(\lambda_p I_d-A)\ph_q\right\rangle=(\lambda_p-\lambda_q)\left\langle\underline{\ad}_A^k(\mu)\ph_p,\ph_q\right\rangle.$$
		Moreover, the result in true for $k=0$. An induction concludes the proof.
	\end{proof}
	\begin{prop}\label{crochet}
		Let $k\in\nn^*$, $K\geqslant 2$ and $\mu_1\in W^{4k-2,\infty}(0,1)$ be such that $\restriction{{\mu_1}^{(2l+1)}}{\{0,1\}}=0$ for $0\leqslant l\leqslant 2k-3$. Then,
		\begin{equation*}\left\langle[\underline{\ad}_A^{k-1}(\mu_1),\underline{\ad}_A^k(\mu_1)]\ph_1,\ph_K\right\rangle=2(-1)^kA_k^1.
		\end{equation*}
	\end{prop}
	\begin{proof} By definition,
		$$2(-1)^kA_k^1=\sum_{j=1}^{+\infty}c_j(\lambda_j-\lambda_1)^{k-1}(\lambda_K-\lambda_j)^k-\sum_{j=1}^{+\infty}c_j(\lambda_j-\lambda_1)^{k}(\lambda_K-\lambda_j)^{k-1}.$$
		Using Lemma \ref{lmcrochet}, we obtain
		$$\sum_{j=1}^{+\infty}\left(\left\langle\underline{\ad}_A^{k-1}(\mu_1)\ph_j,\ph_1\right\rangle\left\langle\underline{\ad}_A^k(\mu_1)\ph_K,\ph_j\right\rangle-\left\langle\underline{\ad}_A^k(\mu_1)\ph_j,\ph_1\right\rangle\left\langle\underline{\ad}_A^{k-1}(\mu_1)\ph_K,\ph_j\right\rangle\right).$$
		Using the symmetry of the operators, we get
		$$2(-1)^kA_k^1=(-1)^{k-1}\left\langle\underline{\ad}_A^{k-1}(\mu_1)\ph_1,\underline{\ad}_A^k(\mu_1)\ph_K\right\rangle+(-1)^{k-1}\left\langle\underline{\ad}_A^{k}(\mu_1)\ph_1,\underline{\ad}_A^{k-1}(\mu_1)\ph_K\right\rangle.$$
		Then, 
		$$2(-1)^kA_k^1=-\left\langle\underline{\ad}_A^k(\mu_1)\left(\underline{\ad}_A^{k-1}(\mu_1)\ph_1\right),\ph_K\right\rangle+\left\langle\underline{\ad}_A^{k-1}(\mu_1)\left(\underline{\ad}_A^{k}(\mu_1)\ph_1\right),\ph_K\right\rangle.$$
		Finally, 
		$$2(-1)^kA_k^1=\left\langle[\underline{\ad}_A^{k-1}(\mu_1),\underline{\ad}_A^{k}(\mu_1)]\ph_1,\ph_K\right\rangle.$$
	\end{proof}
	\begin{prop}\label{crochet3}
		Let $k\in\nn$, $K\geqslant 2$ and $\mu_1,\mu_2\in W^{2k,\infty}(0,1)$ such that $\restriction{{\mu_i}^{(2l+1)}}{\{0,1\}}=0$ for $0\leqslant l\leqslant k-2$ and $i\in\{1,2\}$. Then,
		\begin{equation*}\left\langle\left[\underline{\ad}_A^{\lfloor\frac{k+1}{2}\rfloor}(\mu_1),\underline{\ad}_A^{\lfloor\frac{k}{2}\rfloor}(\mu_2)\right]\ph_1,\ph_K\right\rangle=(-1)^k\gamma_k.
		\end{equation*}
	\end{prop}
	\begin{proof} 
		Using Lemma \ref{lmcrochet}, we notice that
		\begin{equation*}\begin{gathered}
			\gamma_{k}= \displaystyle\sum_{j=1}^{+\infty}\left(\left\langle\underline{\ad}_A^{\lfloor\frac{k+1}{2}\rfloor}(\mu_1)\ph_K,\ph_j\right\rangle\left\langle\underline{\ad}_A^{\lfloor\frac{k}{2}\rfloor}(\mu_2)\ph_j,\ph_1\right\rangle\right.\\\left.-\left\langle\underline{\ad}_A^{\lfloor\frac{k}{2}\rfloor}(\mu_2)\ph_K,\ph_j\right\rangle\left\langle\underline{\ad}_A^{\lfloor\frac{k+1}{2}\rfloor}(\mu_1)\ph_j,\ph_1\right\rangle\right)\end{gathered}
		\end{equation*}
		As $(\ph_j)_{j\geqslant 1}$ is an orthonormal basis of $L^2$ and thanks to the symmetry/skew-symmetry of the operator,
		\begin{equation*}\begin{gathered}
			\gamma_k=(-1)^{\lfloor\frac{k}{2}\rfloor} \left\langle\underline{\ad}_A^{\lfloor\frac{k+1}{2}\rfloor}(\mu_1)\ph_K,\underline{\ad}_A^{\lfloor\frac{k}{2}\rfloor}(\mu_2)\ph_1\right\rangle\\-(-1)^{\lfloor\frac{k+1}{2}\rfloor}\left\langle\underline{\ad}^{\lfloor\frac{k}{2}\rfloor}_A(\mu_2)\ph_K,\underline{\ad}^{\lfloor\frac{k+1}{2}\rfloor}_A(\mu_1)\ph_1\right\rangle.\end{gathered}
		\end{equation*}
		Once again, using the symmetry/skew-symmetry of the operator,
		$$\gamma_k=(-1)^{\lfloor\frac{k}{2}\rfloor+\lfloor\frac{k+1}{2}\rfloor} \left\langle\left[\underline{\ad}_A^{\lfloor\frac{k+1}{2}\rfloor}(\mu_1),\underline{\ad}_A^{\lfloor\frac{k}{2}\rfloor}(\mu_2)\right]\ph_1,\ph_K\right\rangle.$$
		This equality completes the proof.
	\end{proof}
	
	\subsection{Some elements of proof in finite dimension} 
	\subsubsection{A proof of Theorem \ref{Main_DF} via Sussmann's $S(\theta)$-condition}\label{sussman}
	In this section, we present a proof of Theorem \ref{Main_DF} relying on Sussmann's $S(\theta)$-sufficient condition 
	(see \cite[Theorem 7.3]{doi:10.1137/0325011}), recalled in Proposition \ref{Prop:Sussmann} below (see \cite[Theorem 3.29]{coronbook}), 
	for which we need the following definition.
	
	\begin{defi}
		The map $\sigma: \Br(X) \mapsto \mathcal{L}(X)$ is defined by $\sigma(b)=\ev(b)+\pi(\ev(b))$, where
		$\pi:\mathcal{L}(X) \mapsto \mathcal{L}(X)$ is the unique morphism of Lie algebra such that $\pi(X_0)=X_0$, $\pi(X_1)=X_2$ and $\pi(X_2)=X_1$.
	\end{defi}
	
	For instance, if $b=(X_1,(X_1,X_0))$ then $\sigma(b)=[X_1,[X_1,X_0]]+[X_2,[X_2,X_0]]$.
	
	\begin{prop}[Sussmann's $S(\theta)$-condition] \label{Prop:Sussmann}
		Let $f_0, f_1, f_2$ be analytic vector fields on a neighborhood of $0$ in $\rr^d$ 
		that satisfy the Lie algebra rank condition, Lie$(f_0,f_1,f_2)(0)=\rr^d$.
		If there exists $\theta\in[0,1]$ such that, for every $\mathfrak{b} \in \Br(X)$ with $n_0(\mathfrak{b})$ odd and both $n_1(\mathfrak{b})$ and $n_2(\mathfrak{b})$ even,
		we have
		\begin{equation} \label{Hyp_Sussm}
			f_{\sigma(\mathfrak{b})}(0) \in \spn\{ f_b(0); \ b \in \Br(X), n(b)+\theta n_0(b)<n(\mathfrak{b}) +\theta n_0(\mathfrak{b})\},
		\end{equation}
		then the system $x'=f_0(x)+uf_1(x)+v f_2(x)$ is $L^\infty$-STLC.
	\end{prop}

	\begin{proof}[Proof of Theorem \ref{Main_DF} thanks to Proposition \ref{Prop:Sussmann}]
		We prove by induction on $m \in \mathbb{N}$ that for every $d \in \nn^*$, the assumptions of Theorem \ref{Main_DF} imply $W^{m,\infty}$-STLC.
		
		Initialization $(m=0)$: we prove $L^\infty$-STLC by applying Proposition \ref{Prop:Sussmann}. 
		First \eqref{LARC} gives the Lie algebra rank condition $\spn(f_{b_1}(0),\cdots f_{b_d}(0))=\rr^d,$ with $b_1,\cdots,b_r\in\mathcal{B}_1$ such that $\spn(f_{b_1}(0),\cdots f_{b_r}(0))=S_1(f)(0)$ and $b_{r+1},\cdots,b_d\in\mathcal{B}_{2,good}$. One considers $$0<\theta\leqslant \frac{1}{\underset{i\in\llbracket 1,d\rrbracket}\max |b_i|}.$$ \\
		Let $L=\underset{i\in\llbracket r+1,d\rrbracket}\max |b_i|$. Let $\mathfrak{b} \in \Br(X)$ be such that $n_0(\mathfrak{b})$ is odd and both $n_1(\mathfrak{b})$ and $n_2(\mathfrak{b})$ are even. Then,
		\begin{enumerate}
			\item either $n(\mathfrak{b})\geqslant 4$ and then \eqref{LARC} gives the compensation because, for all $i\in\llbracket 1,d\rrbracket$,
			$$n(b_i)+\theta n_0(b_i)\leqslant 2+\theta \left(\underset{i\in\llbracket 1,d\rrbracket}\max |b_i|-1\right)<3<4+\theta n_0(\mathfrak{b})\leqslant n(\mathfrak{b})+\theta n_0(\mathfrak{b}),$$
			\item or $n(\mathfrak{b})=2$ i.e.\ $(n_1(\mathfrak{b}) ,n_2(\mathfrak{b}) )\in \{ (2,0) , (0,2) \}$ i.e.\
			$\sigma(\mathfrak{b}) \in \spn\{ \B_{2,bad} \}$. Thus
			\begin{enumerate} 
				\item[a.] if $|\mathfrak{b}|\leqslant L$, \eqref{compensationb2bad} gives, for all $i\in\llbracket 1,r\rrbracket$,
				$$n(b_i)+\theta n_0(b_i)= 1+\theta n_0(b_i)< 2\leqslant n(\mathfrak{b})+\theta n_0(\mathfrak{b}),$$
				\item[b.] if $|\mathfrak{b}|> L$,  the LARC \eqref{LARC} gives once again the compensation because, 
			\end{enumerate}
			$$\begin{array}{rl}
				n(b_i)+\theta n_0(b_i)\leqslant 1+\theta \left(|b_i|-1\right)<2+\theta n_0(\mathfrak{b})=n(\mathfrak{b})+\theta n_0(\mathfrak{b}), &\text{ for }i\in\llbracket 1,r\rrbracket,\\
				n(b_i)+\theta n_0(b_i)\leqslant 2+\theta \left(L-2\right)<2+\theta(L-1)\leqslant n(\mathfrak{b})+\theta n_0(\mathfrak{b}), &\text{ for }i\in\llbracket r+1,d\rrbracket.
			\end{array}$$
		\end{enumerate}
			Theses points prove \eqref{Hyp_Sussm}. The system $x'=f_0(x)+uf_1(x)+vf_2(x)$ is $L^{\infty}-$STLC.
			\par
		Heredity: we assume that the result is proved up to some $m\in\nn$.
		We consider the extended system
		$$\left\lbrace \begin{array}{l}
			\dot x = f_0(x) + u f_1(x) + v f_2(x) \\
			\dot{u}=w \\
			\dot{v}=\omega
		\end{array}\right.,$$
		with state $Y=(x,u,v) \in \mathbb{R}^{d+2}$ and control $(w,\omega)$, i.e.\ $Y'=F_0(Y)+wF_1(Y)+\omega F_2(Y)$ with
		$$
		F_0(x,u,v)=\begin{pmatrix}
			f_0(x) + u f_1(x) + v f_2(x) \\ 0 \\ 0
		\end{pmatrix}, \quad
		F_1(x,u,v)=\begin{pmatrix}
			0 \\ 1 \\ 0
		\end{pmatrix}, \quad
		F_2(x,u,v)=\begin{pmatrix}
			0 \\ 0 \\ 1
		\end{pmatrix}.
		$$
		Then, for every $i \in\llbracket1,2\rrbracket$,
		$$F_{M_1^i}(Y)=[F_i,F_0](Y)=\begin{pmatrix}
			f_i(x) \\ 0 \\ 0
		\end{pmatrix}.$$
		Thus, when computing $F_{M_j^i}$ for $j\geqslant 2$,
		the derivatives with respect to $u$ and $v$ never come into play and we obtain
		$$\forall i \in\llbracket1,2\rrbracket,\ \forall j \geqslant 1, \quad 
		F_{M_j^i}(Y)=\begin{pmatrix}
			\underline{\ad}_{f_0+uf_1+vf_2}^{j-1}(f_i)(x)
			\\ 0 \\ 0
		\end{pmatrix}.$$
		Thus, for every $i \in\llbracket1,2\rrbracket$ and $l \in \nn$, $F_{W_{1,l}^i} = 0$ 
		and for every $i \in\llbracket1,2\rrbracket$, $j \geqslant 2$, $l\in\nn$,
		$$F_{W_{j,l}^i} (Y) =
		\begin{pmatrix}
			\underline{\ad}_{f_0+uf_1+vf_2}^l [\underline{\ad}_{f_0+uf_1+vf_2}^{j-2} (f_i) , \underline{\ad}_{f_0+uf_1+vf_2}^{j-1}(f_i)] (x)
			\\ 0 \\ 0
		\end{pmatrix}.$$
		Moreover, for every $l\in\nn$, $F_{C_{0,l}} = F_{C_{1,l}}=0$ and, for every $j\geqslant 2$, $l\in\nn$,
		$$F_{C_{j,l}}(Y) =
		\begin{pmatrix}
			(-1)^j\underline{\ad}_{f_0+uf_1+vf_2}^l \left[\underline{\ad}_{f_0+uf_1+vf_2}^{\left\lfloor\frac{j+1}{2}\right\rfloor-1} (f_1) , \underline{\ad}_{f_0+uf_1+vf_2}^{\left\lfloor\frac{j}{2}\right\rfloor-1}(f_2)\right] (x)
			\\ 0 \\ 0
		\end{pmatrix}.$$
		In particular, for every $i\in\llbracket1,2\rrbracket$, $j \geqslant 1$,
		$F_{M_j^i}(0)=\begin{pmatrix}
			f_{M_{j-1}^i}(0)
			\\ 0 \\ 0
		\end{pmatrix}.$
		\\Moreover, for every  $j \geqslant 2$, $l\in\nn$,
		$$F_{W_{j,l}^i} (0)
		= \begin{pmatrix}
			f_{W_{j-1,l}^i}(0)
			\\ 0 \\ 0
		\end{pmatrix}, \qquad F_{C_{j,l}}(0) = 
		\begin{pmatrix}
			f_{C_{j-2,l}}(0)	\\ 0 \\ 0
		\end{pmatrix}.$$
		Thus, the extended system satisfies the assumptions of Theorem \ref{Main_DF} with dimension $d+2$.
		By the induction hypothesis, the extended system is $W^{m,\infty}$-STLC.
		By choosing trajectories of the extended system starting and finishing at $u,v=0$, 
		one concludes that the initial system $x'=f_0(x)+uf_1(x)+v f_2(x)$ is $W^{m+1,\infty}$-STLC.
	\end{proof}
	\subsubsection{Proof of the representation formula of Proposition \ref{Prop:rep_form}}
		\begin{defi}[Support]
		Let $a\in\mathcal{L}(X)$. For $b\in\mathcal{B}$, a Hall set, we denote by $\langle a, b\rangle$ the coefficient of $\ev(b)$ in the expansion of $a$ on the basis $\ev(\mathcal{B})$. We define
		$$\text{supp}(a) := \{b\in\mathcal{B}, \ \langle a,b\rangle\neq 0\}.$$  If $A \subset\Br(X)$, we let supp$(A) := \bigcup_{a\in A}\text{supp}(a)$. 
	\end{defi}
	\begin{defi}We consider the (unital associative) algebra $\widehat{\mathcal{A}}(X)$ of formal power series generated by $\mathcal{A}(X)$. An element $a\in\widehat{\mathcal{A}}(X)$ is a sequence $a=(a_n)_{n\in\nn}$ written $a =\sum_{n\in\nn}a_n$, where $a_n\in\mathcal{A}_n(X)$ with, in particular, $a_0\in\mathbb{R}$ being its constant term. We also define the Lie algebra of formal Lie series $\widehat{\mathcal{L}}(X)$ as the Lie algebra of formal power series $a\in\widehat{\mathcal{A}}(X)$ for which $a_n\in\mathcal{L}(X)$ for each $n\in\nn$.
	\end{defi}
	We consider the formal differential equation
	\begin{equation} \label{formalODE}
		\left\{
		\begin{aligned}
			\dot{y}(t) & = y(t)(X_0+u(t) X_1 + v(t) X_2), \\
			y(0) & = 1,
		\end{aligned}
		\right.
	\end{equation}
	whose solution is a formal series valued function $y:\rr_+ \rightarrow \widehat{\mathcal{A}}(X)$ (see \cite[Section $2.2$]{Beauchard_2023}).
	By \cite[Section $2.4$, Theorem 41]{Beauchard_2023}, for every $t>0$, there exists $\mathcal{Z}_\infty(t;X,(u,v)) \in \widehat{\mathcal{L}}(X)$ such that
	\begin{equation} \label{Magnus1.1}
		y(t) = \exp (t X_0) \exp \left( \mathcal{Z}_\infty(t;X,(u,v)) \right).
	\end{equation}
	If $\B \subset \Br(X)$ is a Hall set, there exists a unique family $(\eta_b)_{b \in \B}$ of maps $\rr_+ \times L^1_{loc}(\rr_+,\rr)^2 \rightarrow \rr$,
	called coordinates of the pseudo-first kind associated with $\B$ and they satisfy the announced homogeneity properties in \eqref{homogeneous}.
	Moreover, by \cite[Section 2.5]{Beauchard_2023}, for every $t>0$ and $u,v \in L^1((0,t),\rr)$,
	\begin{equation} \label{Prod_Inf}
		y(t) = \underset{b \in \mathcal{B}}{\overset{\leftarrow}{\prod}} e^{\xi_b(t,(u,v))b}. 
	\end{equation}
	Let $\B$ be a Hall set as in Proposition \ref{Prop:Bs}. We deduce from (\ref{Magnus1.1}), (\ref{Prod_Inf}) and the maximality of $X_0$ in $\B$ that
	$$\exp \left( \mathcal{Z}_\infty(t;X,(u,v)) \right) = \underset{b \in \mathcal{B} \setminus \{X_0\}}{\overset{\leftarrow}{\prod}} e^{\xi_b(t,(u,v))b} .$$
	By applying  the multivariate CBHD formula \cite[Proposition 34]{Beauchard_2023} to this expression, one obtains
	$\eta_b=\xi_b$ for every $b \in \B_1$ and for every $b \in \B_2$,
	$$\eta_{b} (t,(u,v)) = \xi_{b} (t,(u,v))
	+ \frac{1}{2} \sum \delta_{(b_1,b_2),b} \xi_{b_1}(t,(u,v)) \xi_{b_2}(t,(u,v)) $$
	where the sum is indexed by the set $\{ b_1>b_2 \in \B_1; \ b \in \text{supp}(b_1,b_2) \}$
	and $\delta_{(b_1,b_2),b}$ denotes the coefficient of $b$ in the expansion of $(b_1,b_2)$ on the basis $\B$. 
	In particular this sum is finite and involves only elements $b_1, b_2 \in \B_1$ such that $n_0(b_1)+n_0(b_1)=n_0(b)$. 
	
	The estimate (\ref{x=Z2+O}) is proved in \cite[Proposition 161]{Beauchard_2023}.
	The absolute convergence is proved in \cite[Proposition 103]{Beauchard_2023} and relies on \cite[Theorem 1.9]{Beauchard_2022}. 
	\subsection{Some lemmas about ODEs}
	\begin{lm}\label{gronwallmag}
		Let $\delta>0$ and $z\in\mathcal{C}^1(B(0,{\delta}),\rr^d)$ be such that $\left\|z\right\|_{\infty}\leqslant\delta$. We note $x(\cdot;z,0)$ the solution to the system:
		$\left\lbrace\begin{matrix}x'&=&z(x)\\x(0)&=&0\end{matrix}\right.$. Then,
		$$\left\|x(1;z,0)-z(0)\right\|\leqslant\left\|z(0)\right\|\left\|Dz\right\|_{\infty} e^{\left\|Dz\right\|_{\infty}}.$$
	\end{lm}
	\begin{proof}
		This lemma is proved in \cite[Lemma 160]{Beauchard_2023}.
	\end{proof}
	\begin{lm}\label{devlpodeprerequis} Let $\delta>0$ and $f_0,f_1,f_2:B(0,2\delta)\to\rr^d$ be real continuous functions, with $f_0(0)=0$. There exists $r(\delta)>0$, such that, for all $T_1<T<T_1+1$, $p\in B\left(0,r(\delta)\right)$ and $u,v\in L^{\infty}(T_1,T)$ with $\left\|u\right\|_{L^{\infty}},\left\|v\right\|_{L^{\infty}}<r(\delta)$, one has
		\begin{equation}\label{boulestable}\forall t\in [T_1,T], \quad x(T;(u,v),p,T_1)\in B(0,\delta).\end{equation}
	\end{lm}
	\begin{proof}
		This result is linked to continuity with respect to the initial condition.
	\end{proof}
	\begin{crl}\label{devlpodeprerequis2} Let $\delta>0$ and $f_0,f_1,f_2:B(0,2\delta)\to\rr^d$ be real continuous functions, with $f_0(0)=0$. There exists $r(\delta),C(\delta)>0$, such that, for all $T_1<T<T_1+1$, $p\in B\left(0,r(\delta)\right)$ and $u,v\in L^{\infty}(T_1,T)$ with $\left\|u\right\|_{L^{\infty}},\left\|v\right\|_{L^{\infty}}<r(\delta)$, one has
		$$\left\|x(T;(u,v),p,T_1)\right\|\leqslant C(\delta)\left(\left\|p\right\|+\left\|(u,v)\right\|_{L^{\infty}}\right).$$
	\end{crl}
	\begin{proof}
		This statement is based on the previous lemma and the Grönwall's inequality.
	\end{proof}
	\begin{lm}\label{devlpode} Let $f_0:B(0,2\delta)\to\rr^d$, be $\mathcal{C}^2$ function, for $\delta>0$, with $f_0(0)=0$. There exist $C,\ep>0$ such that, for all $T_1<T<T_1+1$ and $p\in B\left(0,\ep\right)$,
		$$\left\|x(T;(0,0),p,T_1)-e^{(T-T_1)Df_0(0)}p\right\|\leqslant C\left\|p\right\|^2.$$
	\end{lm}
	\begin{proof}
		By definition, there exists $h>0$, s.t. for all $x\in B(0,h)$, 
		\begin{equation}\label{diff2}\left\|f_0(x)-f_0(0)-Df_0(0)x\right\|\leqslant\left(\frac 12\left\|D^2f_0(0)\right\|+1\right)\left\|x\right\|^2.\end{equation}
		Let $p\in B(0,\ep)$ with $\ep=\min\left(r(\delta),r(h)\right)$. Using the Duhamel's principle,
		\begin{equation*}\begin{gathered}\left\|x(T;(0,0),p,T_1)-e^{(T-T_1)Df_0(0)}p\right\|\\\leqslant\int_{T_1}^T\left\|e^{(T-s)Df_0(0)}\right\|\left\|f_0(x(s,(0,0),p,T_1))-Df_0(0)x(s,(0,0),p,T_1)\right\|\ds.\end{gathered}\end{equation*}
		Thank to Lemma \ref{devlpodeprerequis}, one can use the inequality \eqref{diff2} and Corollary \ref{devlpodeprerequis2} gives the result.
	\end{proof}
	\begin{lm}\label{gronwalldecomp} Let $\delta>0$ and $f_0,f_1,f_2:B(0,2\delta)\to\rr^d$ be $\mathcal{C}^2$ functions, with $f_0(0)=0$. There exist $C,\ep>0$, such that, for all $T_1<T<T_1+1$, $p\in B\left(0,\ep\right)$ and $u,v\in L^{\infty}(T_1,T)$ with $\left\|u\right\|_{L^{\infty}},\left\|v\right\|_{L^{\infty}}<\ep$, one has
		$$\left\|x(T;(u,v),p,T_1)-x(T;(u,v),0,T_1)-x(T;(0,0),p,T_1)\right\|\leqslant C\left(\left\|p\right\|\left\|(u,v)\right\|_{L^{\infty}}+\left\|p\right\|^2\right).$$
	\end{lm}
	\begin{proof}
		Let $p\in B\left(0,\varepsilon\right)$, $u,v\in L^{\infty}(T_1,T)$ for $\ep$ small enough. For simplicity, we use the notations: $x:=x(\cdot;(u,v),p,T_1)$, $x_{u,v}:=x(\cdot;(u,v),0,T_1)$ and $x_p:=x(\cdot;(0,0),p,T_1)$. Let $z:=x-x_{u,v}-x_p$. By definition, $z$ is solution to the following control-affine system
		\begin{equation*}\begin{gathered}\label{diffz}z'=\left(f_0(z+x_{u,v}+x_p)-f_0(x_{u,v}+x_p)\right)+\left(f_0(x_{u,v}+x_p)-f_0(x_{u,v})-f_0(x_p)\right)\\+u\left(f_1(z+x_{u,v}+x_p)-f_1(x_{u,v})\right)+v\left(f_2(z+x_{u,v}+x_p)-f_2(x_{u,v})\right),\end{gathered}\end{equation*}
		with $z(T_1)=0$. We estimate each term of this decomposition:
		\begin{enumerate}
			\item[1.] the first one scales like $\left\|z\right\|$,
			\item[2.] the second one is shaped like $\left\|p\right\|\left(\left\|p\right\|+\left\|(u,v)\right\|_{\infty}\right)$,
			\item[3.] the other terms are like $\left\|(u,v)\right\|_{\infty}\left(\left\|z\right\|+\left\|p\right\|\right)$.
		\end{enumerate}
		The Grönwall's lemma gives the result.
	\end{proof}
	\section*{Acknowledgement}
	
	The author would like to take particular care in thanking Karine Beauchard and Frédéric Marbach for the many discussions that brought this article to life.
	
	The author acknowledges support from grants ANR-20-CE40-0009 (Project TRECOS) and ANR-11-LABX-0020 (Labex Lebesgue), as well as from the Fondation Simone et Cino Del Duca – Institut de France.
	\nocite{*}
	\bibliographystyle{plain}
	\bibliography{mabiblio}

\begin{thebibliography}{10}

\bibitem{doi:10.1137/0320042}
{J}.~{M}. {B}all, {J}.~{E}. {M}arsden, and {M}. {S}lemrod.
\newblock {C}ontrollability for {D}istributed {B}ilinear {S}ystems.
\newblock {\em {SIAM} {J}ournal on {C}ontrol and {O}ptimization},
  20(4):575--597, 1982.

\bibitem{BEAUCHARD2005851}
{K}arine {B}eauchard.
\newblock {L}ocal controllability of a 1-{D} {S}chrödinger equation.
\newblock {\em {J}ournal de {M}athématiques {P}ures et {A}ppliquées},
  84(7):851--956, 2005.

\bibitem{Beauchard2010}
{K}arine {B}eauchard.
\newblock Controllablity of a quantum particle in a 1{D} variable domain.
\newblock {\em ESAIM: {C}ontrol, {O}ptimisation and {C}alculus of
  {V}ariations}, 14(1):105--147, 3 2010.

\bibitem{BEAUCHARD2006328}
Karine Beauchard and Jean-Michel Coron.
\newblock Controllability of a quantum particle in a moving potential well.
\newblock {\em Journal of Functional Analysis}, 232(2):328--389, 2006.

\bibitem{beauchard2014minimaltimebilinearcontrol}
Karine Beauchard, Jean-Michel Coron, and Holger Teismann.
\newblock Minimal time for the bilinear control of {S}chr\"odinger equations,
  2014.

\bibitem{articletkc}
Karine Beauchard, Jean-Michel Coron, and Holger Teismann.
\newblock Minimal time for the approximate bilinear control of {S}chrödinger
  equations.
\newblock {\em Mathematical Methods in the Applied Sciences}, 41, 01 2018.

\bibitem{beauchard2010local}
{K}arine {B}eauchard and {C}amille {L}aurent.
\newblock {L}ocal controllability of 1d linear and nonlinear {S}chrödinger
  equations with bilinear control.
\newblock {\em {J}ournal de {M}athématiques {P}ures et {A}ppliquées},
  94(5):520--554, 2010.

\bibitem{Beauchard_2022}
{K}arine {B}eauchard, {J}érémy {L}e {B}orgne, and {F}rédéric {M}arbach.
\newblock {G}rowth of structure constants of free lie algebras relative to hall
  bases.
\newblock {\em {J}ournal of {A}lgebra}, 612:281–378, December 2022.

\bibitem{Beauchard_2023}
Karine Beauchard, Jérémy Le~Borgne, and Frédéric Marbach.
\newblock {O}n expansions for nonlinear systems error estimates and convergence
  issues.
\newblock {\em Comptes Rendus. Mathématique}, 361(G1):97–189, January 2023.

\bibitem{beauchard2017quadratic}
Karine Beauchard and Frédéric Marbach.
\newblock Quadratic obstructions to small-time local controllability for
  scalar-input systems.
\newblock {\em Journal of Differential Equations}, 264(5):3704--3774, 2018.

\bibitem{BEAUCHARD202022}
Karine Beauchard and Frédéric Marbach.
\newblock Unexpected quadratic behaviors for the small-time local null
  controllability of scalar-input parabolic equations.
\newblock {\em Journal de Mathématiques Pures et Appliquées}, 136:22--91,
  2020.

\bibitem{beauchard2024unified}
Karine Beauchard and Frédéric Marbach.
\newblock {A} unified approach of obstructions to small-time local
  controllability for scalar-input systems.
\newblock 2024.

\bibitem{beauchard2013local}
Karine Beauchard and Morgan Morancey.
\newblock {L}ocal controllability of 1{D} {S}chrödinger equations with
  bilinear control and minimal time.
\newblock {\em {M}athematical {C}ontrol and {R}elated {F}ields}, 4(2):125--160,
  2014.

\bibitem{Boscain_2012}
U.~Boscain, M.~Caponigro, T.~Chambrion, and M.~Sigalotti.
\newblock A weak spectral condition for the controllability of the bilinear
  {S}chrödinger equation with application to the control of a rotating planar
  molecule.
\newblock {\em Communications in Mathematical Physics}, 311(2):423–455, March
  2012.

\bibitem{boscain:hal-04496433}
Ugo Boscain, K{\'e}vin Le~Balc'h, and Mario Sigalotti.
\newblock {Schr{\"o}dinger eigenfunctions sharing the same modulus and
  applications to the control of quantum systems}.
\newblock 2024.
\newblock working paper or preprint.

\bibitem{bournissou2022smalltime}
{M}égane {B}ournissou.
\newblock {S}mall-time local controllability of the bilinear {S}chr\"odinger
  equation, despite a quadratic obstruction, thanks to a cubic term.
\newblock 2022.

\bibitem{bournissou2021local}
{M}égane {B}ournissou.
\newblock {L}ocal controllability of the bilinear 1{D} {S}chrödinger equation
  with simultaneous estimates.
\newblock {\em {M}athematical {C}ontrol and {R}elated {F}ields},
  13(3):1047--1080, 2023.

\bibitem{bournissou2021quadratic}
{M}égane {B}ournissou.
\newblock {Q}uadratic behaviors of the 1{D} linear {S}chrödinger equation with
  bilinear control.
\newblock {\em {J}ournal of {D}ifferential {E}quations}, 351:324--360, 2023.

\bibitem{BOUSSAID2020108412}
Nabile Boussaïd, Marco Caponigro, and Thomas Chambrion.
\newblock Regular propagators of bilinear quantum systems.
\newblock {\em Journal of Functional Analysis}, 278(6):108412, 2020.

\bibitem{RogerBrockett2013}
Roger Brockett.
\newblock Controllability with quadratic drift.
\newblock {\em Mathematical Control and Related Fields}, 3(4):433--446, 2013.

\bibitem{doi:10.1137/06065369X}
{E}duardo {C}erpa.
\newblock {E}xact controllability of a {N}onlinear {K}orteweg–de {V}ries
  {E}quation on a {C}ritical {S}patial {D}omain.
\newblock {\em {SIAM} {J}ournal on {C}ontrol and {O}ptimization},
  46(3):877--899, 2007.

\bibitem{Cerpa2009}
{E}duardo {C}erpa and {E}mmanuelle {C}répeau.
\newblock Boundary controllability for the nonlinear {K}orteweg-de {V}ries
  equation on any critical domain.
\newblock {\em {E}nnales de l'I.H.P. {A}nalyse non linéaire}, 26(2):457--475,
  2009.

\bibitem{Chambrion2008ControllabilityOT}
Thomas Chambrion, Paolo Mason, Mario Sigalotti, and Ugo~V. Boscain.
\newblock Controllability of the discrete-spectrum {S}chr{\"o}dinger equation
  driven by an external field.
\newblock {\em Annales De L Institut Henri Poincare-analyse Non Lineaire},
  26:329--349, 2008.

\bibitem{chambrion2022smalltime}
{T}homas {C}hambrion and {E}ugenio {P}ozzoli.
\newblock {S}mall-time bilinear control of {S}chrödinger equations with
  application to rotating linear molecules.
\newblock {\em {A}utomatica}, 153:111028, 2023.

\bibitem{doi:10.1137/18M1215207}
Thomas Chambrion and Laurent Thomann.
\newblock A topological obstruction to the controllability of nonlinear wave
  equations with bilinear control term.
\newblock {\em SIAM Journal on Control and Optimization}, 57(4):2315--2327,
  2019.

\bibitem{chambrion:hal-01901819}
Thomas Chambrion and Laurent Thomann.
\newblock {On the bilinear control of the Gross-Pitaevskii equation}.
\newblock {\em {Annales de l'Institut Henri Poincar{\'e} C, Analyse non
  lin{\'e}aire}}, 37(3):605--626, 2020.

\bibitem{CORON2006103}
Jean-Michel Coron.
\newblock On the small-time local controllability of a quantum particle in a
  moving one-dimensional infinite square potential well.
\newblock {\em Comptes Rendus Mathematique}, 342(2):103--108, 2006.

\bibitem{coronbook}
{J}ean-{M}ichel Coron.
\newblock {\em {C}ontrol and {N}onlinearity}.
\newblock {M}athematical {S}urveys and {M}onographs, vol. 136, 2007.

\bibitem{Coron2004}
Jean-Michel Coron and Emmanuelle Crépeau.
\newblock Exact boundary controllability of a nonlinear {K}d{V} equation with
  critical lengths.
\newblock {\em Journal of the European Mathematical Society}, 006(3):367--398,
  2004.

\bibitem{coron2020smalltime}
Jean-Michel Coron, Armand Koenig, and Hoai-Minh Nguyen.
\newblock {O}n the small-time local controllability of a {K}d{V} system for
  critical lengths.
\newblock {\em Journal of the European Mathematical Society}.

\bibitem{duca2021bilinearcontrolgrowthsobolev}
Alessandro Duca and Vahagn Nersesyan.
\newblock Bilinear control and growth of {S}obolev norms for the nonlinear
  {S}chr\"odinger equation.
\newblock 2021.

\bibitem{duca2022local}
{A}lessandro {D}uca and {V}ahagn {N}ersesyan.
\newblock {L}ocal exact controllability of the 1{D} nonlinear {S}chr\"odinger
  equation in the case of {D}irichlet boundary conditions.
\newblock 2022.

\bibitem{duca2024smalltimecontrollabilitynonlinearschrodinger}
Alessandro Duca and Eugenio Pozzoli.
\newblock Small-time controllability for the nonlinear {S}chr\"odinger equation
  on $\mathbb{R}^n$ via bilinear electromagnetic fields.
\newblock 2024.

\bibitem{ervedozapuel}
Sylvain Ervedoza and Jean-Pierre Puel.
\newblock Approximate controllability for a system of {S}chrödinger equations
  modeling a single trapped ion.
\newblock {\em Annales de l'Institut Henri Poincare (C) Non Linear Analysis},
  26:2111--2136, 11 2009.

\bibitem{hermeskawski}
{H}enry {H}ermes and {M}atthias {K}awski.
\newblock {L}ocal controllability of a single input, affine system.
\newblock {\em Nonlinear analysis and application (Arlington, Tex.)}, pages
  235--248, 1987.

\bibitem{kawski2}
Matthias Kawski.
\newblock High-order small-time local controllability.
\newblock {\em Nonlinear Controllability and Optimal Control}, 133, 03 2000.

\bibitem{Krob1987}
{D}aniel {K}rob.
\newblock Codes limites et factorisations finies du monoïde libre.
\newblock {\em RAIRO - Theoretical Informatics and Applications - Informatique
  Théorique et Applications}, 21(4):437--467, 1987.

\bibitem{Fr_d_ric_Marbach_2018}
Frédéric Marbach.
\newblock An obstruction to small time local null controllability for a viscous
  burgers’ equation.
\newblock {\em Annales scientifiques de l’École normale supérieure},
  51(5):1129–1177, 2018.

\bibitem{rouchon}
Mazyar Mirrahimi and Pierre Rouchon.
\newblock Controllability of quantum harmonic oscillators.
\newblock {\em Automatic Control, IEEE Transactions on}, 49:745 -- 747, 06
  2004.

\bibitem{AIHPC_2014__31_3_501_0}
{M}organ {M}orancey.
\newblock {S}imultaneous local exact controllability of {1D} bilinear
  {Schr\"odinger} equations.
\newblock {\em {A}nnales de l'{I.H.P.} {A}nalyse non lin\'eaire},
  31(3):501--529, 2014.

\bibitem{morancey2013globalexactcontrollability1d}
Morgan Morancey and Vahagn Nersesyan.
\newblock Global exact controllability of a {1D} {S}chr\"odinger equations with
  a polarizability term.
\newblock 2013.

\bibitem{nersesyan}
Vahagn Nersesyan.
\newblock Global approximate controllability for {S}chrödinger equation in
  higher {S}obolev norms and applications.
\newblock {\em Annales de l'Institut Henri Poincaré C, Analyse non linéaire},
  27(3):901--915, 2010.

\bibitem{article}
Vahagn Nersesyan and Hayk Nersisyan.
\newblock Global exact controllability in infinite time of {S}chrödinger
  equation.
\newblock {\em Journal de Mathématiques Pures et Appliqués}, 06 2010.

\bibitem{nguyen2023localcontrollabilitykortewegdevries}
Hoai-Minh Nguyen.
\newblock Local controllability of the {K}orteweg-de {V}ries equation with the
  right {D}irichlet control.
\newblock 2023.

\bibitem{Sigalotti2009GenericCP}
Mario Sigalotti, Paolo Mason, Ugo~V. Boscain, and Thomas Chambrion.
\newblock Generic controllability properties for the bilinear {S}chr{\"o}dinger
  equation.
\newblock {\em Proceedings of the 48h IEEE Conference on Decision and Control
  (CDC) held jointly with 2009 28th Chinese Control Conference}, pages
  3799--3804, 2009.

\bibitem{doi:10.1137/0325011}
H.~J. Sussmann.
\newblock A general theorem on local controllability.
\newblock {\em SIAM Journal on Control and Optimization}, 25(1):158--194, 1987.

\bibitem{turinici:hal-00536518}
{G}abriel {T}urinici.
\newblock {On the controllability of bilinear quantum systems}.
\newblock pages 75--92, 2000.

\end{thebibliography}
\end{document}